\documentclass[11pt,reqno]{amsart}

\usepackage[utf8]{inputenc}
\usepackage[margin=1in]{geometry}
\usepackage{amsmath,amssymb,amsthm,mathtools}
\usepackage{graphicx,xcolor}
\usepackage{enumitem}
\usepackage[hidelinks]{hyperref}
\numberwithin{equation}{section}

\newtheorem{theorem}{Theorem}[section]
\newtheorem{lemma}{Lemma}[section]
\newtheorem{proposition}{Proposition}[section]
\newtheorem{corollary}{Corollary}[section]

\theoremstyle{definition}
\newtheorem{definition}{Definition}[section]
\newtheorem{assumption}{Assumption}[section]
\newtheorem{example}{Example}[section]

\theoremstyle{remark}
\newtheorem{remark}{Remark}[section]

\usepackage{xcolor}

\usepackage[nameinlink,noabbrev]{cleveref}
\crefname{theorem}{Theorem}{Theorems}
\Crefname{theorem}{Theorem}{Theorems}
\crefname{lemma}{Lemma}{Lemmas}
\Crefname{lemma}{Lemma}{Lemmas}
\crefname{proposition}{Proposition}{Propositions}
\Crefname{proposition}{Proposition}{Propositions}
\crefname{corollary}{Corollary}{Corollaries}
\Crefname{corollary}{Corollary}{Corollaries}
\crefname{definition}{Definition}{Definitions}
\Crefname{definition}{Definition}{Definitions}
\crefname{assumption}{Assumption}{Assumptions}
\Crefname{assumption}{Assumption}{Assumptions}
\crefname{example}{Example}{Examples}
\Crefname{example}{Example}{Examples}
\crefname{remark}{Remark}{Remarks}
\Crefname{remark}{Remark}{Remarks}
\crefname{equation}{equation}{equations}
\Crefname{equation}{Equation}{Equations}

\newcommand{\R}{\mathbb{R}}
\newcommand{\dd}{\,\mathrm d}
\newcommand{\D}{\Delta_d}
\newcommand{\Dint}{\operatorname{int}\Delta_d}
\newcommand{\one}{\mathbf 1}
\newcommand{\osc}{\operatorname{osc}}
\newcommand{\T}{\mathsf{T}}
\newcommand{\KL}{\operatorname{KL}}

\title[Entropy-Regularized Mean Field Control]{
Entropy-Regularized Finite-State Mean Field Control:
Regularity and Stability}
\author{Asaf Cohen}
\author{Aditya Priya}
\date{}

\begin{document}

\begin{abstract}

We study infinite-horizon discounted mean field control of finite-state
continuous-time Markov chains with unbounded transition rates and a
relative-entropy rate penalty. Under mild assumptions, the value function
is the unique bounded viscosity solution of the stationary
Hamilton--Jacobi--Bellman equation on the closed probability simplex.
Joint convexity of a quadratically corrected flux Lagrangian and
smoothness of the data yield global semiconvexity, intrinsic
$\mathcal C^{1,1}$ regularity, and optimal Gibbs feedback.
For positive symmetric reference rates, an entropy-plus-value free
energy gives global exponential convergence under a quantitative
curvature condition. We also establish inverse-temperature limits for
values and, under parameter-uniform hypotheses, for feedbacks and
closed-loop trajectories. With population-independent reference rates
and a potential state-cost structure, optimal planner trajectories
induce discounted finite-state mean field game equilibria.
As an application, size-dependent market-making quotes with exponential
demand admit an exact revenue-to-entropy reformulation; a convex
inventory-crowding potential connects optimal quoting to stabilization
of the inventory distribution.

\end{abstract} \
\maketitle


\section{Introduction}

Many operational systems can be described by the fractions of a large
population in finitely many states. Examples include customers using
different service channels, vehicles located in different service
zones, swarms of robots, and machines in operating, idle, or maintenance modes. A planner
can influence how quickly units move between states through routing,
dispatching, incentives, or maintenance decisions.

Single-asset market making provides another example. A dealer's
inventory position is a state, customer executions move the dealer
between states, and size-dependent quotes control the transition rates.
The population distribution records the fractions of dealers at each
inventory level, while inventory-crowding costs couple their decisions.
In \Cref{sec:market-making}, exponential demand with common sensitivity
to total transaction concession yields an exact revenue-to-entropy
identity and a potential mean field game specialization of the control
framework.

Mean-field games (MFGs) and mean-field control problems (MFCPs) provide related
frameworks for decision making in large populations; see the
foundational works
\cite{LasryLions2007,HuangCainesMalhame2006} and the monographs
\cite{BensoussanFrehseYam2013,CarmonaDelarue2018I,
CarmonaDelarue2018II}. Continuous-time finite-state models have been
studied in
\cite{GomesMohrSouza2013,BayraktarCohen2018,CarmonaWang2021,
CecchinFischer2020,CecchinPelino2019,Cecchin2021}.
The closest control analysis is Cecchin \cite{Cecchin2021}, who studies
a finite-horizon problem with bounded transition-rate controls and
analyzes its simplex Hamilton--Jacobi--Bellman (HJB) equation and finite-agent approximation.
Carmona and Wang \cite{CarmonaWang2021} study decentralized
finite-state MFG equilibria and their finite-player interpretation.
Our primary object is instead a stationary centralized control problem
with unbounded transition intensities.

We formulate this decision problem as an infinite-horizon discounted
mean field control problem on the probability simplex. At population
state $\eta$ and rate matrix $a$, the running cost is
\[
f(\eta,a)
=
C(\eta)
+
\frac1\beta
\sum_{i=1}^d\eta_i
\sum_{j\ne i}
\left[
a_{ij}\log\frac{a_{ij}}{q_{ij}(\eta)}
-a_{ij}+q_{ij}(\eta)
\right],
\]
with the convention $0\log 0=0$. Here $C(\eta)$ is the system-level cost, $q(\eta)$ describes the
nominal transition rates, and $\beta>0$ is the inverse-temperature
parameter: $1/\beta$ is the weight assigned to deviations from those
rates. Thus smaller $\beta$ makes such deviations more costly. Before
multiplication by $1/\beta$, the population-weighted Kullback--Leibler (KL) sum is the
instantaneous relative-entropy rate of the controlled jump law with
respect to the nominal jump law
\cite{KesidisWalrand1993}. This KL penalty has several useful features. It is zero when the planner
uses the nominal rate $a_{ij}=q_{ij}(\eta)$. For a small perturbation of
a positive nominal rate,
\[
(q+\delta)\log\frac{q+\delta}{q}-(q+\delta)+q
=
\frac{\delta^2}{2q}+o(\delta^2),
\qquad \delta\to0,
\]
so locally it behaves like a quadratic adjustment cost, scaled by the
nominal rate. At the same time, it grows superlinearly as
$a_{ij}\to\infty$, which controls very large transition rates without
requiring an artificial upper bound on the control set. It is also
compatible with population fluxes: writing
$w_{ij}=\eta_i a_{ij}$ and
$r_{ij}=\eta_iq_{ij}(\eta)$ gives
\[
\eta_i
\left[
a_{ij}\log\frac{a_{ij}}{q_{ij}(\eta)}
-a_{ij}+q_{ij}(\eta)
\right]
=
w_{ij}\log\frac{w_{ij}}{r_{ij}}-w_{ij}+r_{ij},
\]
with the natural convention when $\eta_i=0$. This identity is useful
both for handling the boundary of the simplex and for exploiting
convexity in the population fluxes.

{\color{black}
The KL penalty also determines the optimal dynamics. At interior
states where $V$ is differentiable, edgewise Hamiltonian minimization
converts the logarithmic derivative of the penalty into the Gibbs rule
\[
a^*_{ij}(\eta)=q_{ij}(\eta)
\exp\!\left\{-\beta\bigl[(D_{\T}V(\eta))_j-(D_{\T}V(\eta))_i\bigr]\right\}.
\]
For symmetric reference rates, the resulting nonlinear jump dynamics
dissipates the free energy $\sum_i\eta_i\log\eta_i+2\beta V(\eta)$,
providing a finite-state analogue of the free-energy structure of
overdamped Langevin (McKean--Vlasov) dynamics
\cite{ErbarFathiLaschosSchlichting2016}. This links the KL control
formulation directly to an explicit optimal policy and to the
free-energy mechanism underlying our stability theorem.
}

The main analytical challenges are that the planner's state is the
continuous population vector, the control set is unbounded, and the
factor $\eta_i$ makes both the drift and the reference flux degenerate
when state $i$ has zero population mass. Although the formal Gibbs rule
can be computed separately for each transition, this calculation alone
does not establish dynamic programming at the boundary, regularity and
verification of the feedback, or long-run stability of the resulting
nonlinear dynamics. We establish all of these properties within one
stationary framework.

{\color{black}
From the planner's perspective, the analysis addresses five linked
questions: whether the discounted problem remains well posed with
unbounded rate decisions; whether it admits an implementable stationary
optimal policy; how the value and policy depend on the strength of the
entropy regularization; whether the resulting policy stabilizes the
population distribution; and when the centralized policy can be
decentralized as a mean field game equilibrium.
}
Our contributions answer these questions
as follows.

\begin{itemize}
    \item Under the mild viscosity-level assumptions, without convexity, we prove that the value function is the unique bounded viscosity solution on the closed simplex of the HJB equation. The viscosity inequalities follow from the dynamic programming principle (DPP), while uniqueness follows by reducing the equation to the finite-state mean-field Hamiltonian treated by Kraaij's comparison theorem \cite{Kraaij2016}.

    \item Under joint convexity of a quadratically corrected flux Lagrangian and
$\mathcal C^{1,1}$ regularity of the data, we prove that the value function is globally semiconvex and intrinsically $\mathcal{C}^{1,1}$ on the closed simplex. This yields a Gibbs-form feedback that is the unique Hamiltonian minimizer at interior states, global well-posedness of the closed-loop dynamics from interior initial conditions, and verification and uniqueness of the optimal control.

    \item For positive symmetric reference rates, which may depend on the
population distribution, we identify an entropy-plus-value free energy
and establish a global entropy-production inequality.
For example, if $C$ is convex and of class $\mathcal C^{1,1}$ and the
reference rates are constant and symmetric, then the constant in
\Cref{ass:regularity} may be taken to be $K=0$. Hence the condition
$2\beta K<1$ is automatic for every $\beta>0$.
More generally, if $K$ is the semiconvexity
constant, the condition $2\beta K<1$ makes the free energy strongly
convex and implies global exponential convergence of the closed-loop
dynamics to a unique equilibrium.

{\color{black}
\item Treating the inverse-temperature parameter as a model parameter,
we prove that $V_\beta$ is nonincreasing in $\beta$, converges uniformly
to the reference-dynamics value as $\beta\downarrow0$, and converges
uniformly to $C_{\min}/\rho$ as $\beta\to\infty$. Under the
\textcolor{black}{parameter-uniform boundedness of $\beta K_\beta$}, the Gibbs feedback
converges uniformly on the closed simplex to the reference rates, with
corresponding finite-horizon and discounted infinite-horizon convergence
of the closed-loop dynamics.
\textcolor{black}{With symmetry and a uniform positive margin in $2\beta K_\beta<1$, we also obtain small-$\beta$ exponential convergence}
with a decay rate and state-error prefactor independent of $\beta$
and the initial state, and identify the small- and large-$\beta$
limits of the stable equilibria under the respective hypotheses.
}

    \item Under additional state-independence and potential assumptions, the optimal planner trajectory induces a discounted finite-state mean field game equilibrium, with the same Gibbs feedback for the planner and the representative agent.

\item 
We apply the framework to single-asset market making with
size-dependent quotes and exponential demand with common
sensitivity to total transaction concession.
An exact revenue-to-entropy identity identifies the myopic
reference rates and the corresponding state-cost correction.
A convex inventory-crowding potential satisfies the regularity
hypotheses and yields optimal quotes and a potential mean field
game equilibrium. With symmetric reference rates, the induced
inventory distribution converges exponentially to its equilibrium.

\end{itemize}

{\color{black}
These inverse-temperature refinements follow the same assumption
hierarchy. \Cref{thm:beta-limits} gives monotonicity and both value
limits under the viscosity-level assumptions, with Lipschitz reference
rates for the small-$\beta$ limit. Parameter-uniform regularity yields
uniform Gibbs-feedback and trajectory convergence in
\Cref{thm:uniform-small-beta-feedback,cor:small-beta-trajectories}.
\textcolor{black}{With symmetry and uniform strong convexity,} \Cref{thm:uniform-small-beta-stability} gives a common
exponential rate and state-error prefactor for small $\beta$, while
\Cref{cor:beta-stability-limits} identifies the equilibrium limits
under its respective small- and large-$\beta$ hypotheses.
}

Numerical experiments in \Cref{sec:numerical-experiments} illustrate both the control-derived chain from the HJB to stability and the standalone Gibbs-stability mechanism. \textcolor{black}{Additional temperature sweeps test the value, feedback, trajectory, and equilibrium limits, including a state-dependent reference model covered by the relaxed assumptions.}

\subsection{Positioning our model and contribution within the literature}

We relate the analysis to entropy regularization, which produces the Gibbs feedback, and long-time stability, and place the market-making application in the literature on execution intensities and population interaction.

\subsubsection{Entropic regularization}

Entropy enters mean-field models through several distinct mechanisms.
One relevant comparison is Guo, Xu, and Zariphopoulou
\cite{GuoXuZariphopoulou2022}, who study a finite-horizon MFG in which
entropy regularizes a randomized distribution over the representative
agent's actions, with emphasis on exploration and learning. Delarue and
Lavigne study finite-horizon robust MFG and mean field control problems
in continuous state space
\cite{DelarueLavigne2026RobustMFG,
DelarueLavigne2026RobustMFC}. A similar line of work by Cohen and coauthors treats stationary
singular-control and mean field game equilibria, robust control under
model ambiguity in dividend, harvesting, and queueing models, and
moderate-deviation queueing control and differential games
\cite{CohenSun2026,ChakrabortyCohenYoung2023,
CohenHeningSun2022,Cohen2019BrownianUncertainty,
Cohen2019QueueingUncertainty,AtarCohen2017,AtarCohen2016}.
These papers use entropy regularization
to penalize distortions of the reference model
chosen by an adverse player. Our formulation regularizes a different object: the planner's
continuous-time jump intensities are penalized by their relative-entropy
rate with respect to a reference generator. These works are therefore
complementary rather than exact precedents. Here the questions concern
stationary HJB well-posedness, Gibbs feedback, and global closed-loop
stability. Related relaxed-control approaches include
\cite{WangZariphopoulouZhou2020,
FrikhaGermainLaurierePhamSong2025,FrikhaPhamSong2024,
BayraktarHernandezYanZhu2026PolicyGradient,
BayraktarHernandezYanZhu2026PhiBE}.

This mechanism is related to linearly solvable and KL-controlled Markov models
\cite{Todorov2006,Todorov2009} and to the classical relative entropy of continuous-time Markov transition rates \cite{KesidisWalrand1993}. In our setting, however, the controlled state is the population distribution itself, the reference generator may depend nonlinearly on that distribution, and the HJB equation is posed on the simplex. The exponential transform therefore does not linearize the problem; state constraints, boundary degeneracy, regularity, and stability remain essential issues.

\subsubsection{Long-time stability}

A common feature of several earlier approaches to nonlinear McKean--Vlasov dynamics is the use of an underlying Markovian $N$-particle system as an analytical intermediary. Malrieu \cite{Malrieu2001}, in a continuous-state diffusion setting, derives estimates for a system of weakly interacting diffusions that are uniform in the particle number and then transfers them to the nonlinear McKean--Vlasov equation. In the finite-state setting, Budhiraja, Dupuis, Fischer, and Ramanan \cite{BudhirajaDupuisFischerRamanan2015} similarly exploit the Markovian $N$-particle model: invariant measures and large-deviation properties of the empirical measure motivate Lyapunov functions and an associated Hamilton--Jacobi equation for the limiting nonlinear Kolmogorov dynamics.

Cohen and Huffman \cite{CohenHuffman2026} study uniform-in-time
approximation of finite-state mean-field interacting jump processes by
their limiting nonlinear Markov chain. Under exponential stability of
the limiting chain, they establish a uniform-in-time weak error of order
$1/N$ and give a sufficient stability criterion formulated through the
linearized Kolmogorov equation. Their principal objective is therefore
to control the finite-particle approximation uniformly in time. Here,
by contrast, we prove global exponential stability of the limiting
Gibbs dynamics itself from an explicit free-energy dissipation
inequality; no finite-particle approximation is used.

A complementary finite-state program works directly with the limiting equation by introducing a nonlinear discrete transport geometry. Erbar, Fathi, Laschos, and Schlichting \cite{ErbarFathiLaschosSchlichting2016} identify a free-energy gradient-flow structure for reversible local Gibbs dynamics. Building on this framework, Erbar, Fathi, and Schlichting \cite{ErbarFathiSchlichting2020} introduce an entropic Ricci-curvature notion based on geodesic convexity of the free energy in the associated nonlinear transportation metric. A positive curvature lower bound yields functional inequalities, including a modified logarithmic Sobolev inequality, and consequently exponential free-energy decay. Verifying such curvature bounds requires controlling the Hessian of the free energy in the nonlinear transport geometry, including additional terms generated by the state dependence of the rates.

Our stability argument is complementary to these approaches. The
curvature criterion in \cite{ErbarFathiSchlichting2020} is formulated
through geodesic convexity in a state-dependent nonlinear
transportation metric, and its verification requires estimates in that
geometry. We instead work directly in the Euclidean geometry of the
simplex. For a prescribed potential, our sufficient hypotheses are
positive symmetric reference rates satisfying the standing bounds,
global Euclidean $K$-semiconvexity and local
$\mathcal C^{1,1}$ regularity of the potential, and the quantitative
condition $2\beta K<1$. These are different sufficient criteria, not
equivalent formulations of the entropic Ricci-curvature condition.

In the control problem, the required semiconvexity of the potential
$V$ is not assumed directly: it follows from joint convexity of the
quadratically corrected flux Lagrangian. The examples below reduce
this verification to ordinary convexity of a state term and concavity
of explicitly transformed reference fluxes. Once semiconvexity is
available, however, the entropy-production argument uses neither the
HJB equation nor the control origin of the potential. It therefore
applies to any prescribed potential that is continuous on $\D$,
globally $K$-semiconvex, and locally $\mathcal C^{1,1}$ on $\Dint$,
along any global interior solution of the associated Gibbs dynamics.
Global well-posedness for such a separately prescribed potential must
be justified separately. Thus the stability mechanism extends beyond
the present mean field control problem and requires neither an
$N$-particle approximation nor an entropic Ricci-curvature bound. The
paired exponential fluxes are also related to the cosh-type generalized
gradient structures studied in
\cite{MielkePeletierRenger2014,PeletierSchlichting2023}. The specific
mechanisms behind our results are summarized next.

\subsubsection{Market-making application}

The inventory approach to dealer pricing has a long history. Ho and Stoll \cite{mmHS1981} model stochastic trading demand by Poisson jumps together with price risk. Avellaneda and Stoikov \cite{mmAS2008} formulate quote control with price-sensitive execution intensities and exponential utility of terminal wealth. Their exponential intensity specification gives the standard tractable trade-off between spread earned per execution and execution frequency. Gu\'eant, Lehalle, and Fernandez-Tapia \cite{mmGLFT2013} analyze inventory constraints, a linear-ODE reduction, and long-horizon behavior of optimal quotes. Thus neither controlled inventory jumps nor long-horizon quoting alone is the contribution here.

Several extensions clarify the scope of this execution-based approach. Guilbaud and Pham \cite{mmGP2013} combine limit and market orders, with a finite-state spread process and Cox execution processes. Fodra and Labadie \cite{mmFL2013} allow a broader Markov environment including directional information. Gu\'eant \cite{mmG2017} develops a general framework and multi-asset approximations. Cartea, Jaimungal, and Ricci \cite{mmCJR2014} incorporate mutually exciting order flow, going beyond a memoryless Poisson specification. These models belong to the controlled point-process literature; they should not all be described as independent homogeneous Poisson models. Our inventory chains can be constructed from Poisson random measures, but their controlled, state-dependent intensities do not imply independent stationary increments for realized executions.

Bergault and Gu\'eant \cite{mmBG2021} provide the closest link to the present application: their marked executions allow different sizes, and their expected-wealth objective with running and terminal inventory penalties can be rewritten using spread revenue. Their state is one dealer's multi-asset inventory; ours is a population distribution over finitely many positions in one asset. Barzykin, Bergault, and Gu\'eant \cite{mmBBG2023} study FX client tiers, size-dependent pricing ladders, and hedging using fitted logistic demand. Their evidence supports size-indexed quoting but does not establish our common sensitivity to total dollar concession. We therefore state that scaling as an explicit assumption, rather than as an empirical consequence of these references.

Mean field interaction has also been introduced into market making and dynamic pricing. Baldacci, Bergault, and Possama\"i \cite{mmBBP2023} study a major market maker facing strategic minor traders; our population instead consists of dealers whose inventories enter a potential running cost. Donnelly and Li \cite{mmDL2025} study discrete-inventory sellers of a perishable product with competition through quoted prices. That interaction involves controls, whereas our direct planner-to-game correspondence uses population-independent reference intensities and state-distribution coupling in costs. We do not claim to encompass either model. The contribution of the financial specialization is the combination of an exact revenue-to-entropy identity, a discounted potential-game construction, and quantitative stabilization of the induced inventory distribution. It does not include a finite-player approximate-Nash theorem or empirical calibration.

\subsection{Proof architecture}

The viscosity proof begins with the observation that entropy coercivity
controls integrated population flux, even though the transition rates
themselves are unbounded. For a nearly optimal short-time control, the
entropy dual inequality bounds
$\int\sum_{i\ne j}\eta_i a_{ij}$ and therefore the total displacement
of the population state. This keeps the trajectory inside the
test-function neighborhood and permits the viscosity inequalities to
be proved even at boundary points. A change of unknown then puts the
HJB equation in the source-mass form covered by Kraaij's comparison
principle.

The value limits use two complementary mechanisms. As
$\beta\downarrow0$, comparison with the reference control forces
the discounted entropy expenditure of nearly optimal controls to be
$O(\beta)$; entropy coercivity then gives convergence of fluxes,
trajectories, and values. As $\beta\to\infty$, a fast-steering
competitor reaches a minimizer of $C$ in a vanishing initial layer
while its entropy cost divided by $\beta$ tends to zero.

The regularity argument has two complementary parts. Interpolating
initial states and fluxes in the quadratically corrected problem gives
global semiconvexity. Semiconcavity is more delicate because the
entropy terms appear singular as a population coordinate approaches
zero. We apply the same open-loop generator to nearby initial
conditions. Positivity of its stochastic evolution family shows that
the perturbation of each coordinate remains bounded relative to the
corresponding unperturbed mass. Consequently, large transition rates
enter the second-difference estimate only through the macroscopic
fluxes, whose discounted integral is controlled by entropy coercivity.
This first gives interior $\mathcal C^{1,1}$ regularity and identifies
the Gibbs feedback. The HJB equation and semiconvexity then give
uniform bounds on the optimal rates; freezing this verified bounded
control makes the semiconcavity estimate uniform up to the boundary
and yields intrinsic $\mathcal C^{1,1}$ regularity on the closed
simplex.

{\color{black}
\textcolor{black}{Under boundedness of $\beta K_\beta$, semiconvexity on short segments and the HJB}
equation additionally control the scaled gradient differences in the
Gibbs exponent, yielding uniform feedback convergence. Gronwall's
inequality and discounted tail bounds then give the finite-horizon
and discounted infinite-horizon trajectory limits.
}

For stability, symmetry pairs the two orientations of every edge and
produces an exact free-energy dissipation identity. The quantitative
estimate splits the simplex into two regions. In the interior, a
complete-graph identity and strong convexity control the free-energy
gap by the dissipation. Near the boundary, the logarithmic marginal
free-energy differences diverge and compensate for the vanishing
population masses, leaving a uniformly positive dissipative
contribution from at least one edge.

{\color{black}
\textcolor{black}{When $1-2\beta K_\beta$ has a uniform positive lower bound, the existing interior and boundary constants stay bounded away from zero for small $\beta$.}
Uniform strong convexity and a uniform free-energy gap bound yield
the common state-error prefactor. Minimization of the free energy
and the equilibrium HJB identity give the extreme-temperature
equilibrium limits under their respective hypotheses.
}

Finally, for the potential-MFG
correspondence, a first-order dynamic-programming sensitivity identity
derives the planner adjoint equation without evaluating a Hessian of
the merely $\mathcal C^{1,1}$ value function along the optimal
trajectory.

\subsection{Organization}

The remainder of the paper is organized as follows. \Cref{sec:2}
formulates the entropy-regularized control problem and states the
hierarchy of assumptions. \Cref{sec:main-results} contains four
main-result subsections: viscosity well-posedness, regularity and
Gibbs feedback, global stability, and the potential-MFG correspondence.
The inverse-temperature refinements appear with the general results
whose assumptions they use. 
\Cref{sec:market-making} develops the market-making interpretation
and its potential-game specialization.
\Cref{sec:numerical-experiments} presents the numerical experiments.
 The detailed proofs follow the same order in
\Cref{app:viscosity,app:regularity,app:convergence,app:potential-mfg}.

\section{The model setup}\label{sec:2}

\subsection{Notation and control problem}

Fix $d\ge2$ and write $[d]=\{1,\ldots,d\}$.
Let $\one:=(1,\ldots,1)^\top\in\R^d$. The probability simplex and its tangent space are
\[
\D:=\left\{\eta\in\R_+^d:
\sum_{i=1}^d\eta_i=1\right\},
\qquad
\T:=\left\{h\in\R^d:
\sum_{i=1}^d h_i=0\right\}.
\]
Throughout,
$\Dint:=\{\eta\in\D:\eta_i>0\text{ for every }i\}$
denotes the relative interior of $\D$. 
Thus $\T$ is the space of mass-preserving
perturbations of a population distribution.

{\color{black}
For a function $u$ differentiable along the affine hull of $\D$, its
intrinsic gradient $D_{\T}u(\eta)$ is the unique vector in $\T$ such that
\[
Du(\eta)[h]=D_{\T}u(\eta)\cdot h,
\qquad h\in\T.
\]
Equivalently, if $\widetilde p\in\R^d$ is any ambient gradient
representative, then
\[
D_{\T}u(\eta)
=
\widetilde p-
\left(\frac1d\sum_{i=1}^d\widetilde p_i\right)\one.
\]
In particular,
\[
\bigl(D_{\T}u(\eta)\bigr)_j-
\bigl(D_{\T}u(\eta)\bigr)_i
=
\widetilde p_j-\widetilde p_i,
\]
so these component differences are independent of the chosen ambient
representative.
}

We regard population distributions as row vectors. Let
\[
A:=\left\{a\in\R^{d\times d}:a_{ij}\ge0\ (i\ne j),\quad a_{ii}=-\sum_{j\ne i}a_{ij}\right\}
\]
be the set of transition-rate matrices.

\begin{definition}[Admissible rate controls]\label{def:admissible}
For $\eta_0\in\D$, an admissible control is a measurable map $a:[0,\infty)\to A$ such that, for every $T<\infty$,
\[
\int_0^T\sum_{i\ne j}a_{ij}(t)\dd t<\infty.
\]
The admissible class is denoted by $\mathcal A(\eta_0)$.
\end{definition}

For $a\in\mathcal A(\eta_0)$, finite-dimensional linear ordinary differential equation (ODE) theory gives a unique absolutely continuous stochastic evolution family $P^a(s,t)$ satisfying
\begin{equation}\label{eq:evolution-family}
\partial_tP^a(s,t)=P^a(s,t)a(t),\qquad P^a(s,s)=I.
\end{equation}
Thus the controlled state is $\eta(t)=\eta_0P^a(0,t)$ and is the unique solution of
\begin{equation}\label{eq:state-equation}
\dot\eta_i(t)=\sum_{j\ne i}\eta_j(t)a_{ji}(t)-\eta_i(t)\sum_{j\ne i}a_{ij}(t),\qquad \eta(0)=\eta_0.
\end{equation}
In particular, $\eta(t)\in\D$ for all $t\ge0$. If $\eta_0\in\Dint$, then every admissible trajectory remains in the interior at every finite time: from \eqref{eq:state-equation} and Gronwall's inequality,
\begin{equation}\label{eq:interior-preservation}
\eta_i(t)\ge \eta_i(0)\exp\left(-\int_0^t\sum_{j\ne i}a_{ij}(s)\dd s\right)>0,
\qquad t<\infty.
\end{equation}

For $x,r\ge0$, define
\[
\ell(x,r):=
\begin{cases}
 x\log(x/r)-x+r,&r>0,\\
 0,&x=r=0,\\
 +\infty,&r=0<x.
\end{cases}
\]
For a row $a_i=(a_{ij})_{j\ne i}$ and {\it reference row} $q_i(\eta)$, set
\[
\KL(a_i\Vert q_i(\eta)):=\sum_{j\ne i}\ell(a_{ij},q_{ij}(\eta)).
\]
Fix $\rho>0$ and $\beta>0$. The {\it running cost} is
\[
f(\eta,a):=C(\eta)+\frac1\beta\sum_{i=1}^d\eta_i\KL(a_i\Vert q_i(\eta)),
\]
and the {\it discounted cost} and {\it discounted value} are
\begin{equation}\label{eq:value-function}
J(\eta_0;a):=\int_0^\infty e^{-\rho t}f(\eta(t),a(t))\dd t,
\qquad
V(\eta_0):=\inf_{a\in\mathcal A(\eta_0)}J(\eta_0;a).
\end{equation}
The cost is extended-real valued; finite cost is not part of admissibility.

For semiconvexity and entropy estimates we use the derived {\it fluxes} $w_{ij}=\eta_i a_{ij}$ and reference fluxes $r_{ij}(\eta)=\eta_iq_{ij}(\eta)$. Let
\[
W:=\left\{w\in\R^{d\times d}: w_{ij}\ge0\ (i\ne j),\quad w_{ii}=-\sum_{j\ne i}w_{ij}\right\},
\]
and define $b_i(w):=\sum_{j\ne i}w_{ji}-\sum_{j\ne i}w_{ij}$. The fluxes are derived variables rather than a second primitive admissible class. The flux Lagrangian is
\[
L(\eta,w):=C(\eta)+\frac1\beta\sum_{i\ne j}\ell\bigl(w_{ij},r_{ij}(\eta)\bigr).
\]
Whenever $w_{ij}=\eta_i a_{ij}$, the perspective identity $\ell(\eta_i a_{ij},\eta_iq_{ij})=\eta_i\ell(a_{ij},q_{ij})$ holds, with the natural convention on a zero-mass row.

For every $\eta\in\D$ and $p\in\mathbb R^d$, the {\it Hamiltonian} is: 
\begin{equation}\label{eq:hamiltonian-variational}
H(\eta,p)
=
\inf_{a\in A}
\left\{
f(\eta,a)+p\cdot(\eta a)
\right\},
\end{equation}
where
\[
p\cdot(\eta a)
=
\sum_{i\ne j}\eta_i a_{ij}(p_j-p_i).
\]
Indeed,
\[
\inf_{a\ge0}
\left\{
a(p_j-p_i)+\frac1\beta\ell(a,q_{ij}(\eta))
\right\}
=
\frac{q_{ij}(\eta)}{\beta}
\left(1-e^{-\beta(p_j-p_i)}\right).
\]
If $\eta_i>0$, the unique {\it edgewise minimizer} is
\[
a_{ij}=q_{ij}(\eta)e^{-\beta(p_j-p_i)}.
\]
When $\eta_i=0$, the corresponding row does not affect the variational
problem, which is why uniqueness of the minimizing rate matrix is
asserted only in $\Dint$.

The Hamiltonian can be explicitly written as
\begin{equation}\label{eq:hamiltonian}
H(\eta,p)=C(\eta)+\frac1\beta\sum_{i\ne j}\eta_iq_{ij}(\eta)\left(1-e^{-\beta(p_j-p_i)}\right).
\end{equation}
It is invariant under $p\mapsto p+c\one$, as required by the affine geometry of the simplex. The associated HJB equation is
\begin{equation}\label{eq:hjb}
\rho U(\eta)-H(\eta,DU(\eta))=0,\qquad \eta\in\D.
\end{equation}
Here derivatives are understood intrinsically on the affine hyperplane containing $\D$; only the differences $\partial_jU-\partial_iU$ are intrinsic.

{\color{black}
\begin{remark}[KL-constrained interpretation]\label{rem:kl-constraint}
For a fixed initial state $\eta_0$, write
\[
\begin{aligned}
\mathcal C_{\eta_0}(a)&:=\int_0^\infty e^{-\rho t}C(\eta(t))\dd t,\\
\mathcal K_{\eta_0}(a)&:=\int_0^\infty e^{-\rho t}\sum_{i=1}^d\eta_i(t)
\KL\bigl(a_i(t)\Vert q_i(\eta(t))\bigr)\dd t.
\end{aligned}
\]
The entropy-regularized problem is the Lagrangian formulation of cost minimization under a budget $R\ge0$ on the total discounted, population-weighted KL cost:
\[
\inf_{a\in\mathcal A(\eta_0)}\mathcal C_{\eta_0}(a)
\qquad\text{subject to}\qquad \mathcal K_{\eta_0}(a)\le R.
\]
Indeed, the Lagrangian at multiplier $\lambda=1/\beta$ is
\[
\mathcal C_{\eta_0}(a)+\frac1\beta\bigl(\mathcal K_{\eta_0}(a)-R\bigr)
=J(\eta_0;a)-\frac R\beta,
\]
which has the same minimizers as $J(\eta_0;\cdot)$. In particular, any finite-cost optimal control $a^\beta$ for the regularized problem also solves the constrained problem with $R=\mathcal K_{\eta_0}(a^\beta)$; this implication does not require convexity. Conversely, for a prescribed budget, the constrained problem admits this penalized representation whenever strong duality holds and a positive optimal multiplier is attained. Thus $1/\beta$ serves as the Lagrange multiplier associated with the KL budget.
\end{remark}
}

\subsection{Assumptions}\label{sec:assumptions}

We state separate assumptions for viscosity well-posedness and for regularity and stability. The first theorem uses only \Cref{ass:viscosity}.

\begin{assumption}[Viscosity-level hypotheses]\label{ass:viscosity}
The functions $q_{ij}:\D\to(0,\infty)$ are continuous for every $i\ne j$, and there are constants $0<q_{\min}\le q_{\max}<\infty$ such that
\[
q_{\min}\le q_{ij}(\eta)\le q_{\max},
\qquad
\eta\in\D.
\]
The state cost satisfies $C\in\mathcal{C}_b(\D)$.
\end{assumption}

A quantitative semiconvexity condition and $\mathcal{C}^{1,1}$ regularity enter in the second theorem, where we prove regularity and identify the Gibbs optimizer.

\begin{assumption}[Semiconvexity and regularity hypotheses]\label{ass:regularity}
In addition to \Cref{ass:viscosity}, each $q_{ij}$ and $C$ is of class
$\mathcal{C}^{1,1}$ on the affine simplex, in the sense of restriction
from a neighborhood. There exists a constant $K\ge0$ such that, with
$\Phi_K(\eta):=\frac K2\|\eta\|_2^2$, the modified flux Lagrangian
\begin{equation}\label{eq:modified-flux-lagrangian}
\widehat L_K(\eta,w)
:=
L(\eta,w)
+\rho\Phi_K(\eta)
-D\Phi_K(\eta)\cdot b(w)
=
L(\eta,w)
+\frac{\rho K}{2}\|\eta\|_2^2
-K\eta\cdot b(w)
\end{equation}
is jointly convex on $\D\times W$ as an extended-real-valued function.
\end{assumption}

The global exponential convergence theorem additionally assumes symmetry of the reference generator.

\begin{assumption}[Symmetric reference rates]\label{ass:symmetry}
For every $i\ne j$ and $\eta\in\D$,
$q_{ij}(\eta)=q_{ji}(\eta)$.
\end{assumption}

\begin{remark}[Role and scope of the assumptions]
\Cref{ass:viscosity} contains the hypotheses needed for the dynamic-programming and viscosity arguments. Continuity and uniform positivity of the reference rates ensure the required continuity and boundary structure of the Hamiltonian, while the uniform upper bound enters the entropy--flux estimates. No differentiability or convexity assumption is needed at this level. This is similar in spirit to the viscosity regime in Cecchin \cite{Cecchin2021}, where no convexity assumption is imposed. A notable difference is that Cecchin works with a compact action space and bounded transition rates, whereas our rate-control set is unbounded; here entropy coercivity provides the macroscopic flux control needed in the dynamic-programming argument.

\Cref{ass:regularity} adds the regularity and second-order lower bound used to obtain the Gibbs feedback. The $\mathcal{C}^{1,1}$ assumptions on $C$ and the $q_{ij}$ are used in the local second-difference estimates and, after identification of the optimal feedback, in the local Lipschitz continuity of the closed-loop vector field. Joint convexity of \eqref{eq:modified-flux-lagrangian} implies convexity of $V+\Phi_K$, hence global $K$-semiconvexity of $V$. This is closely related to the convexity assumptions used by Cecchin \cite{Cecchin2021} to obtain $\mathcal{C}^{1,1}$ regularity, where the corresponding flux formulation is assumed jointly convex in the flux and population variables. The case $K=0$ recovers ordinary joint convexity of our original flux Lagrangian, while $K>0$ allows the corresponding semiconvex extension through the quadratic correction $\Phi_K$.

Finally, \Cref{ass:symmetry} is used directly in the free-energy dissipation identity, where the two orientations of each edge are paired. State independence is not required for the stability theorem. The additional quantitative restriction $2\beta K<1$ in \Cref{thm:convergence} ensures that the strong convexity contributed by the entropy dominates the possible negative curvature of $V$. This differs from other approaches to long-time behavior of nonlinear Markov dynamics. In the continuous-state McKean--Vlasov setting, Malrieu \cite{Malrieu2001} assumes a convex even interaction potential and a uniformly convex confinement potential and derives quantitative convergence through the associated particle system and logarithmic Sobolev estimates. In the finite-state setting, Erbar, Fathi, and Schlichting \cite{ErbarFathiSchlichting2020} consider rates reversible with respect to a local Gibbs measure and formulate the relevant convexity condition through entropic Ricci curvature in a nonlinear discrete transportation geometry. Our argument instead uses the Euclidean geometry of the simplex: symmetry yields the edgewise dissipation identity, while semiconvexity together with $2\beta K<1$ supplies the strong convexity needed for the global entropy-production estimate.
\end{remark}

\begin{example}[A $K=0$ state-dependent class satisfying \Cref{ass:regularity}]
Let
\[
q_{ij}(\eta)=\alpha_{ij}-b_{ij}\eta_i,
\qquad
\alpha_{ij}>b_{ij}\ge0,
\]
and let
\[
C(\eta)=\frac{\kappa}{2}\|\eta\|_2^2,
\qquad
\kappa\ge
\frac{2}{\beta}\max_i\sum_{j\ne i}b_{ij}.
\]
Then $r_{ij}(\eta)=\alpha_{ij}\eta_i-b_{ij}\eta_i^2$ is concave, and the original flux Lagrangian is jointly convex. Indeed, writing
\[
\ell(w,r)=\bigl[w\log(w/r)-w\bigr]+r,
\]
the bracketed function is jointly convex and nonincreasing in $r$, so its composition with the concave function $r_{ij}$ is convex. The stated bound on $\kappa$ makes $C+\beta^{-1}\sum_{i\ne j}r_{ij}$ convex. Thus this family satisfies \Cref{ass:regularity} with $K=0$.
\end{example}

\begin{example}[State-dependent symmetric rates covered by \Cref{thm:convergence}]
A concrete sufficient condition for joint convexity of \eqref{eq:modified-flux-lagrangian} can be expressed in terms of the reference fluxes. With
$r_{ij}(\eta):=\eta_iq_{ij}(\eta)$, set
\[
\widetilde r_{ij}^{\,K}(\eta)
:=
r_{ij}(\eta)e^{\beta K(\eta_j-\eta_i)}.
\]
Since $(w,r)\mapsto w\log(w/r)-w$ is jointly convex and nonincreasing in $r$, it is sufficient that every $\widetilde r_{ij}^{\,K}$ be concave and that
\[
\eta\longmapsto
C(\eta)
+\frac{\rho K}{2}\|\eta\|_2^2
+\frac1\beta\sum_{i\ne j}r_{ij}(\eta)
\]
be convex.

For instance, when $d\ge3$, let
\[
q_{ij}(\eta)=q_{ji}(\eta)
=
\bar q_{ij}e^{-\gamma(\eta_i+\eta_j)},
\qquad
\bar q_{ij}=\bar q_{ji}>0,
\qquad
0<\gamma<\frac12,
\]
and take $K=\gamma/\beta$. Then
$\widetilde r_{ij}^{\,K}(\eta)=\bar q_{ij}\eta_i e^{-2\gamma\eta_i}$,
which is concave because $x\mapsto xe^{-2\gamma x}$ is concave on $[0,1]$ for $\gamma\le1$. The remaining state term can be made convex, for example, by choosing $C$ sufficiently strongly convex. Since $2\beta K=2\gamma<1$, this gives genuinely state-dependent symmetric rates satisfying the hypotheses of \Cref{thm:convergence}.
\end{example}

\begin{example}[Constant symmetric reference rates]
Suppose $q_{ij}(\eta)\equiv\bar q_{ij}=\bar q_{ji}>0$ for $i\ne j$,
and $C$ is convex and of class $\mathcal{C}^{1,1}$ on the affine simplex.
Then \Cref{ass:viscosity,ass:regularity,ass:symmetry} hold with $K=0$.
Indeed, each reference flux $r_{ij}(\eta)=\eta_i\bar q_{ij}$ is affine,
and joint convexity of the relative-entropy integrand implies joint
convexity of the flux Lagrangian. Consequently, the condition
$2\beta K<1$ holds for every $\beta>0$.
\end{example}

{\color{black}
\begin{example}[A fixed state-dependent model for all temperatures]
\label{ex:temperature-state-dependent}
Let $d=2$, write $\eta=(x,1-x)$, fix $q_0>0$ and $0<\sigma<1$,
and take any fixed convex $\mathcal C^{1,1}$ running cost. Set
\begin{equation}\label{eq:temperature-state-dependent-rates}
q_{12}(\eta)=q_{21}(\eta)=q_0e^{-\sigma x},
\qquad K_\beta=\frac{\sigma}{2\beta}.
\end{equation}
The modified reference fluxes in the sufficient criterion above are
$\widetilde r_{12}^{\,K_\beta}(x)=q_0e^{\sigma/2}xe^{-2\sigma x}$
and $\widetilde r_{21}^{\,K_\beta}(x)=q_0e^{-\sigma/2}(1-x)$.
The first is concave on $[0,1]$, since its second derivative has
the sign of $\sigma x-1$, and the second is affine. Moreover,
$r_{12}+r_{21}=q_0e^{-\sigma x}$ is convex. The remaining state
term in that criterion is therefore convex for every $\beta>0$.
Thus \Cref{ass:regularity,ass:symmetry} hold with fixed $C$ and $q$
for all temperatures. Here $\beta K_\beta=\sigma/2$ does not tend
to zero, but $1-2\beta K_\beta=1-\sigma$ is uniformly positive.
This example is covered by the small-inverse-temperature feedback
and stability results below, as well as the large-$\beta$
equilibrium conclusion.
\end{example}
}

\section{Main results}\label{sec:main-results}

{\color{black}
This section presents four groups of main results, following the
hierarchy of assumptions. First, under \Cref{ass:viscosity}, we
characterize the value function as the unique bounded viscosity
solution on the closed simplex. Monotonicity and both value limits
are refinements at this level, with Lipschitz reference rates for the
small-$\beta$ limit. Second, under \Cref{ass:regularity}, we establish
global semiconvexity, intrinsic $\mathcal C^{1,1}$ regularity, and the
optimal Gibbs feedback; parameter-uniform regularity yields the
small-$\beta$ feedback and trajectory limits. Third, under symmetry
and $2\beta K<1$, we prove global exponential stability, followed by
its uniform small-$\beta$ refinement and the extreme-temperature
equilibrium limits under their respective hypotheses. Fourth, under
the state-independent-reference and potential conditions, we obtain
the planner-to-MFG correspondence.
}

\subsection{Main result 1: Viscosity solutions on the closed simplex}

Let
\[
\mathsf H
:=
\left\{
\eta\in\mathbb R^d:\sum_{i=1}^d\eta_i=1
\right\}.
\]
Test functions are restrictions to $\D$ of $\mathcal{C}^1$ functions
defined on a neighborhood of $\D$ in $\mathsf H$. Their differentials
are intrinsic covectors on $\T$. Equivalently, one may use any ambient
gradient representative, since two representatives differ by a multiple
of $\one$ and
\[
H(\eta,p+c\one)=H(\eta,p).
\]

\begin{definition}[Viscosity solution on the closed simplex]
\label{def:closed-simplex-viscosity}
A bounded upper semicontinuous function $u:\D\to\mathbb R$ is a
viscosity subsolution of \eqref{eq:hjb} on $\D$ if, whenever
$u-\phi$ has a local maximum at $\eta_0\in\D$ relative to $\D$, one has
\[
\rho u(\eta_0)-H(\eta_0,D\phi(\eta_0))\le0.
\]

A bounded lower semicontinuous function $u:\D\to\mathbb R$ is a
viscosity supersolution of \eqref{eq:hjb} on $\D$ if, whenever
$u-\phi$ has a local minimum at $\eta_0\in\D$ relative to $\D$, one has
\[
\rho u(\eta_0)-H(\eta_0,D\phi(\eta_0))\ge0.
\]

A continuous function is a viscosity solution on the closed simplex if
it is both a subsolution and a supersolution on $\D$.
\end{definition}

This is the closed-domain viscosity convention used in
\cite[Definition~1]{Kraaij2016}. It imposes both viscosity inequalities
at boundary points and is therefore stronger there than the usual
Soner state-constraint convention, in which the subsolution inequality
is imposed only on $\Dint$. The dynamic-programming argument below
establishes this stronger boundary subsolution property directly.

\begin{theorem}[Closed-simplex viscosity characterization]
\label{thm:viscosity}
Under \Cref{ass:viscosity}, the value function in
\eqref{eq:value-function} is continuous on $\D$ and is the unique
bounded viscosity solution of \eqref{eq:hjb} on the closed simplex in
the sense of \Cref{def:closed-simplex-viscosity}.
\end{theorem}

\emph{Sketch of proof.}
We first establish the dynamic programming principle in \Cref{lem:dpp}. The admissible class is stable under time translation and finite concatenation: local integrability of the transition rates is preserved by both operations, and uniqueness of the linear state equation allows the corresponding trajectories to be spliced at any fixed time. Splitting the discounted cost and concatenating with an $\varepsilon$-optimal continuation therefore gives the dynamic programming identity directly, without requiring the existence of an optimal control or a measurable-selection argument.

The viscosity inequalities then follow from the dynamic programming principle. For the subsolution inequality, one tests over a short time interval with a fixed finite rate matrix attaining the pointwise minimum in the Hamiltonian and lets the length of the interval tend to zero. The supersolution inequality requires an additional argument because a nearly optimal short-time control need not have uniformly bounded transition rates. Near optimality nevertheless gives a uniform bound on its entropy expenditure. The entropy dual inequality
$\ell(a,q)\geq \gamma a-q(e^\gamma-1)$
then controls the integrated macroscopic fluxes $\eta_i a_{ij}$; choosing $\gamma=\frac12\log(1/h)$ shows that the total transported mass over an interval of length $h$ tends to zero with $h$. Since the norm of the state velocity is controlled by the total flux, the corresponding nearly optimal trajectory remains in the touching neighborhood. The usual test-function and chain-rule argument can therefore be applied and yields the supersolution inequality up to and including the boundary.

Finally, after the change of unknown $u=-\beta U$, the HJB equation becomes the finite-state mean-field Hamilton--Jacobi equation with jump intensities
$v(i,j,\eta)=\eta_iq_{ij}(\eta)$. This is precisely the source-mass structure covered by Kraaij's comparison theorem \cite[Theorem~4]{Kraaij2016}. Comparison identifies the upper and lower semicontinuous envelopes of the value function, proving continuity and uniqueness of the bounded viscosity solution on the closed simplex. The details are given in \Cref{app:viscosity}.

{\color{black}
To make the dependence on the inverse-temperature parameter explicit,
write $J_\beta$ and $V_\beta$ for the cost and value in
\eqref{eq:value-function} when the coefficient of the KL penalty is
$1/\beta$.

\begin{theorem}[Dependence on the inverse-temperature parameter]
\label{thm:beta-limits}
Suppose \Cref{ass:viscosity} holds, and keep $\rho$, $C$, and $q$
fixed as $\beta$ varies.

\begin{enumerate}[label=(\roman*)]
\item If $0<\beta_1<\beta_2$, then
\[
V_{\beta_1}(\eta)\ge V_{\beta_2}(\eta),
\qquad \eta\in\D.
\]

\item Suppose additionally that every $q_{ij}$ is Lipschitz continuous
on $\D$. Complete the reference rates to a generator by setting
\[
q_{ii}(\eta):=-\sum_{j\ne i}q_{ij}(\eta),
\]
let $\eta^q(\cdot;\eta_0)$ be the unique solution of
\[
\dot\eta^q(t)
=
\eta^q(t)q\bigl(\eta^q(t)\bigr),
\qquad
\eta^q(0)=\eta_0,
\]
and define the value of the reference dynamics by
\[
V^q(\eta_0)
:=
\int_0^\infty e^{-\rho t}
C\bigl(\eta^q(t;\eta_0)\bigr)\dd t.
\]
Then
\[
\lim_{\beta\downarrow0}
\|V_\beta-V^q\|_\infty=0.
\]

\item Let
\[
C_{\min}:=\min_{\eta\in\D}C(\eta).
\]
Then, without any additional Lipschitz or convexity assumption,
\[
\lim_{\beta\to\infty}
\left\|
V_\beta-\frac{C_{\min}}{\rho}
\right\|_\infty=0.
\]
\end{enumerate}

Consequently, when the reference rates are Lipschitz,
\[
\frac{C_{\min}}{\rho}
\le V_\beta(\eta)\le V^q(\eta),
\qquad \eta\in\D,
\]
and $V_\beta$ decreases monotonically from the reference-dynamics
value to $C_{\min}/\rho$ as $\beta$ increases from $0$ to $\infty$.
\end{theorem}

\emph{Sketch of proof.}
Monotonicity follows immediately because $1/\beta$ is the coefficient
of a nonnegative cost. As $\beta\downarrow0$, comparison with the
reference control shows that the discounted KL expenditure of every
nearly optimal control is $O(\beta)$. Entropy coercivity then forces
its population flux and trajectory to approach those of the reference
dynamics on every finite time interval. As $\beta\to\infty$, one may
use rates of order $\sqrt{\beta}$ to steer the population toward a
minimizer of $C$ in a vanishing initial layer, while the corresponding
KL cost divided by $\beta$ tends to zero. The full proof is given
in \Cref{app:viscosity}.

\begin{remark}[Value convergence versus feedback convergence]
\label{rem:beta-feedback-limit}
The uniform convergence $V_\beta\to V^q$ in
\Cref{thm:beta-limits} does not, by itself, imply convergence of the
Gibbs feedback, because the feedback depends on the scaled intrinsic
gradient differences
\[
\beta\left[
\bigl(D_{\T}V_\beta\bigr)_j
-
\bigl(D_{\T}V_\beta\bigr)_i
\right].
\]
\color{black}
Uniform convergence of the value functions does not generally control
their derivatives. \Cref{thm:uniform-small-beta-feedback} shows that,
under the additional
parameter-uniform regularity condition
$\sup_{0<\beta\le\beta_0}\beta K_\beta<\infty$, the HJB equation and semiconvexity on short segments control
these scaled gradient differences directly. In particular, convergence
or uniform boundedness of $D_{\T}V_\beta$ itself is not required.

The situation as $\beta\to\infty$ is different. In that regime the
relative-entropy penalty vanishes, and
\Cref{thm:beta-limits} supplies no compactness or convergence result
for the optimal rates. Increasingly fast rates and an impulsive
limiting behavior are possible, and no large-$\beta$ feedback limit is
asserted here.
\end{remark}
}

\subsection{Main result 2: Regular solutions and Gibbs form}

The second result establishes the regularity needed to identify the optimal control explicitly. Under \Cref{ass:regularity}, the interpolation argument is applied to the modified flux Lagrangian \eqref{eq:modified-flux-lagrangian}, yielding convexity of $V+\Phi_K$ and therefore global $K$-semiconvexity of $V$. The complementary semiconcavity argument is first local in the interior: the same nearly optimal open-loop rate control is applied to nearby initial conditions, and the resulting perturbations remain controlled relative to the evolving population mass. This relative control prevents the logarithmic entropy terms from producing singular second-order errors and ensures that potentially large transition rates enter the estimates only through the corresponding macroscopic fluxes, which are controlled by entropy coercivity. After the Gibbs feedback is identified and its rates are bounded uniformly, the perturbation argument becomes uniform up to the boundary. Combining the resulting global semiconcavity with global semiconvexity gives intrinsic $\mathcal{C}^{1,1}$ regularity on $\D$.

The HJB equation can therefore be interpreted pointwise in the interior. Its minimization in the rate variables separates over directed edges, and the strict convexity of the entropy term gives a unique minimizer on each edge. This yields the explicit Gibbs feedback stated below. The semiconvexity estimate also replaces the convexity bound used to obtain uniform intrinsic-gradient bounds, and hence uniform upper and lower bounds on the optimal rates.

\begin{theorem}[Global regularity and Gibbs feedback]\label{thm:gibbs}
Under \Cref{ass:regularity}, the function $V+\Phi_K$ is convex on $\D$, equivalently $V$ is globally $K$-semiconvex. Moreover, the intrinsic gradient $D_{\T}V$, initially defined on $\Dint$, extends to a globally Lipschitz map on $\D$; equivalently, $V$ is intrinsically $\mathcal{C}^{1,1}$ on the closed simplex. For $\eta\in\Dint$, the Hamiltonian has the unique minimizing rate matrix
\begin{equation}\label{eq:gibbs-feedback}
{\color{black}
a^*_{ij}(\eta)
=
q_{ij}(\eta)
\exp\left[
-\beta\left(
\bigl(D_{\T}V(\eta)\bigr)_j
-
\bigl(D_{\T}V(\eta)\bigr)_i
\right)
\right],
\qquad i\ne j.
}
\end{equation}
Moreover, there are constants $0<\underline a\le\overline a<\infty$, depending only on the model data, such that
\begin{equation}\label{eq:global-rate-bounds}
\underline a\le a^*_{ij}(\eta)\le\overline a,
\qquad \eta\in\Dint,\quad i\ne j.
\end{equation}
For every $\eta_0\in\Dint$, the closed-loop equation generated by \eqref{eq:gibbs-feedback} has a unique global solution remaining in $\Dint$ at every finite time. The feedback is optimal, and every optimal admissible rate control agrees with $a^*(\eta(\cdot))$ almost everywhere along its trajectory. \textcolor{black}{
Finally, the right-hand side of \eqref{eq:gibbs-feedback} extends
continuously to every $\eta\in\D$ and defines a minimizing selection
there. At a boundary state, each row $i$ with $\eta_i>0$ is uniquely
determined by \eqref{eq:gibbs-feedback}, whereas a row with $\eta_i=0$
does not enter the pointwise Hamiltonian minimization and may be chosen
arbitrarily. Thus the full minimizing rate matrix is generally
nonunique at the boundary.
}
\end{theorem}

\emph{Sketch of proof.}
The discounted chain rule shows that adding the correction in \eqref{eq:modified-flux-lagrangian} adds exactly $\Phi_K(\eta_0)$ to the cost. Thus $V+\Phi_K$ is the value function for the modified flux problem. Interpolation of initial states and fluxes, together with joint convexity of the modified Lagrangian, gives convexity of $V+\Phi_K$ and hence global $K$-semiconvexity of $V$.

For semiconcavity, fix a compact set $K_0\subset\Dint$ and perturb an initial condition $\eta\in K_0$ to $\eta\pm h$. Apply the same nearly optimal open-loop rate control to the three initial conditions. Once the rates are fixed, the state equation is linear, so the perturbed trajectories remain symmetric around the original trajectory. More importantly, positivity of the associated stochastic evolution family gives the relative perturbation estimate \eqref{eq:relative-perturbation}: the perturbation of each coordinate remains controlled by the corresponding mass of the unperturbed trajectory. This prevents the logarithmic entropy terms from producing singular second-order errors. Potentially large transition rates consequently enter the estimates only through the macroscopic fluxes $\eta_i a_{ij}$, whose discounted integral is controlled by entropy coercivity. The resulting quadratic upper bound on second differences gives local semiconcavity and hence interior $\mathcal{C}^{1,1}$ regularity. After the uniform rate bounds are available, the detailed proof upgrades this estimate to global semiconcavity on $\D$; together with global semiconvexity, this yields the asserted intrinsic $\mathcal{C}^{1,1}$ regularity on the closed simplex.

The HJB equation is therefore classical in the interior. Its minimization separates over directed edges, and each scalar entropy-regularized edge objective is strictly convex. Its first-order condition gives \eqref{eq:gibbs-feedback}. The HJB equation bounds the total optimal macroscopic flux uniformly. Choosing a coordinate with mass at least $1/d$, the semiconvex inequality for $V$ supplies the missing opposite bound on the intrinsic gradient differences; substitution into \eqref{eq:gibbs-feedback} yields \eqref{eq:global-rate-bounds}. The global rate bound prevents the closed-loop solution from reaching the boundary in finite time, and local $\mathcal{C}^{1,1}$ regularity gives uniqueness. Verification follows from the Hamiltonian gap exactly as in the detailed proof. The full proof is given in \Cref{app:regularity}.

{\color{black}
\begin{theorem}[Uniform small-$\beta$ limit of the Gibbs feedback]
\label{thm:uniform-small-beta-feedback}
Keep $\rho$, $C$, and $q$ fixed. Suppose that there exists
$\beta_0>0$ such that, for every $\beta\in(0,\beta_0]$, the
corresponding instance of \Cref{ass:regularity} holds with
semiconvexity constant $K_\beta$, and assume that
\begin{equation}\label{eq:small-beta-semiconvexity-condition}
L:=\sup_{0<\beta\le\beta_0}\beta K_\beta<\infty.
\end{equation}
Let $V_\beta$ be the corresponding value function, and let
$a^{*,\beta}$ be the continuous Gibbs minimizing selection on $\D$
defined by \eqref{eq:gibbs-feedback} with $V=V_\beta$. Then
\begin{equation}\label{eq:uniform-feedback-convergence}
\max_{i\ne j}\sup_{\eta\in\D}
\left|
a_{ij}^{*,\beta}(\eta)-q_{ij}(\eta)
\right|
\longrightarrow0
\qquad\text{as }\beta\downarrow0.
\end{equation}
{\color{black}Under \eqref{eq:small-beta-semiconvexity-condition}, the
convergence has order $O(\sqrt\beta)$.}
In particular, if one constant $K$ can be used in
\Cref{ass:regularity} for every sufficiently small $\beta$, then the
convergence in \eqref{eq:uniform-feedback-convergence} is of order
$O(\beta)$.
\end{theorem}
{\color{black}
\emph{Sketch of proof.}
The value functions have uniformly bounded oscillation. Apply
semiconvexity on a segment of length proportional to $\sqrt\beta$
toward a vertex maximizing the scaled gradient. This makes the
largest component of $\beta D_{\T}V_\beta$ differ from its
population-weighted average by $O(\sqrt\beta)$, even when
$K_\beta$ grows like $1/\beta$. The HJB equation then controls
the average above the smallest component by the same order.
All scaled gradient differences consequently vanish uniformly,
which proves the feedback limit. If $K_\beta$ is uniformly bounded,
a full segment instead gives order $O(\beta)$. The proof appears
in \Cref{app:regularity}.
}

\begin{corollary}[Small-$\beta$ convergence of the closed-loop dynamics]
\label{cor:small-beta-trajectories}
Under the hypotheses of \Cref{thm:uniform-small-beta-feedback}, let
$\eta^\beta(\cdot;\eta_0)$ and $\eta^q(\cdot;\eta_0)$ denote the
trajectories with common initial state $\eta_0\in\D$ generated by
$a^{*,\beta}$ and $q$, respectively. Then
\[
\sup_{\eta_0\in\D}\sup_{t\ge0}\max_{i\ne j}
\left|a_{ij}^{*,\beta}(\eta^\beta(t;\eta_0))
-q_{ij}(\eta^\beta(t;\eta_0))\right|\longrightarrow0.
\]
For every $T<\infty$,
\begin{equation}\label{eq:small-beta-finite-horizon-trajectories}
\sup_{\eta_0\in\D}\sup_{0\le t\le T}
\|\eta^\beta(t;\eta_0)-\eta^q(t;\eta_0)\|_1
\longrightarrow0.
\end{equation}
Moreover, the following discounted infinite-horizon limits hold:
\begin{equation}\label{eq:small-beta-discounted-trajectories}
\sup_{\eta_0\in\D}\int_0^\infty e^{-\rho t}
\|\eta^\beta(t;\eta_0)-\eta^q(t;\eta_0)\|_1\dd t
\longrightarrow0
\end{equation}
and
\begin{equation}\label{eq:small-beta-discounted-rates}
\sup_{\eta_0\in\D}\int_0^\infty e^{-\rho t}
\sum_{i\ne j}\left|
a_{ij}^{*,\beta}(\eta^\beta(t;\eta_0))
-q_{ij}(\eta^q(t;\eta_0))\right|\dd t
\longrightarrow0.
\end{equation}
All limits are as $\beta\downarrow0$.
\end{corollary}

\begin{remark}[Scope of the uniform limit]
\label{rem:small-beta-feedback-scope}
The convergence in \eqref{eq:uniform-feedback-convergence} is global
in the state variable: it holds on the entire closed simplex.
At a boundary state, however, the assertion concerns the canonical
continuous Gibbs selection. Arbitrary minimizing rows with
$\eta_i=0$ remain nonunique and need not converge. For arbitrary
nearly optimal open-loop controls, entropy estimates control
mass-weighted fluxes and do not by themselves imply uniform
convergence of every rate.

The boundedness condition \eqref{eq:small-beta-semiconvexity-condition}
is additional to \Cref{ass:regularity} at each fixed $\beta$.
It holds for constant reference rates and convex $C$, with $K_\beta=0$,
and also for the state-dependent family in
\Cref{ex:temperature-state-dependent}.

{\color{black}
The stronger condition $\beta K_\beta\to0$, in conjunction with
\Cref{ass:regularity} and fixed $C$ and $q$, would force the
reference rates to be state-independent.
Indeed, as $\beta\downarrow0$,
\[
\beta\widehat L_{K_\beta}(\eta,w)
\longrightarrow
F(\eta,w):=\sum_{i\ne j}\ell(w_{ij},r_{ij}(\eta)).
\]
Hence $F$ is jointly convex.
Setting $w=0$ shows that $\sum_{i\ne j}r_{ij}$ is convex.
For a fixed edge $(i,j)$, set all other off-diagonal
fluxes to zero and minimize $F-sw_{ij}$ over
$w_{ij}\ge0$. The resulting function
\[
\sum_{k\ne l}r_{kl}(\eta)-e^s r_{ij}(\eta)
\]
is convex for every $s\in\mathbb R$.
Multiplying by $e^{-s}$ and letting $s\to\infty$
shows that $r_{ij}$ is concave.
Thus the sum of these concave functions is affine,
and each summand $r_{ij}$ must be affine.
Since $r_{ij}$ vanishes on the face $\{\eta_i=0\}$,
it has the form $r_{ij}(\eta)=\bar q_{ij}\eta_i$.
Continuity gives $q_{ij}\equiv\bar q_{ij}$ on $\D$.
This rigidity argument does not apply to
\eqref{eq:small-beta-semiconvexity-condition}: the scaled quadratic
correction need not disappear. The example just cited shows that
the relaxed hypothesis genuinely admits state-dependent rates.
}

The trajectory conclusions are uniform on finite time intervals
and on the discounted infinite horizon. Uniform-in-time convergence
$\sup_{t\ge0}\|\eta^\beta(t;\eta_0)-\eta^q(t;\eta_0)\|_1\to0$
is not asserted without an additional contraction or
incremental-stability assumption.
\end{remark}
}

\subsection{Main result 3: Free-energy dissipation and global exponential stability}
The third result concerns the long-time behavior of the nonlinear closed-loop dynamics generated by the Gibbs feedback. In addition to \Cref{ass:regularity}, we impose the symmetry condition \Cref{ass:symmetry}. The reference rates may depend on $\eta$. Let $K$ be the semiconvexity constant from \Cref{ass:regularity} and assume $2\beta K<1$.

Define
\textcolor{black}{the {\it free-energy (Lyapunov) functional}}
\begin{equation}\label{eq:free-energy}
\mathcal F(\eta)
:=
\sum_{i=1}^d\eta_i\log\eta_i+2\beta V(\eta),
\qquad \eta\in\D,
\end{equation}
with $0\log0=0$, and put $\mu:=1-2\beta K>0$. For $\eta\in\Dint$, \textcolor{black}{define the {\it free-energy potentials}}
\[
\psi_i(\eta):=\log\eta_i+2\beta\partial_iV(\eta).
\]
Only differences of the $\psi_i$ are intrinsic; $\psi_j-\psi_i$ is the marginal change in free energy associated with transferring mass from state $i$ to state $j$. The orthogonal projection onto $\T$ is
\begin{equation}\label{eq:tangent-projection}
\Pi_{\T}z
:=
z-\left(\frac1d\sum_{i=1}^dz_i\right)\one.
\end{equation}

Using \eqref{eq:gibbs-feedback}, the optimal flux can be written as
\begin{equation}\label{eq:gibbs-flux}
\eta_i a^*_{ij}(\eta)
=
q_{ij}(\eta)\sqrt{\eta_i\eta_j}\,
e^{-(\psi_j(\eta)-\psi_i(\eta))/2}.
\end{equation}
Define
\textcolor{black}{the {\it free-energy dissipation rate}}
\begin{equation}\label{eq:dissipation}
\mathcal D(\eta)
:=
\frac12\sum_{i\ne j}
q_{ij}(\eta)\sqrt{\eta_i\eta_j}\,
\bigl(\psi_j(\eta)-\psi_i(\eta)\bigr)
\left(
e^{(\psi_j(\eta)-\psi_i(\eta))/2}
-
e^{-(\psi_j(\eta)-\psi_i(\eta))/2}
\right).
\end{equation}
Every term is nonnegative because $x(e^{x/2}-e^{-x/2})\ge0$.

\begin{theorem}[Global exponential convergence]\label{thm:convergence}
Under \Cref{ass:regularity,ass:symmetry}, suppose $2\beta K<1$. Then $\mathcal F$ is $\mu$-strongly convex on the affine simplex, has a unique minimizer $\eta^\infty\in\Dint$, and this point is the unique equilibrium of the closed-loop dynamics in $\Dint$. For every $\eta_0\in\Dint$, along the closed-loop solution of \Cref{thm:gibbs} with $\eta(0)=\eta_0$,
\begin{equation}\label{eq:dissipation-identity}
\frac{\dd}{\dd t}\mathcal F(\eta(t))
=-\mathcal D(\eta(t)),
\qquad
\mathcal D(\eta(t))\ge0.
\end{equation}
Moreover, there exists $\lambda>0$, depending only on the model data, such that the entropy-production inequality
\begin{equation}\label{eq:entropy-production}
\mathcal D(\eta)
\ge
\lambda\bigl(\mathcal F(\eta)-\mathcal F(\eta^\infty)\bigr),
\qquad
\eta\in\Dint,
\end{equation}
holds uniformly on the interior of the simplex. Consequently,
\begin{equation}\label{eq:free-energy-decay}
\mathcal F(\eta(t))-\mathcal F(\eta^\infty)
\le
e^{-\lambda t}
\bigl(\mathcal F(\eta(0))-\mathcal F(\eta^\infty)\bigr),
\end{equation}
and
\begin{equation}\label{eq:state-decay}
\|\eta(t)-\eta^\infty\|_2
\le
\sqrt{\frac{2}{\mu}\bigl(\mathcal F(\eta(0))-\mathcal F(\eta^\infty)\bigr)}\,
e^{-\lambda t/2}.
\end{equation}

\end{theorem}

\emph{Sketch of proof.}
The entropy $\sum_i\eta_i\log\eta_i$ is $1$-strongly convex on the affine simplex, while $V$ is $K$-semiconvex by \Cref{thm:gibbs}. Hence $\mathcal F$ is $\mu$-strongly convex. The projected Polyak--\L{}ojasiewicz (PL) and quadratic gap estimates in \Cref{lem:pl} therefore hold with the constant $\mu$.

The complete-graph identity \eqref{eq:complete-graph} links the edgewise chemical-potential differences to the projected gradient. Symmetry $q_{ij}(\eta)=q_{ji}(\eta)$ pairs the two orientations of each edge in \eqref{eq:gibbs-flux}, giving the dissipation identity \eqref{eq:dissipation-identity}.

It remains to prove the uniform estimate \eqref{eq:entropy-production}. On an interior region where all coordinates are bounded below, the edge weights in $\mathcal D$ are uniformly positive. Together with \eqref{eq:complete-graph} and the projected PL estimate, this bounds the free-energy gap by the dissipation. Near the boundary, if one coordinate is small and another has mass at least $1/d$, the semiconvex inequality bounds the value-function contribution to the corresponding chemical-potential difference by a uniform constant depending on $K$, while the logarithmic term diverges. A single edge then contributes a fixed positive amount to $\mathcal D$. Since the free-energy gap is bounded above on the compact simplex, \eqref{eq:entropy-production} also holds uniformly in the boundary region.

Combining the two regions proves \eqref{eq:entropy-production}. Using this estimate in \eqref{eq:dissipation-identity} and applying Gronwall's inequality gives \eqref{eq:free-energy-decay}, while the strong-convexity gap estimate then yields \eqref{eq:state-decay}. The full proof is given in \Cref{app:convergence}.

\begin{theorem}[Uniform exponential convergence for small inverse temperature]
\label{thm:uniform-small-beta-stability}
Keep $\rho$, $C$, and $q$ fixed, and assume
\Cref{ass:symmetry}. Suppose that there exists $\beta_0>0$ such that,
for every $\beta\in(0,\beta_0]$, \Cref{ass:regularity} holds with
semiconvexity constant $K_\beta$. Assume that, for some fixed
$\varepsilon\in(0,1]$,
\begin{equation}\label{eq:uniform-curvature-margin}
2\beta K_\beta\le1-\varepsilon,
\qquad 0<\beta\le\beta_0.
\end{equation}
(For example, positive constant symmetric reference rates and
a convex $\mathcal C^{1,1}$ state cost satisfy this condition
with $K_\beta=0$ and $\varepsilon=1$.
The fixed state-dependent model in
\Cref{ex:temperature-state-dependent} satisfies it with
$K_\beta=\sigma/(2\beta)$ and $\varepsilon=1-\sigma$.) 
Let $\mathcal F_\beta$ and $\mathcal D_\beta$ denote the free energy
and dissipation in \eqref{eq:free-energy} and \eqref{eq:dissipation}, respectively,
with $V$ replaced by $V_\beta$.

Then there exist $\beta_*\in(0,\beta_0]$ and constants
$\lambda_*>0$ and $B_*<\infty$, independent of $\beta$ and the
initial state, such that the following holds for every
$\beta\in(0,\beta_*]$.

The closed-loop dynamics has a unique equilibrium
$\eta_\beta^\infty\in\Dint$, and
\begin{equation}\label{eq:uniform-small-beta-entropy-production}
\mathcal D_\beta(\eta)
\ge
\lambda_*
\bigl(
\mathcal F_\beta(\eta)
-\mathcal F_\beta(\eta_\beta^\infty)
\bigr),
\qquad \eta\in\Dint.
\end{equation}
Consequently, for every $\eta_0\in\Dint$, the corresponding
closed-loop trajectory satisfies, for all $t\ge0$,
\begin{equation}\label{eq:uniform-small-beta-energy-decay}
\mathcal F_\beta(\eta^\beta(t;\eta_0))
-\mathcal F_\beta(\eta_\beta^\infty)
\le
e^{-\lambda_*t}
\bigl(
\mathcal F_\beta(\eta_0)
-\mathcal F_\beta(\eta_\beta^\infty)
\bigr)
\end{equation}
and
\begin{equation}\label{eq:uniform-small-beta-state-decay}
\|\eta^\beta(t;\eta_0)-\eta_\beta^\infty\|_2
\le B_*e^{-\lambda_*t/2}.
\end{equation}
One may take
\[
\lambda_*
:=
\min\left\{
\frac{q_{\min}(1-e^{-1})e^{-1}}
{2d(\log d+1)},
\;
2\varepsilon q_{\min}e^{-3}
\right\},
\qquad
B_*:=\sqrt{\frac{2(\log d+1)}{\varepsilon}}.
\]
\end{theorem}

\emph{Sketch of proof.}
This is a uniform-constant consequence of \Cref{thm:convergence}.
Comparison with the reference control bounds the oscillation of
$V_\beta$ independently of $\beta$. The margin in
\eqref{eq:uniform-curvature-margin} keeps the strong-convexity
modulus positive, while both the interior and boundary constants
from that theorem remain bounded away from zero. Its dissipation
estimate therefore has a common positive rate. The bounded initial
free-energy gap and uniform strong convexity also give a common
state-error prefactor; no new boundary argument is required.

\begin{corollary}[Extreme-temperature limits of the stable equilibria]
\label{cor:beta-stability-limits}
Assume \Cref{ass:symmetry} and keep $\rho$, $C$, and $q$ fixed.

\begin{enumerate}[label=(\roman*)]
\item Suppose that, for every sufficiently small $\beta>0$,
\Cref{ass:regularity} holds with constant $K_\beta$ satisfying
\eqref{eq:uniform-curvature-margin}. Put
\[
u:=\frac1d\one,
\qquad
M:=\frac{\osc_{\D}C}{\rho}.
\]
For all sufficiently small $\beta$, let $\eta_\beta^\infty$ be the
equilibrium from \Cref{thm:convergence}. Then
\[
\KL(\eta_\beta^\infty\Vert u)
\le2\beta M,
\qquad
\|\eta_\beta^\infty-u\|_1
\le2\sqrt{\beta M}.
\]
Moreover, the exponential decay rate in
\Cref{thm:convergence} can be chosen bounded away from zero uniformly
for all sufficiently small $\beta$.

The limiting reference dynamics has the unique equilibrium $u$ and
satisfies
\[
\KL(\eta^q(t;\eta_0)\Vert u)
\le
e^{-dq_{\min}t}
\KL(\eta_0\Vert u),
\qquad
\eta_0\in\D.
\]

\item Suppose that, for every sufficiently large $\beta$,
\Cref{ass:regularity} holds with constant $K_\beta$ satisfying
$2\beta K_\beta<1$. Let $\eta_\beta^\infty$ be the corresponding
equilibrium, set
\[
Q:=(d-1)q_{\max},
\qquad
\mathcal M:=\operatorname*{argmin}_{\D}C.
\]
Then
\[
0\le
C(\eta_\beta^\infty)-C_{\min}
\le
\frac{Q+\frac{\rho}{2}\log d}{\beta}.
\]
Consequently,
\[
\operatorname{dist}(\eta_\beta^\infty,\mathcal M)
\longrightarrow0.
\]
In particular, if $C$ has a unique minimizer $\bar\eta$, then
$\eta_\beta^\infty\to\bar\eta$.
\end{enumerate}
\end{corollary}

\begin{remark}[Uniformity of the large-$\beta$ decay rate]
The large-$\beta$ conclusion in
\Cref{cor:beta-stability-limits} concerns the equilibria and does not
assert that the exponential rate in \Cref{thm:convergence} is uniform
as $\beta\to\infty$. Indeed, the hypothesis
$2\beta K_\beta<1$ already forces $K_\beta=O(\beta^{-1})$, and the
constants in the proof also depend on
$\beta\osc_{\D}V_\beta$. A sufficient additional condition for a
uniform large-$\beta$ rate is
\[
\inf_{\beta\ge\beta_0}(1-2\beta K_\beta)>0,
\qquad
\sup_{\beta\ge\beta_0}
\beta\osc_{\D}V_\beta<\infty.
\]
These bounds are not consequences of the present assumptions.
\end{remark}

\begin{remark}[Standalone entropy-production mechanism] 
The key estimate in the proof is the uniform entropy-production inequality \eqref{eq:entropy-production}. It says that the instantaneous free-energy dissipation controls the current free-energy gap, and together with the dissipation identity \eqref{eq:dissipation-identity} directly yields exponential decay. The proof of \eqref{eq:entropy-production} does not use the HJB equation,
the optimal-control origin of $V$, or the uniform optimal-rate bounds
from \Cref{thm:gibbs}. More generally, the same entropy-production
argument applies when $V$ is replaced by any prescribed potential
continuous on $\D$, globally $K$-semiconvex, and locally of class
$\mathcal{C}^{1,1}$ on $\Dint$, with positive symmetric reference rates
satisfying the stated bounds and $2\beta K<1$. In particular, convex
potentials correspond to $K=0$, so this restriction is automatic for
every $\beta>0$. The same decay estimates then hold along any global
interior solution of the associated Gibbs dynamics. Global
well-posedness for a separately prescribed potential must be justified
separately; for the planner value considered here, it is supplied by
\Cref{thm:gibbs}.
\end{remark}

\subsection{Main result 4: Application to potential mean field games}
For this result only, we specialize to a population-independent reference generator:
\begin{equation}\label{eq:constant-reference}
q_{ij}(\eta)\equiv \bar q_{ij}>0,
\qquad i\ne j.
\end{equation}
No symmetry assumption is needed for the correspondence.
We retain $a^*(\eta)$ for the planner's feedback map, use $\alpha$ for
individual time-dependent controls, and write $\widehat\alpha(z)$ for
the individual Gibbs minimizer as a function of a value vector.

Fix a deterministic population flow $m:[0,\infty)\to\D$. A representative agent in state $X_s\in[d]$ chooses a locally integrable time-dependent transition-rate generator $\alpha(s)\in A$. Starting at time $t$ from state $i$, its discounted cost is
\[
J_i^{\mathrm{MFG}}(t;m,\alpha)
:=
\mathbb E_{t,i}^{\alpha}
\left[
\int_t^\infty e^{-\rho(s-t)}
\left(
c(X_s,m(s))+\frac1\beta\KL(\alpha_{X_s}(s)\Vert \bar q_{X_s})
\right)\dd s
\right].
\]
We impose the potential condition
\begin{equation}\label{eq:potential-condition}
c(i,\eta)=\partial_iC(\eta),
\qquad i\in[d].
\end{equation}
The derivatives are understood through a fixed $\mathcal{C}^{1,1}$ extension of $C$ to a neighborhood of the affine simplex; intrinsically, only their differences affect the optimal rates.

For $z\in\R^d$ and $i\in[d]$, define the individual row Hamiltonian
\begin{equation}\label{eq:MFG-row-Hamiltonian}
H_i^{\mathrm{MFG}}(z)
:=
\inf_{\alpha_i\in[0,\infty)^{d-1}}
\left\{
\sum_{j\ne i}\alpha_{ij}(z_j-z_i)
+\frac1\beta\KL(\alpha_i\Vert \bar q_i)
\right\},
\end{equation}
where $\alpha_i=(\alpha_{ij})_{j\ne i}$ and $\bar q_i=(\bar q_{ij})_{j\ne i}$. The minimization separates over outgoing edges, and strict convexity gives the unique minimizer
\begin{equation}\label{eq:MFG-row-minimizer}
\widehat\alpha_{ij}(z)=\bar q_{ij}e^{-\beta(z_j-z_i)},
\qquad j\ne i,
\end{equation}
with
\begin{equation}\label{eq:MFG-row-Hamiltonian-explicit}
H_i^{\mathrm{MFG}}(z)
=
\frac1\beta\sum_{j\ne i}\bar q_{ij}\left(1-e^{-\beta(z_j-z_i)}\right).
\end{equation}
The row Hamiltonian depends only on component differences, so
$H_i^{\mathrm{MFG}}(z+\gamma\one)=H_i^{\mathrm{MFG}}(z)$ for every scalar $\gamma$.

Under \eqref{eq:constant-reference}, the planner HJB equation \eqref{eq:hjb} is equivalently
\begin{equation}\label{eq:planner-HJB-potential-MFG}
\rho V(\eta)
=
C(\eta)+\sum_{i=1}^d\eta_iH_i^{\mathrm{MFG}}(DV(\eta)),
\qquad \eta\in\Dint.
\end{equation}

For a given flow $m$, let $v^m(t)=(v_i^m(t))_{i=1}^d$ denote the representative-agent value vector. Its discounted Bellman system is
\begin{equation}\label{eq:MFG-Bellman}
\rho v_i^m(t)-\dot v_i^m(t)
=
c(i,m(t))+H_i^{\mathrm{MFG}}(v^m(t)),
\qquad i\in[d].
\end{equation}
The corresponding unique minimizing feedback is $\widehat\alpha(v^m(t))$, with diagonal entries chosen so that each row sums to zero.

{\color{black}
\begin{definition}[Discounted finite-state MFG equilibrium]
Fix $\eta_0\in\D$. A triple $(m,v,\alpha)$ is a discounted finite-state mean field game equilibrium if:
\begin{enumerate}[label=(\roman*)]
\item \emph{Optimality.} For the exogenous population flow $m$, the bounded locally absolutely continuous value vector $v$ solves \eqref{eq:MFG-Bellman}, and $\alpha(t)=\widehat\alpha(v(t))$ almost everywhere, where \eqref{eq:MFG-row-minimizer} defines the unique Gibbs minimizer.
\item \emph{Consistency.} The population starts from $m(0)=\eta_0$ and satisfies
\begin{equation}\label{eq:MFG-consistency}
\dot m(t)=m(t)\alpha(t).
\end{equation}
\end{enumerate}
Thus the flow used by the representative agent to compute a best response coincides with the flow generated when the entire population uses that best response.
\end{definition}
}

\begin{theorem}[Planner-to-MFG correspondence]\label{thm:planner-to-MFG}
Let \Cref{ass:regularity} hold, suppose \eqref{eq:constant-reference} and \eqref{eq:potential-condition} hold, and fix $\eta_0\in\Dint$. Let $\eta^*$ be the optimal closed-loop planner trajectory from \Cref{thm:gibbs}, and let $\alpha^*(t):=a^*(\eta^*(t))$ be the time-dependent rate matrix obtained by evaluating the planner feedback \eqref{eq:gibbs-feedback} along this trajectory. Then there exists a unique bounded representative-agent value vector $v^*:[0,\infty)\to\R^d$ for the exogenous flow $m=\eta^*$ such that
\begin{equation}\label{eq:value-gradient-differences}
v_j^*(t)-v_i^*(t)
=
\partial_jV(\eta^*(t))-\partial_iV(\eta^*(t)),
\qquad i,j\in[d].
\end{equation}
Consequently, the representative agent's unique Gibbs best-response feedback coincides with the planner feedback, and the planner forward equation is exactly the MFG consistency equation. Hence $(\eta^*,v^*,\alpha^*)$ is a discounted finite-state MFG equilibrium. In particular, the best-response feedback to the fixed population flow $\eta^*$ is unique up to equality almost everywhere in time.
\end{theorem}

\begin{remark}[Scope of the correspondence]
\Cref{thm:planner-to-MFG} proves that the optimal planner pair induces an MFG equilibrium and that the representative agent's best-response feedback to the induced population flow is unique. It does not, by itself, exclude the existence of other MFG fixed points that are not generated by this planner trajectory. A claim of global uniqueness of the MFG equilibrium would require an additional monotonicity, contraction, or planner-equivalence argument.
\end{remark}

\begin{remark}[Why constant reference rates are used]
The constant-reference assumption \eqref{eq:constant-reference} makes the individual row Hamiltonians independent of the population variable. A population-weighted row decomposition remains available when the reference rates depend on the population, but differentiating the planner Hamiltonian then produces additional terms involving derivatives of those rates. These terms are absent from the representative agent's Bellman equation, where the population flow is treated as exogenous. Thus the uncorrected potential condition \eqref{eq:potential-condition} is sufficient for the direct correspondence in the constant-reference case. For state-dependent reference rates, an appropriate correction to the individual cost would generally be required.
\end{remark}

\medskip
\noindent\emph{Sketch of proof.}
Let $p(t)=D_{\T}V(\eta^*(t))$ and regard
$\alpha^*(t)=a^*(\eta^*(t))$ as a fixed open-loop control.
If $V$ were of class $\mathcal{C}^{2}$, differentiating
\eqref{eq:planner-HJB-potential-MFG} in a tangent direction $h$,
using the Gibbs minimizer, and applying the chain rule along the
optimal trajectory would give
\[
\rho\,p(t)\cdot h
=\sum_i h_i\bigl[c(i,\eta^*(t))+H_i^{\mathrm{MFG}}(p(t))\bigr]
+D_{\T}^{2}V(\eta^*(t))[\dot\eta^*(t),h],
\]
Here symmetry of the tangential Hessian and the chain rule give
$D_{\T}^{2}V(\eta^*(t))[\dot\eta^*(t),h]=\dot p(t)\cdot h$.
This yields the planner's current-value Pontryagin adjoint equation
\eqref{eq:p-equation-with-kappa}. Since the available regularity is
$\mathcal{C}^{1,1}$, the detailed proof obtains that equation instead
from a first-order DPP sensitivity identity, without evaluating a
Hessian along the trajectory. The adjoint equation and the individual
Bellman equation \eqref{eq:MFG-Bellman} differ only by a scalar term
common to all states, which may depend on time. This motivates a
representative-agent value of the form $v^*=p+\gamma\one$; a bounded
scalar correction removes the common term and preserves every Gibbs
rate. Finally, Hamiltonian minimization and discounted Dynkin's
formula verify individual optimality. The planner forward equation
then supplies consistency. The full proof is given in
\Cref{app:potential-mfg}.

\section{Single-asset market making
as a potential mean field game}\label{sec:market-making}

We apply our framework to a population of dealers trading one asset.
An exact relation between spread revenue and entropy determines the
reference rates and the adjustment to the state cost.
Our convex inventory-crowding model satisfies the assumptions for
regularity and the potential-MFG correspondence.
With symmetric reference rates, the inventory distribution converges
exponentially to equilibrium.

\subsection{Inventory states, quotes, and the finite-dealer model}
Consider $N$ dealers trading one asset with external customers. Let
$\mathcal Z=\{z_1,\ldots,z_d\}$ be distinct admissible inventory levels, including negative levels if short positions are permitted. Dealer $n$ occupies state $I_t^n\in[d]$, holds $Q_t^n=z_{I_t^n}$ units, and observes the empirical distribution
\[
\eta_i^N(t)=\frac1N\sum_{n=1}^N\mathbf 1_{\{I_t^n=i\}}.
\]
Thus $d$ counts inventory positions, not assets. For a transition $i\to j$, the transaction size is $s_{ij}=|z_j-z_i|$. A dealer in state $i$ quotes a per-unit offset $\delta_{ij}^n$: it buys at $S_t-\delta_{ij}^n$ when $z_j>z_i$, and sells at $S_t+\delta_{ij}^n$ when $z_j<z_i$. Each execution earns the marked-to-market spread $D_{ij}^n:=s_{ij}\delta_{ij}^n$, provided the reference price does not jump at execution.

This is an RFQ-style model with exogenously specified request sizes and whole-trade executions, as motivated by \cite{mmBG2021,mmBBG2023}. We assume positive baseline flow for every retained transition. Allowing all sizes joining the finite inventory grid then gives the complete graph required by \Cref{ass:viscosity}. Unit-sized trades alone would instead give a sparse graph. Finite inventory bounds are enforced by offering no transition outside $\mathcal Z$. Requests are not modeled as customers optimally splitting orders or selecting a different size from an executable menu; such substitution would require joint demand modeling. The trade counterparty is outside the dealer population, so executions change only one modeled inventory. Direct interdealer transactions would require joint jumps.

For now allow $\delta_{ij}^n\in\R$, and let
\begin{equation}\label{eq:mm-demand}
a_{ij}^n=A_{ij}\exp(-\beta s_{ij}\delta_{ij}^n),
\qquad A_{ij}>0,\quad \beta>0.
\end{equation}
Quote withdrawal implements $a_{ij}^n=0$ as a limiting action. Positive rates and finite offsets are in one-to-one correspondence. The baseline $A_{ij}$ is population independent in the game specialization. It describes the execution rate at zero spread; it is not a probability. For feedback rates $a^n_{ij}(\boldsymbol i)$, the joint-state generator is
\[
\mathcal L^{a}f(\boldsymbol i)
=\sum_{n=1}^N\sum_{j\ne i_n}
a^n_{i_nj}(\boldsymbol i)
\bigl[f(\boldsymbol i^{n\to j})-f(\boldsymbol i)\bigr].
\]
Predictable intensities can be implemented by thinning independent Poisson random measures, subject to integrability and nonexplosion. Under a common feedback $a(\eta^N)$, the empirical measure jumps by $(e_j-e_i)/N$ at rate $N\eta_i^Na_{ij}(\eta^N)$. Its drift is $\eta^Na(\eta^N)$, giving the formal limiting equation $\dot\eta=\eta a$. Transaction sizes remain discrete in this limit: it is the fraction of dealers changing state, not the number of units in a trade, that becomes infinitesimal.

Dealer $n$ minimizes the discounted cost
\begin{equation}\label{eq:mm-N-objective}
J_n^N=\mathbb E\int_0^\infty e^{-\rho t}
\left[r(I_t^n,\eta_t^N)
-\sum_{j\ne I_t^n}s_{I_t^nj}\delta_{I_t^nj}^n(t)a_{I_t^nj}^n(t)\right]\dd t.
\end{equation}
Here $r$ is a specified running inventory-risk or carrying-cost functional, not a terminal utility. If $S$ is a continuous martingale and the discounted inventory-price stochastic integral is integrable, expected marked-to-market trading gains give this spread-revenue criterion after subtracting the specified inventory costs. Bounded inventories and a constant-volatility Brownian reference price suffice for square integrability of that discounted stochastic integral. Price drift, adverse selection, fees, and endogenous price impact would add terms and are not included. In particular, \eqref{eq:mm-N-objective} is not obtained by simply discounting terminal marked-to-market wealth.

\subsection{Why entropy appears: an exact revenue identity}
The exponential per-unit intensity in \cite{mmAS2008,mmGLFT2013} motivates the form of \eqref{eq:mm-demand}, but its common coefficient is an additional restriction. For a general law $a_{ij}=A_{ij}e^{-k_{ij}\delta_{ij}}$, the entropy weight is $s_{ij}/k_{ij}$. A common weight $1/\beta$ therefore requires $k_{ij}=\beta s_{ij}$: sensitivity is common with respect to the total concession $D_{ij}=s_{ij}\delta_{ij}$. This is not a change of units that removes heterogeneous demand sensitivities. The size-dependent literature cited above does not establish this scaling as a general empirical law; in particular the FX study \cite{mmBBG2023} uses logistic curves with size-consistent shape parameters within tiers.

A possible acceptance model for nonnegative concessions is that size-specific prospective customers arrive at rate $A_{ij}$ and accept $D$ if $B\ge D$, where $B$ has an exponential distribution with parameter $\beta$, common across sizes. Then $A_{ij}\Pr(B\ge D)=A_{ij}e^{-\beta D}$. The substantive assumption is a size-independent distribution of total-dollar concession tolerance. It implies a myopic optimal total charge $1/\beta$, hence a per-unit charge inversely proportional to trade size. This acceptance construction applies to $D\ge0$ and caps rates at $A_{ij}$; it does not justify the unrestricted negative-offset extension, which is a separate reduced-form demand specification.

Define the myopic reference rate by $\bar q_{ij}=A_{ij}/e$. Indeed, $D A_{ij}e^{-\beta D}$ is maximized at $D=1/\beta$, with execution rate $\bar q_{ij}$ and revenue rate $\bar q_{ij}/\beta$. Solving \eqref{eq:mm-demand} for $D$ yields the identity
\begin{equation}\label{eq:mm-entropy-identity}
-s_{ij}\delta_{ij}a_{ij}
=\frac1\beta a_{ij}\log\frac{a_{ij}}{A_{ij}}
=\frac1\beta\ell(a_{ij},\bar q_{ij})-\frac{\bar q_{ij}}\beta.
\end{equation}
The identity extends to $a_{ij}=0$ by continuity. Thus entropy is the exact instantaneous revenue shortfall relative to myopic quoting, not an independently imposed preference for a nominal policy. A dynamically optimal dealer may accept this shortfall to improve future inventory.

Put $b_i=\beta^{-1}\sum_{j\ne i}\bar q_{ij}$ and
\begin{equation}\label{eq:mm-corrected-cost}
c(i,\eta)=r(i,\eta)-b_i.
\end{equation}
The running cost in \eqref{eq:mm-N-objective} becomes
$c(i,\eta)+\beta^{-1}\KL(a_i\Vert\bar q_i)$ exactly. The correction cannot generally be discarded: even constant edge rates can have different row sums, so the myopic earning opportunity depends on the inventory state.

\subsection{Inventory crowding and the potential correspondence}
An individual quadratic inventory penalty captures one's own exposure, but not the additional risk of trying to unwind a position when other dealers hold similar positions. We model this second effect by a specified crowding cost. This is a reduced-form risk charge for shared external liquidity, not a derived market-clearing price-impact model. Constant $\bar q$ leaves customer execution demand fixed; strategic interaction occurs through the inventory cost rather than direct competition among contemporaneous quotes.

Suppose there is a smooth inventory potential $R$ such that $r(i,\eta)=\partial_iR(\eta)$. Then define
\begin{equation}\label{eq:mm-potential}
C(\eta)=R(\eta)-\sum_i b_i\eta_i.
\end{equation}
With the corresponding ambient extensions, $\partial_iC=c(i,\eta)$, precisely \eqref{eq:potential-condition}. Generally $C\ne\sum_i\eta_ic(i,\eta)$: the potential planner is not the minimizer of summed private costs.

Here is a concrete convex specification. Write $z_i^+=\max(z_i,0)$, $z_i^-=\max(-z_i,0)$, and $m_\pm(\eta)=\sum_i\eta_i z_i^\pm$. For $\gamma>0$ and $\Gamma\ge0$, take
\begin{align}
R(\eta)&=\sum_i\eta_i\frac\gamma2z_i^2
+\frac\Gamma2\bigl(m_+(\eta)^2+m_-(\eta)^2\bigr),\label{eq:mm-crowding-potential}\\
r(i,\eta)&=\frac\gamma2z_i^2
+\Gamma\bigl(z_i^+m_+(\eta)+z_i^-m_-(\eta)\bigr).
\label{eq:mm-crowding-cost}
\end{align}
Long positions become more costly when aggregate long inventory increases, and short positions when aggregate short inventory increases. Opposite-side positions are not allowed to cancel this gross crowding measure. These are explicit modeling choices. All costs are smooth in $\eta$ because the inventory levels are fixed, and the Hessian of $R$ is
$\Gamma(z^+(z^+)^\top+z^-(z^-)^\top)$, hence positive semidefinite.

For any positive constant reference generator, $C$ in \eqref{eq:mm-potential} is convex and smooth. Its flux Lagrangian is
\[
L(\eta,w)=C(\eta)+\frac1\beta\sum_{i\ne j}
\left[w_{ij}\log\frac{w_{ij}}{\eta_i\bar q_{ij}}-w_{ij}+\eta_i\bar q_{ij}\right].
\]
Each entropy summand is a convex perspective, with the usual extended-value boundary convention. Consequently \Cref{ass:regularity} holds with $K=0$. This verifies the full joint flux condition, not just convexity of a state cost. The regularity and planner-to-MFG results therefore apply to this financial example for every fixed $\beta>0$ in the unrestricted-rate formulation.

Fix an initial population distribution $\eta_0\in\Dint$,
as required by \Cref{thm:planner-to-MFG}. Along the induced equilibrium flow, the optimal quotes are
\begin{equation}\label{eq:mm-optimal-quotes}
a^*_{ij}(\eta)=\bar q_{ij}e^{-\beta(\partial_jV(\eta)-\partial_iV(\eta))},
\qquad
\delta^*_{ij}(\eta)
=\frac{\beta^{-1}+\partial_jV(\eta)-\partial_iV(\eta)}{s_{ij}}.
\end{equation}
Equivalently, along the flow the gradient difference equals $v_j^*(t)-v_i^*(t)$ by \Cref{thm:planner-to-MFG}. A destination with greater future cost requires greater spread compensation and receives fewer executions. With symmetric $\bar q$, \Cref{thm:convergence} gives exponential convergence of the induced inventory distribution; $2\beta K<1$ is automatic here. Symmetry is a balanced baseline-flow assumption, not a conclusion of exponential demand and not symmetry of the controlled quotes.

\subsection{Admissible quotes, temperature comparisons, and scope}
The unrestricted financial model allows negative offsets, interpreted as sacrificing immediate spread to obtain executions. With a fixed prospective-customer stream and acceptance probabilities, rates must instead satisfy $a_{ij}\le A_{ij}$, equivalently $\delta_{ij}\ge0$. In that constrained model the row optimizer is
\[
a_{ij}^{\mathrm{cap}}=\min\{A_{ij},\bar q_{ij}e^{-\beta(v_j-v_i)}\},
\]
so the unrestricted Gibbs formula cannot be used without a feasibility check.
A useful sufficient check is available for the equilibrium above. Define
$c_-:=\min_{i,\eta}c(i,\eta)$, $c_+:=\max_{i,\eta}c(i,\eta)$ and $\Delta_c=c_+-c_-$. Nonnegativity of entropy gives $v_i^m(t)\ge c_-/\rho$, while choosing the reference rates gives $v_i^m(t)\le c_+/\rho$ for every exogenous flow $m$. Hence
\begin{equation}\label{eq:mm-feasible-quotes}
\beta\Delta_c\le\rho
\quad\Longrightarrow\quad
\beta^{-1}+v_j^*(t)-v_i^*(t)\ge0.
\end{equation}
Under this sufficient condition, the unrestricted equilibrium is feasible and remains optimal in the smaller nonnegative-offset class. Thus the capped acceptance interpretation can be used without changing the equilibrium construction in this parameter regime. Without such a condition, it is the explicitly unrestricted reduced-form model that is covered.

For a particularly transparent specialization choose
$\bar q_{ij}=q_0/(d-1)$, $i\ne j$. The reference is positive and symmetric with common row sum $q_0$, and on the simplex $C_\beta=R-q_0/\beta$. In this financial interpretation, varying $\beta$ changes customers' sensitivity to total transaction concession, while the arrival baselines and inventory-cost
parameters remain fixed. The entropy coefficient is determined by demand rather than independently chosen
as a dealer preference. This additive constant changes the planner value by $-q_0/(\beta\rho)$ and the representative-agent value by the same amount, but changes neither optimal rates nor free-energy gaps. Thus the fixed-$R$ temperature results apply to feedbacks, trajectories, and equilibria of the same inventory-cost model. For raw financial values the additive shift must be restored; they need not share the finite small-$\beta$ value limit of the normalized control problem. With unequal reference row sums, varying $\beta$ at fixed financial $r$ instead changes the state cost nontrivially, so the fixed-$C$ temperature theorems cannot be transferred without further analysis. The numerical experiments in
\Cref{sec:numerical-experiments} illustrate the mathematical mechanisms, not a calibration of customer demand or inventory crowding.

Finally, \eqref{eq:mm-N-objective} specifies the underlying finite-dealer game, whereas \Cref{thm:planner-to-MFG} establishes an equilibrium of its proposed mean field formulation. The drift calculation identifies the candidate population equation; this paper does not prove convergence of finite-player Nash equilibria or an $\varepsilon_N$-Nash property. Such conclusions require additional finite-player approximation estimates. Nor does the correspondence alone establish uniqueness among all MFG fixed points. Within its stated scope, the application gives an exact financial interpretation of entropy and a concrete interacting inventory model to which the existing regularity, equilibrium-construction, and stability results apply.

\section{Numerical experiments}\label{sec:numerical-experiments}

We report two complementary sets of experiments.  The first is a control-derived two-state example in which the value function is obtained by numerically solving the stationary HJB equation before constructing the Gibbs feedback and the closed-loop dynamics.  This tests the full chain
\[
\text{control problem}\longrightarrow V\longrightarrow a^*\longrightarrow\text{closed-loop stability}.
\]
The second set starts from a prescribed convex potential and tests the standalone Gibbs-stability mechanism of \Cref{thm:convergence}, including its behavior across several state-space dimensions.  This distinction reflects the fact that the stability theorem itself does not require the potential to arise from the preceding control problem.

\subsection{A control-derived two-state experiment}

Take $d=2$, write $\eta=(x,1-x)$, and choose
\[
q_{12}=q_{21}=1,\qquad \rho=1,\qquad \beta=2,
\qquad
C(x)=\frac{\kappa}{2}(x-x_{\mathrm{tar}})^2,
\quad \kappa=1.5,\quad x_{\mathrm{tar}}=0.70.
\]
If $v(x):=V(x,1-x)$, then the HJB equation \eqref{eq:hjb} reduces to
\begin{equation}\label{eq:numerical-two-state-hjb}
\rho v(x)
=
C(x)+\frac1\beta\left[
 x\bigl(1-e^{\beta v'(x)}\bigr)
 +(1-x)\bigl(1-e^{-\beta v'(x)}\bigr)
\right],
\qquad 0\le x\le1.
\end{equation}
The Gibbs formula \eqref{eq:gibbs-feedback} and the state equation \eqref{eq:state-equation} become
\begin{equation}\label{eq:numerical-two-state-feedback}
a^*_{12}(x)=e^{\beta v'(x)},\qquad
a^*_{21}(x)=e^{-\beta v'(x)},\qquad
\dot x=(1-x)a^*_{21}(x)-x a^*_{12}(x).
\end{equation}

We solved \eqref{eq:numerical-two-state-hjb} on a uniform grid with $401$ points by nonlinear finite-difference collocation.  Centered differences were used in the interior and second-order one-sided differences at the endpoints; no artificial Dirichlet values were imposed.  The maximum HJB residual evaluated afterward on a dense grid of $2001$ points was $3.47\times10^{-7}$.  As a refinement check, changing the HJB grid from $401$ to $801$ points changed the interpolated value function by at most $9.65\times10^{-9}$ on the $401$-point grid.  The computed value was convex numerically, with the minimum sampled second derivative equal to approximately $0.253$.

The feedback in \eqref{eq:numerical-two-state-feedback} was then evaluated from the numerical derivative of $v$ and used to integrate the optimal closed-loop equation from $x(0)=0.08$.  The trajectory converged to
\[
x^\infty\approx0.5663166.
\]
Independently minimizing the free energy \eqref{eq:free-energy} and solving $\dot x=0$ in \eqref{eq:numerical-two-state-feedback} gave equilibria differing by only $8.04\times10^{-9}$.  The computed inward-pointing limiting drifts were positive as
$x\downarrow0$ and negative as $x\uparrow1$.

\Cref{fig:control-derived-loop} displays the numerical HJB value, the resulting Gibbs rates, the optimal closed-loop trajectory, and the normalized free-energy gap.  Over the numerically resolved regime, the latter is fitted extremely well by an exponential with empirical rate $5.10$; the fitted rate is a diagnostic for this example and is not asserted to be the constant $\lambda$ in \eqref{eq:free-energy-decay}.
\begin{figure}[t]
\centering
\includegraphics[width=\textwidth]{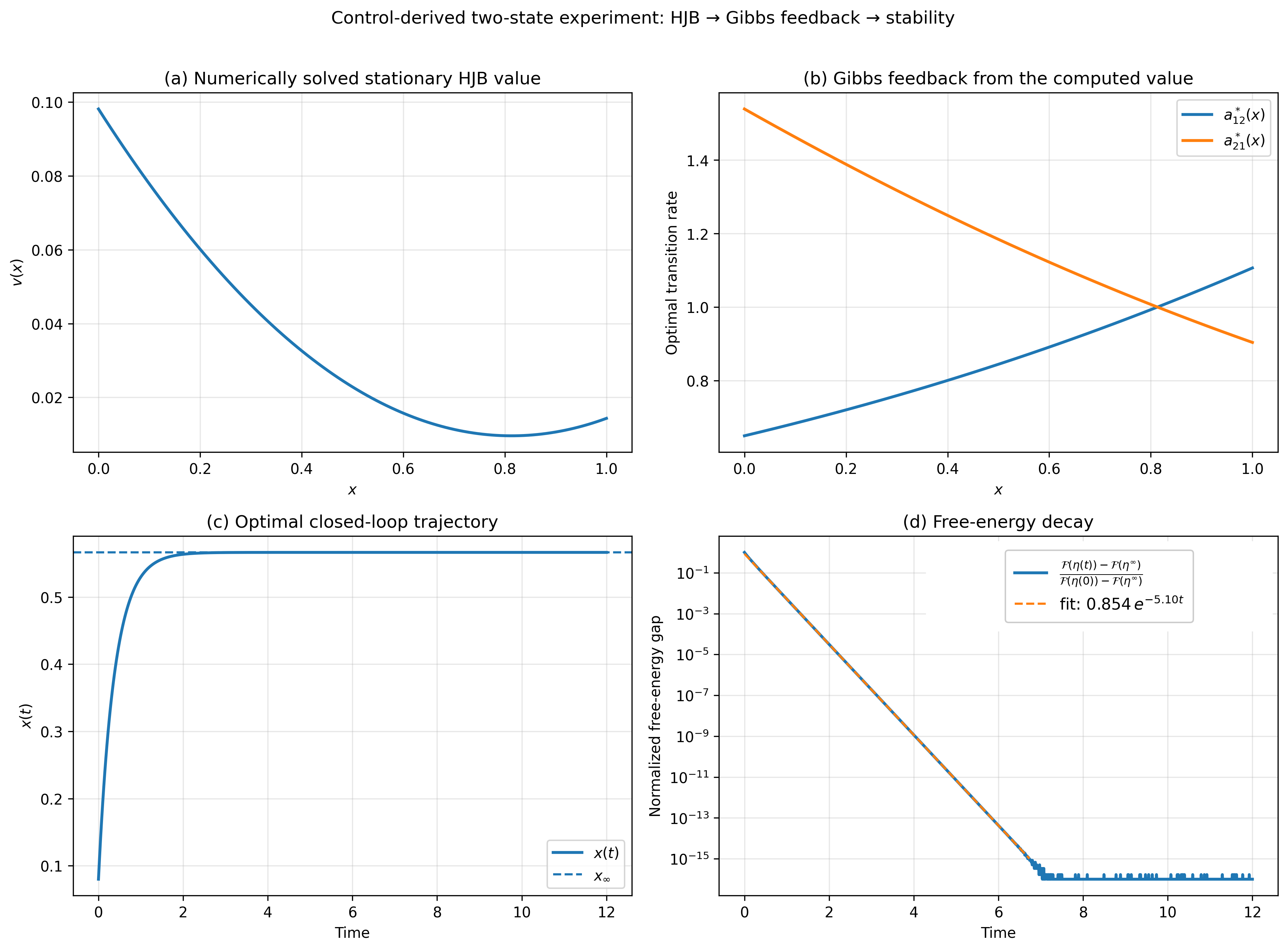}
\caption{Control-derived two-state experiment with $q_{12}=q_{21}=1$, $\rho=1$, $\beta=2$, $\kappa=1.5$, and $x_{\mathrm{tar}}=0.70$. \textbf{(a)} Numerical solution $v$ of the stationary HJB equation. \textbf{(b)} Gibbs rates obtained from the computed derivative $v'$. \textbf{(c)} Optimal closed-loop trajectory and its equilibrium. \textbf{(d)} Normalized free-energy gap on a logarithmic scale, together with an empirical exponential fit proportional to $e^{-5.10t}$.}
\label{fig:control-derived-loop}
\end{figure}

We also checked the control interpretation independently of the plotted convergence.  Along the computed feedback trajectory, the finite-horizon verification identity gives
\[
v(x_0)
=
\int_0^T e^{-\rho t}f\bigl(\eta(t),a^*(\eta(t))\bigr)\dd t
+e^{-\rho T}v(x(T)).
\]
For $T=12$, the HJB value was $v(x_0)=0.0817250593$, while the numerically integrated right-hand side was $0.0817250503$, a relative discrepancy of $1.10\times10^{-7}$.  Thus the value obtained from the HJB and the cost generated by its Gibbs feedback agree to numerical accuracy.

Finally, we evaluated both sides of the dissipation identity \eqref{eq:dissipation-identity} directly along this same control-derived trajectory.  The derivative was computed from the vector field rather than by finite differences in time.  As shown in \Cref{fig:control-derived-dissipation}, the curves for $-\frac{\dd}{\dd t}\mathcal F(\eta(t))$ and $\mathcal D(\eta(t))$ are visually indistinguishable; their relative discrepancy was $2.11\times10^{-16}$, at floating-point precision.  The largest observed increase of the sampled free energy was $2.22\times10^{-16}$.

\begin{figure}[t]
\centering
\includegraphics[width=0.82\textwidth]{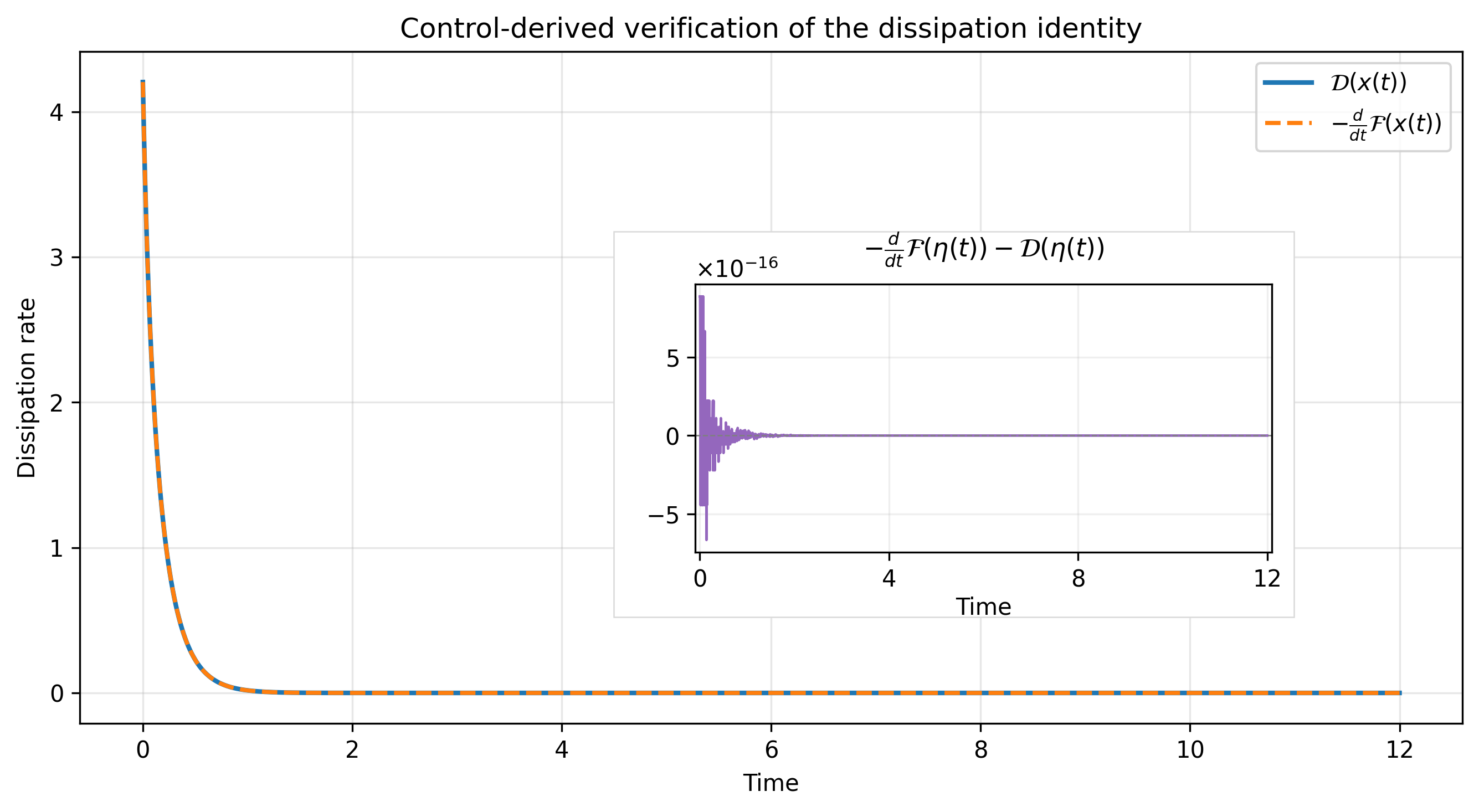}
\caption{Direct verification of the dissipation identity for the control-derived two-state experiment.  The dissipation $\mathcal D(\eta(t))$ and the independently evaluated quantity $-\frac{\dd}{\dd t}\mathcal F(\eta(t))$ overlap to floating-point precision; the inset shows their signed difference, $-\frac{\dd}{\dd t}\mathcal F(\eta(t))-\mathcal D(\eta(t))$.}
\label{fig:control-derived-dissipation}
\end{figure}

\subsection{Standalone Gibbs stability in three states}

We next test the stability mechanism of \Cref{thm:convergence} without requiring the potential to come from an HJB equation.  We take $d=3$, $\beta=2$, and prescribe the convex quadratic potential and symmetric reference rates
\begin{equation}\label{eq:numerical-three-state-data}
V(\eta)=\frac{\kappa}{2}\|\eta-\eta^{\mathrm{tar}}\|_2^2,
\qquad
\eta^{\mathrm{tar}}=(0.60,0.25,0.15),
\qquad
\kappa=1.5,
\qquad
q_{ij}=1\quad(i\ne j).
\end{equation}
For the data in \eqref{eq:numerical-three-state-data}, the Gibbs feedback is given by \eqref{eq:gibbs-feedback}, and the trajectory is obtained from \eqref{eq:state-equation}.  The free energy and dissipation are evaluated using \eqref{eq:free-energy} and \eqref{eq:dissipation}, respectively.  Numerical minimization of the free energy over the simplex gives
\[
\eta^\infty\approx(0.515334,0.271898,0.212768),
\qquad
\mathcal F(\eta^\infty)\approx-0.990243.
\]
The equilibrium does not coincide with $\eta^{\mathrm{tar}}$: the quadratic potential favors the target, whereas the entropy term favors a more diffuse distribution.

We integrated the closed-loop equation with an explicit eighth-order Runge--Kutta method (DOP853), using relative tolerance $10^{-10}$ and absolute tolerance $10^{-12}$.  The distinguished trajectory starts from $\eta(0)=(0.05,0.20,0.75)$.  To test convergence beyond a single trajectory, we also simulated twelve additional interior initial conditions drawn from the uniform Dirichlet distribution with fixed random seed $2026$.  All thirteen trajectories converged to the same interior equilibrium; the largest final Euclidean distance to $\eta^\infty$ was approximately $2.77\times10^{-9}$.

\Cref{fig:numerical-stability} summarizes the experiment.  The free-energy gap is nearly linear on a logarithmic scale after a short transient, and a least-squares fit over the numerically resolved regime gives an empirical decay proportional to $e^{-17.51t}$.  Panel (c) evaluates the two sides of \eqref{eq:dissipation-identity} directly from the vector field; their relative discrepancy is $2.12\times10^{-16}$.  Panel (d) shows the decay of the chemical-potential imbalance $\max_i\psi_i-\min_i\psi_i$, illustrating the approach to the zero-dissipation condition.

\begin{figure}[t]
\centering
\includegraphics[width=\textwidth]{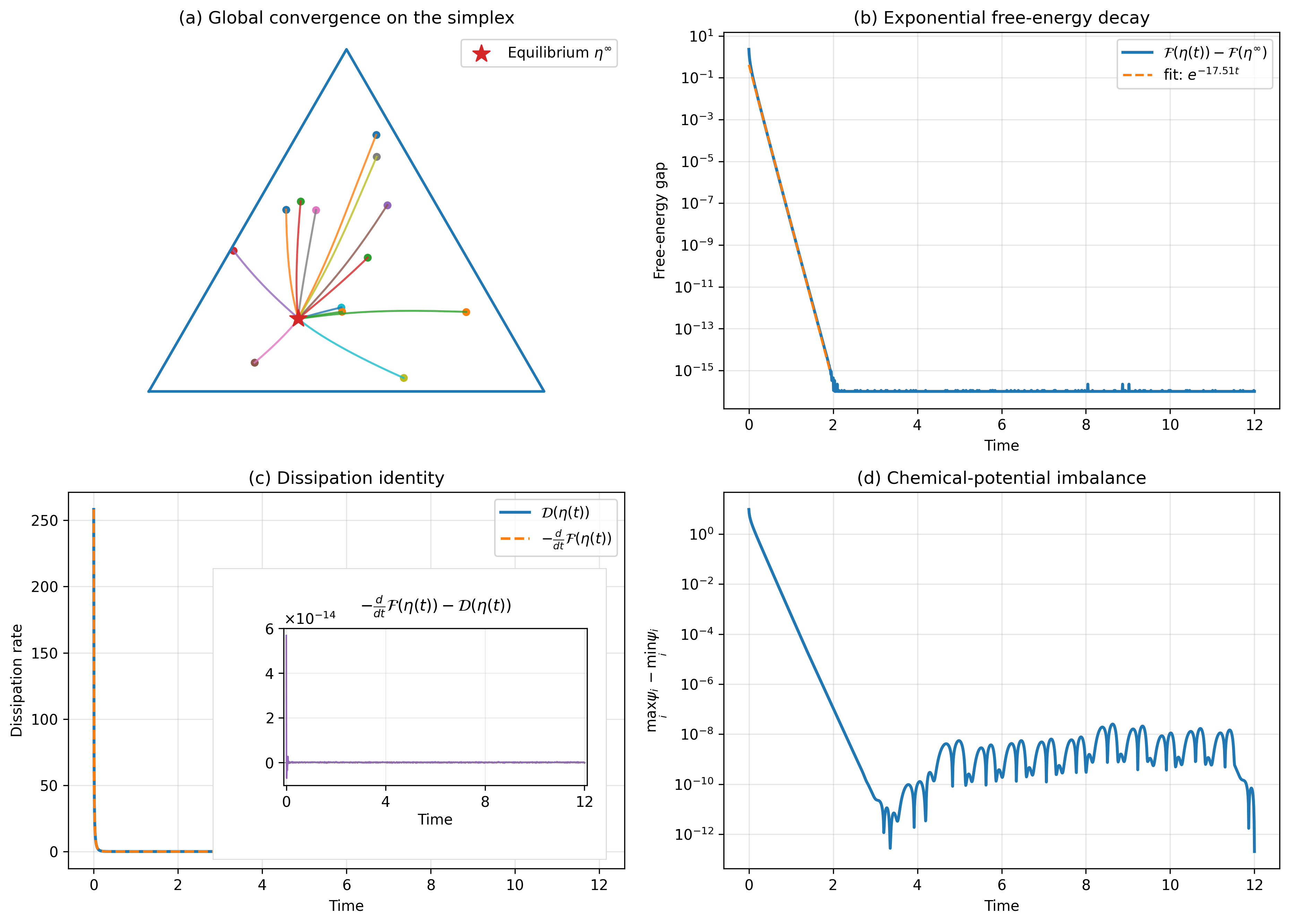}
\caption{Numerical stability of the three-state closed-loop Gibbs dynamics for the data in \eqref{eq:numerical-three-state-data}. \textbf{(a)} Trajectories from thirteen interior initial conditions on the probability simplex; all converge to the same equilibrium $\eta^\infty$. \textbf{(b)} Free-energy gap on a logarithmic scale, together with an empirical fit proportional to $e^{-17.51t}$. \textbf{(c)} Direct numerical verification of the dissipation identity \eqref{eq:dissipation-identity}, with an inset showing the signed difference between the negative free-energy derivative and the dissipation. \textbf{(d)} Decay of the chemical-potential imbalance $\max_i\psi_i-\min_i\psi_i$.}
\label{fig:numerical-stability}
\end{figure}

For a componentwise view of zero dissipation, set $\overline\psi(t)=\frac13\sum_{i=1}^3\psi_i(\eta(t))$. \Cref{fig:chemical-potential-equalization} plots $\psi_i(\eta(t))-\overline\psi(t)$ for $i=1,2,3$.  All three centered chemical potentials converge rapidly to zero, consistent with equality of the marginal free-energy values at the minimizer.

\begin{figure}[t]
\centering
\includegraphics[width=0.82\textwidth]{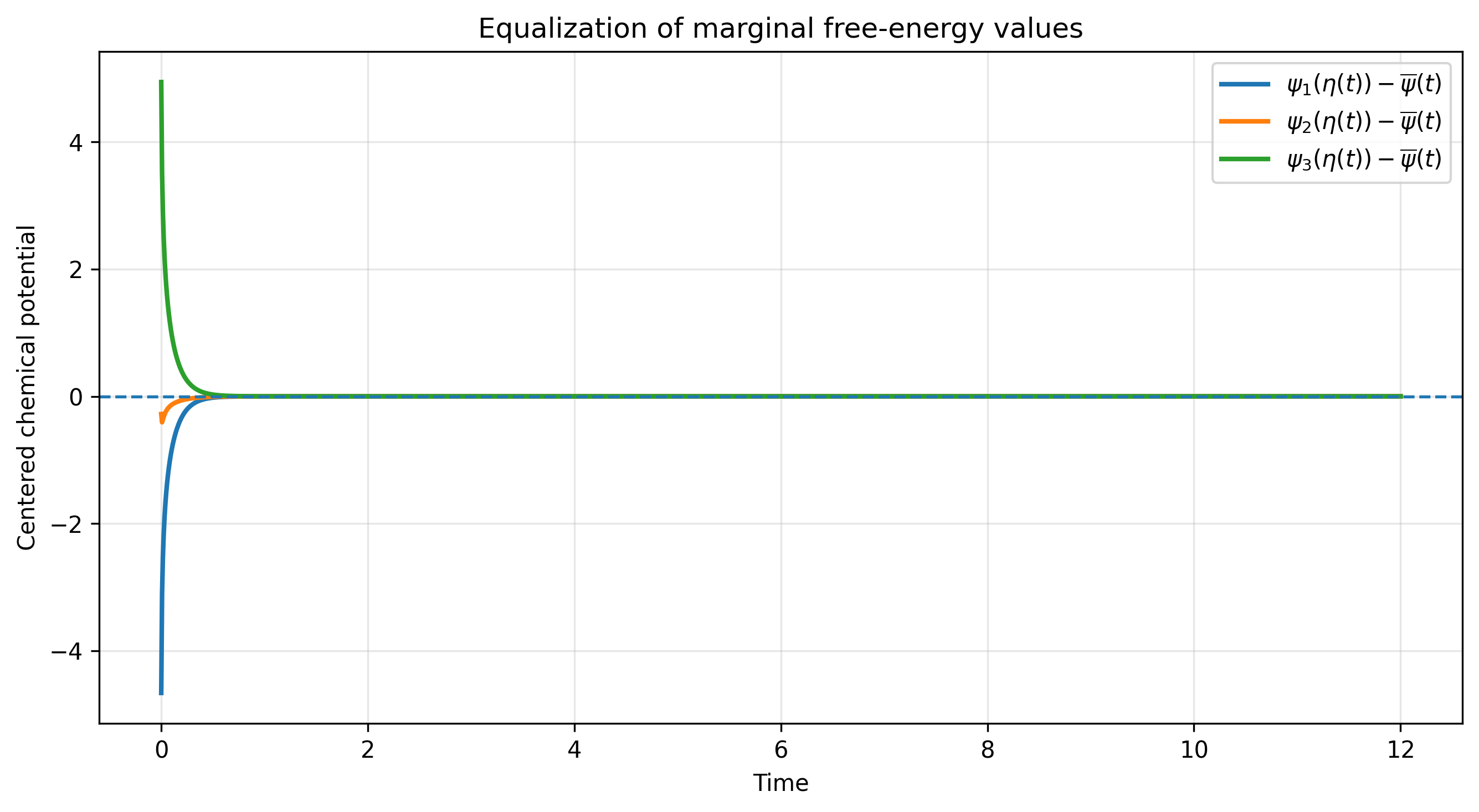}
\caption{Equalization of the marginal free-energy values along the distinguished three-state trajectory.  Each curve is a centered chemical potential $\psi_i(\eta(t))-\overline\psi(t)$; all three converge to zero.}
\label{fig:chemical-potential-equalization}
\end{figure}

\subsection{Comparison across dimensions}

This experiment concerns only the standalone stability result in \Cref{thm:convergence}; it is not derived from the control problem or the potential-MFG correspondence. We finally repeat the standalone stability experiment for $d\in\{3,5,10,20\}$, using five independently sampled interior initial conditions for each dimension.  To keep the total reference exit intensity fixed as $d$ varies, we set
\[
q_{ij}^{(d)}=\frac1{d-1},\qquad i\ne j,
\]
so that $\sum_{j\ne i}q_{ij}^{(d)}=1$ for every state.  For each dimension we use
\[
V_d(\eta)=\frac{\kappa}{2}\|\eta-\eta^{\mathrm{tar},d}\|_2^2,
\qquad
\eta_i^{\mathrm{tar},d}
=
\frac{\exp\!\left(-\frac{2(i-1)}{d-1}\right)}{\sum_{k=1}^d\exp\!\left(-\frac{2(k-1)}{d-1}\right)},
\]
with the same parameters $\beta=2$ and $\kappa=1.5$.  The five initial conditions in each dimension were sampled from the uniform Dirichlet distribution with fixed random seed $2026$.  Trajectories were integrated on $0\le t\le3$ by DOP853 with relative tolerance $10^{-9}$ and absolute tolerance $10^{-11}$.

To compare shapes of the decay curves rather than their dimension-dependent initial magnitudes, \Cref{fig:dimension-comparison} uses the normalized free-energy gap
\[
\widehat G_d(t)
=
\frac{\mathcal F_d(\eta(t))-\mathcal F_d(\eta_d^\infty)}{\mathcal F_d(\eta(0))-\mathcal F_d(\eta_d^\infty)}
\]
and the normalized chemical-potential imbalance $\widehat\Psi_d(t)=\Psi_d(\eta(t))/\Psi_d(\eta(0))$, where $\Psi_d(\eta)=\max_i\psi_i(\eta)-\min_i\psi_i(\eta)$.  For every fixed dimension the five curves remain tightly grouped and show approximately exponential decay after a short transient.  Least-squares fits over $10^{-8}<\widehat G_d(t)<10^{-2}$ gave mean empirical decay rates
\[
8.85,\qquad 5.47,\qquad 3.57,\qquad 2.71
\]
for $d=3,5,10,20$, respectively, with corresponding standard deviations $0.29$, $0.03$, $0.06$, and $0.02$.  The observed slowdown with dimension is compatible with \Cref{thm:convergence}, whose exponential constant may depend on the model data, including $d$; the experiment does not claim a dimension-uniform rate.

\begin{figure}[t]
\centering
\includegraphics[width=\textwidth]{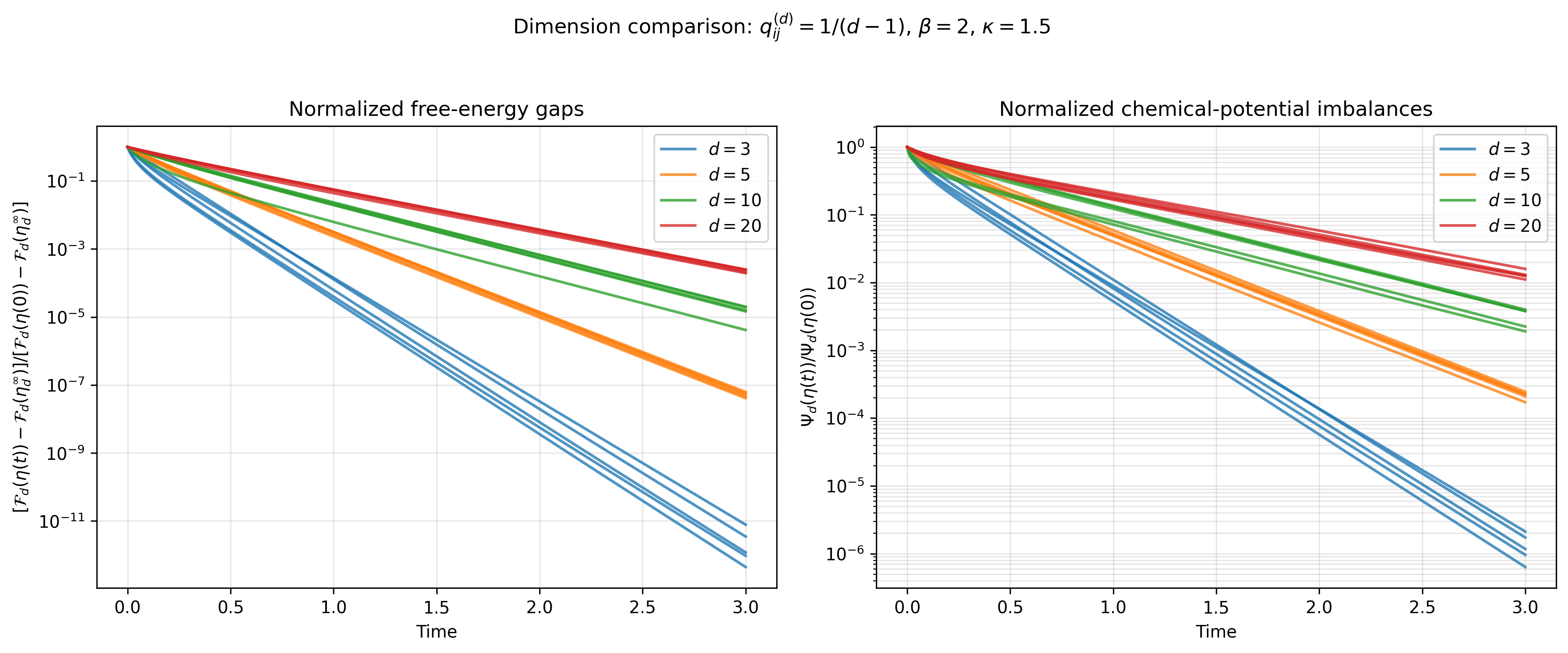}
\caption{Comparison of the closed-loop Gibbs dynamics for $d\in\{3,5,10,20\}$, with five interior initial conditions for each dimension and normalized symmetric rates $q_{ij}^{(d)}=1/(d-1)$. \textbf{(a)} Normalized free-energy gaps $\widehat G_d(t)$. \textbf{(b)} Normalized chemical-potential imbalances $\widehat\Psi_d(t)$.  Curves of the same color correspond to different initial conditions in the same dimension.}
\label{fig:dimension-comparison}
\end{figure}

{\color{black}
The preceding experiments concern a fixed temperature. We now keep
the cost and reference generator fixed while varying $\beta$, in
order to examine the temperature results and their relaxed assumptions.

\subsection{Temperature limits with fixed, state-dependent reference rates}

We retain the two-state running cost, discount, and target from the
first experiment, and compare the constant reference $q_{12}=q_{21}=1$
with \eqref{eq:temperature-state-dependent-rates} at $q_0=1$ and
$\sigma=0.6$. The latter satisfies \Cref{ass:regularity} with
$K_\beta=0.3/\beta$ for every $\beta>0$ by
\Cref{ex:temperature-state-dependent}. Hence $\beta K_\beta=0.3$
does not vanish, while the strong-convexity margin is $0.4$.
The models therefore test both the original constant-reference case
and the genuinely state-dependent extension.

For each model, we solve the stationary HJB at nineteen temperatures
from $0.002$ to $512$: the eight values $0.002\cdot2^k$ for
$k=0,\ldots,7$ and the eleven values $0.5\cdot2^k$ for
$k=0,\ldots,10$.
The state-dependent equation is obtained from
\eqref{eq:numerical-two-state-hjb} by multiplying the entire bracket
by $q(x)$; each rate in \eqref{eq:numerical-two-state-feedback}
is multiplied by the same factor. We use $1601$ grid points,
centered differences in the interior, second-order one-sided endpoint
derivatives, and damped Newton iteration with continuation in $\beta$.
The HJB equation is imposed at the endpoints without artificial
Dirichlet data. Cubic-spline interpolation of the computed values
supplies the derivatives used in the feedback and trajectory solver.
The reference value is obtained independently from its linear
discounted equation and checked by trajectory quadrature.

Across both models and all temperatures, the largest collocation
residual is $2.91\times10^{-12}$, and the largest HJB residual on
a separate $10001$-point grid is $4.83\times10^{-7}$.
Repeating the calculation with $3201$ points changes the interpolated
values by at most $6.89\times10^{-9}$ and the rates by at most
$1.87\times10^{-4}$ over the full temperature sweep; for
$\beta\le0.5$, the largest rate change is $3.87\times10^{-8}$.
An independent integration of the running cost, including the
discounted terminal value, agrees with the HJB value to within
$8.47\times10^{-9}$ at fifteen state--temperature pairs per model.
These checks assess numerical resolution; sampled extrema are not
certified suprema on the continuum.

\Cref{fig:temperature-values} examines \Cref{thm:beta-limits}.
The values decrease with $\beta$ at every HJB grid point. At small
$\beta$, they approach the reference value; at large $\beta$, they
approach $C_{\min}/\rho=0$. The two reference models have similar
numerical behavior, so we display them in separate columns with
identical axes. Within each column, the profiles use the well-separated
inverse temperatures $\beta\in\{0.002,2,32,512\}$, and the error
panel shows both limiting comparisons over the entire sweep.
Departure from the reference value increases as the error relative
to the minimum-cost value decreases. The values remain bounded;
departure here does not mean divergence to infinity.
Throughout these experiments, $\beta$ is inverse temperature, so
$\beta\downarrow0$ is the high-temperature regime.
\begin{figure}[t]
\centering
\includegraphics[width=\textwidth]{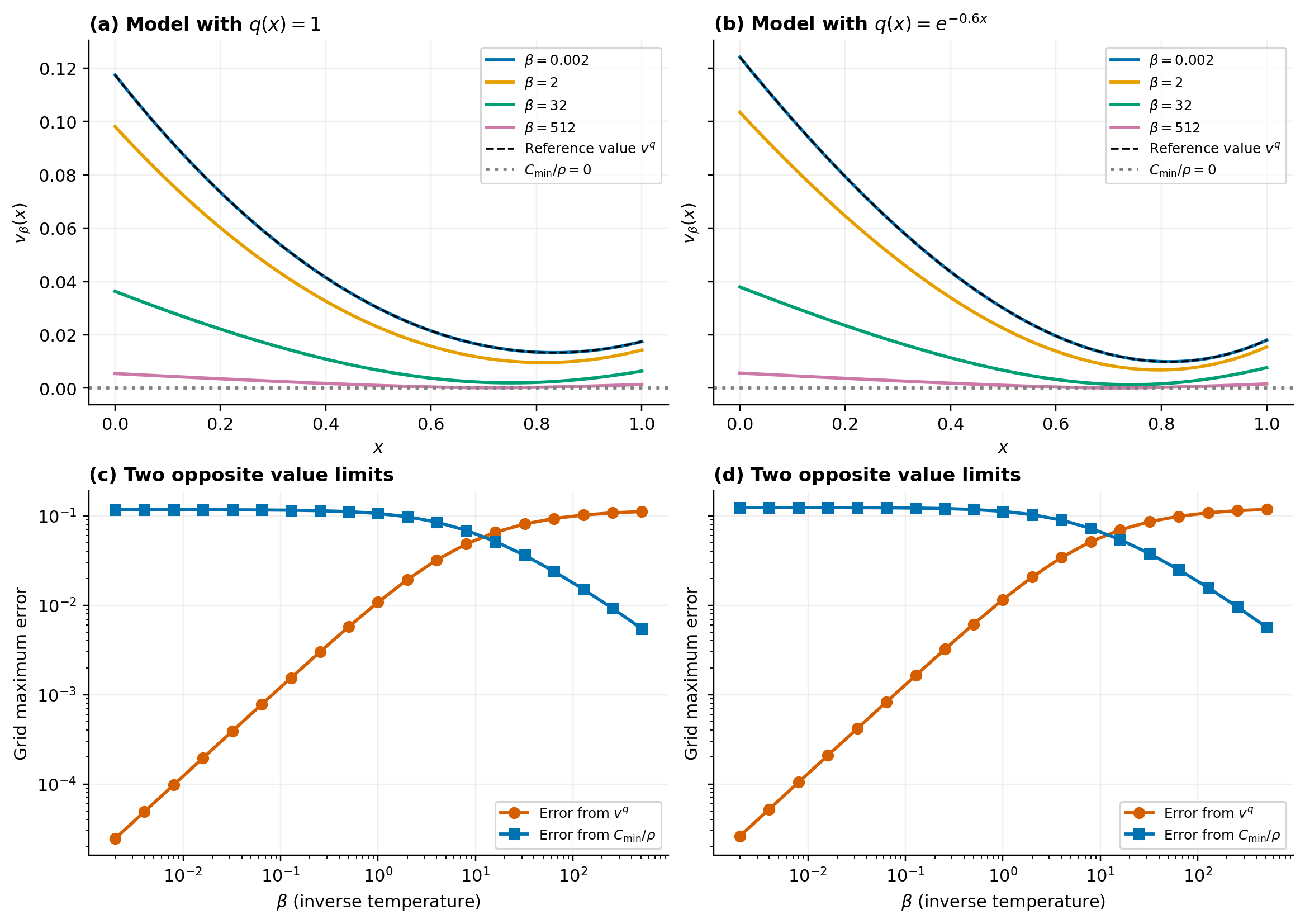}
\color{black}
\caption{\color{black} Value limits, with the constant-reference model on the left
and the state-dependent model on the right.
\textbf{(a)--(b)} Value profiles at four widely separated inverse temperatures,
compared with both limiting values.
\textbf{(c)--(d)} Grid maximum errors relative to the reference value
and to the minimum running cost divided by the discount.}
\label{fig:temperature-values}
\end{figure}

\Cref{fig:temperature-feedback} tests
\Cref{thm:uniform-small-beta-feedback}. The feedback error decreases
uniformly over the sampled states and both directed edges, including
the endpoints. A log--log fit over the five smallest inverse temperatures
has slope approximately $1.001$ for both models. The observed
first-order behavior is faster than the general $O(\sqrt\beta)$
guarantee for bounded $\beta K_\beta$. That guarantee is an upper
bound, not a prediction of a sharp exponent for every state-dependent
model. Neither example demonstrates sharpness of the square-root rate.
The insets show the empirical errors alongside vertically offset slope
guides, which are not error bounds.

The rate profiles and full error sweep also show the opposite behavior
at large $\beta$: feedback departs increasingly from the fixed reference
rates. For these particular models this growth has a direct explanation.
At $x=0$, the HJB gives
$a^\beta_{21}(0)=1+\beta[C(0)-\rho v_\beta(0)]$.
Since \Cref{thm:beta-limits} gives $v_\beta(0)\to0$ and
$C(0)=0.3675$, we obtain $a^\beta_{21}(0)/\beta\to0.3675$.
Thus this endpoint rate diverges linearly as $\beta\to\infty$,
although the values remain bounded. The dashed large-$\beta$ guide
in the error panels shows this endpoint scale. This example-specific
divergence is not a claim of the small-$\beta$ feedback theorem.
\begin{figure}[t]
\centering
\includegraphics[width=\textwidth]{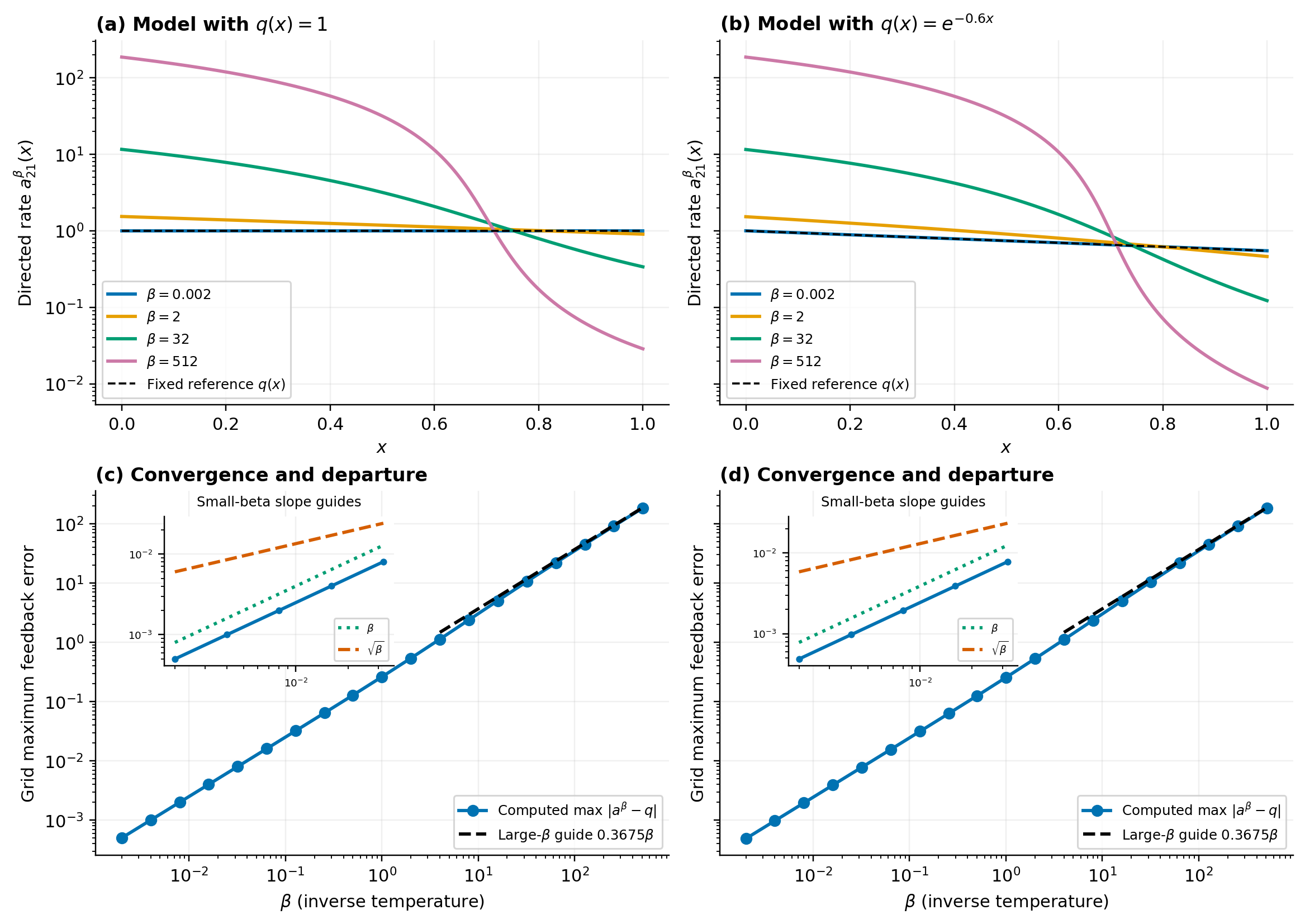}
\color{black}
\caption{\color{black}Gibbs feedback, with one reference model per column.
\textbf{(a)--(b)} The directed rate $a^\beta_{21}(x)$ at four widely
separated inverse temperatures, on a logarithmic rate axis; the fixed
reference rate is dashed.
\textbf{(c)--(d)} Grid maximum feedback error over both directed edges
and the full temperature sweep, with the linear endpoint-growth scale.
Insets isolate the small-$\beta$ regime and its slope comparisons.}
\label{fig:temperature-feedback}
\end{figure}

For \Cref{cor:small-beta-trajectories}, we use the seven initial states
$x_0\in\{0.001,0.08,0.3,0.5,0.8,0.95,0.999\}$.
\Cref{fig:temperature-trajectories} reports the maximum trajectory
error over these starts and sampled times up to $T=4$, together with
discounted state and rate errors integrated to $T=20$ over the full
nineteen-temperature sweep. We supplement the uniform time grid with
geometrically spaced points near zero to resolve the rapid motion at
large $\beta$. Doubling this time sampling at $\beta=512$ changes
the maximum discounted errors by less than $10^{-5}$. The trajectory
panels use the same four inverse temperatures as the value and rate
profiles, with a time axis that is linear below $0.01$ and logarithmic
above it.
The omitted state-error tail is at most $2e^{-20}<4.13\times10^{-9}$.
For the rate-error tail, the global feedback-error bound from
\Cref{thm:uniform-small-beta-feedback}, together with the range of
$q$, gives the bound $2e^{-20}(E_\beta+1-e^{-\sigma})$, where
$E_\beta$ denotes the true uniform feedback error. The proof of that theorem gives
$E_\beta\le q_{\max}(e^{A_\beta}-1)$ with its explicit
scaled-gradient bound $A_\beta$. Using the model parameters yields
a rate-error tail below $1.96\times10^{-8}$ throughout the
small-$\beta$ sweep. Across the full sweep, substituting the grid
maximum feedback error in the same tail expression gives a numerical
tail diagnostic below $7.66\times10^{-7}$; this latter quantity is
not a certified continuum bound. All three plotted errors decrease as
$\beta\downarrow0$. No uniform-in-time trajectory convergence is inferred.
\begin{figure}[t]
\centering
\includegraphics[width=\textwidth]{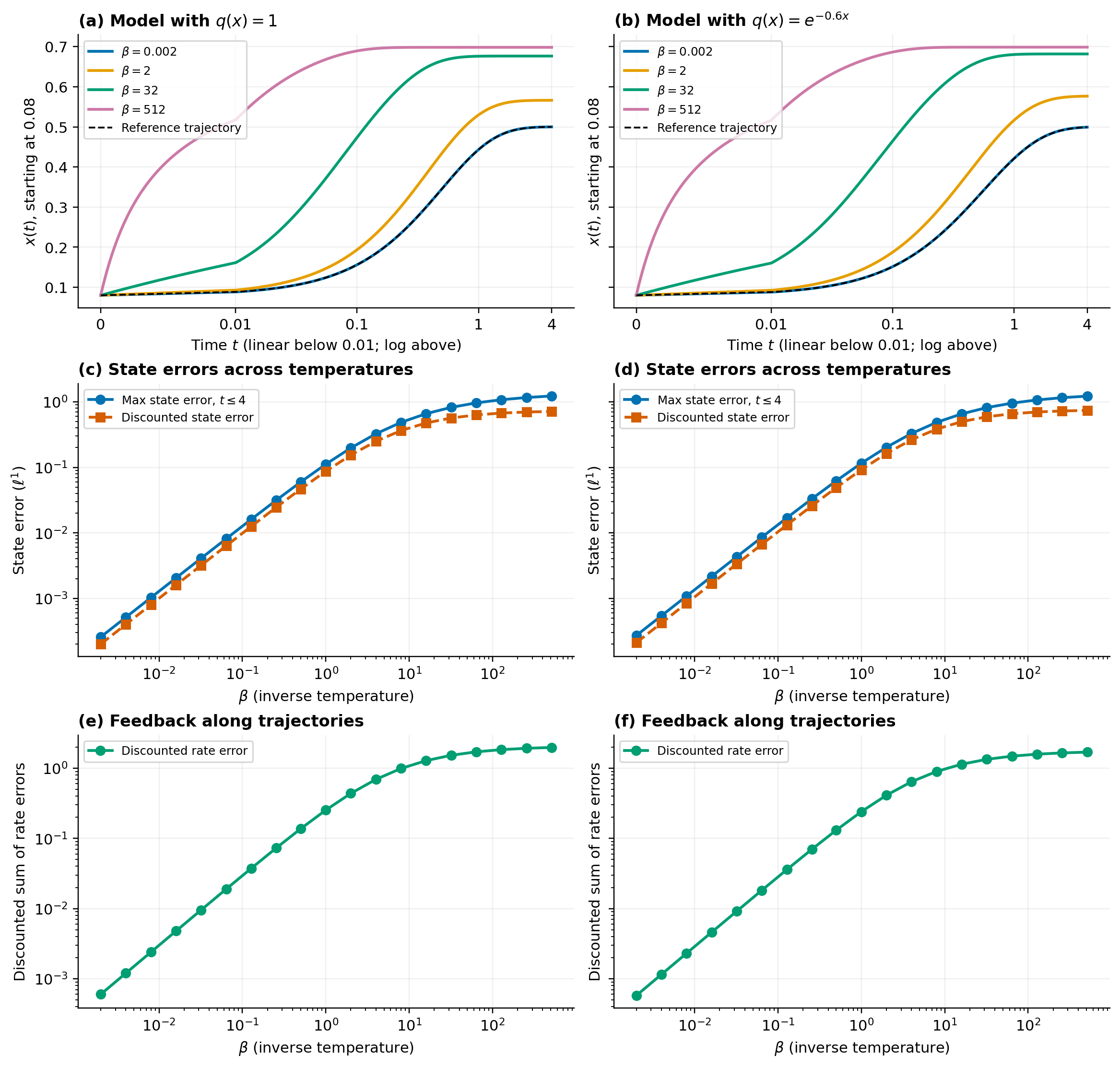}
\color{black}
\caption{\color{black}Trajectory and discounted-feedback limits, with one
reference model per column.
\textbf{(a)--(b)} Trajectories from $x_0=0.08$ at four widely separated
inverse temperatures and the corresponding reference trajectory.
\textbf{(c)--(d)} Maximum sampled state errors up to $T=4$ and
discounted state errors integrated to $T=20$.
\textbf{(e)--(f)} Discounted rate errors up to $T=20$, comparing Gibbs
feedback on its trajectory with reference rates on the reference trajectory.
Error maxima use seven prescribed initial states.}
\label{fig:temperature-trajectories}
\end{figure}

\Cref{fig:temperature-equilibria} tests both parts of
\Cref{cor:beta-stability-limits}. The equilibrium coordinate approaches
$1/2$ for small $\beta$ and the cost minimizer $0.70$ for large
$\beta$. For the state-dependent reference it equals approximately
$0.5001208$ at $\beta=0.002$ and $0.6987528$ at $\beta=512$.
The small-$\beta$ distance and large-$\beta$ equilibrium cost remain
below their respective analytic bounds. Here $M=0.3675$ and $Q=1$
for both models, so the bounds are directly comparable.
\begin{figure}[t]
\centering
\includegraphics[width=\textwidth]{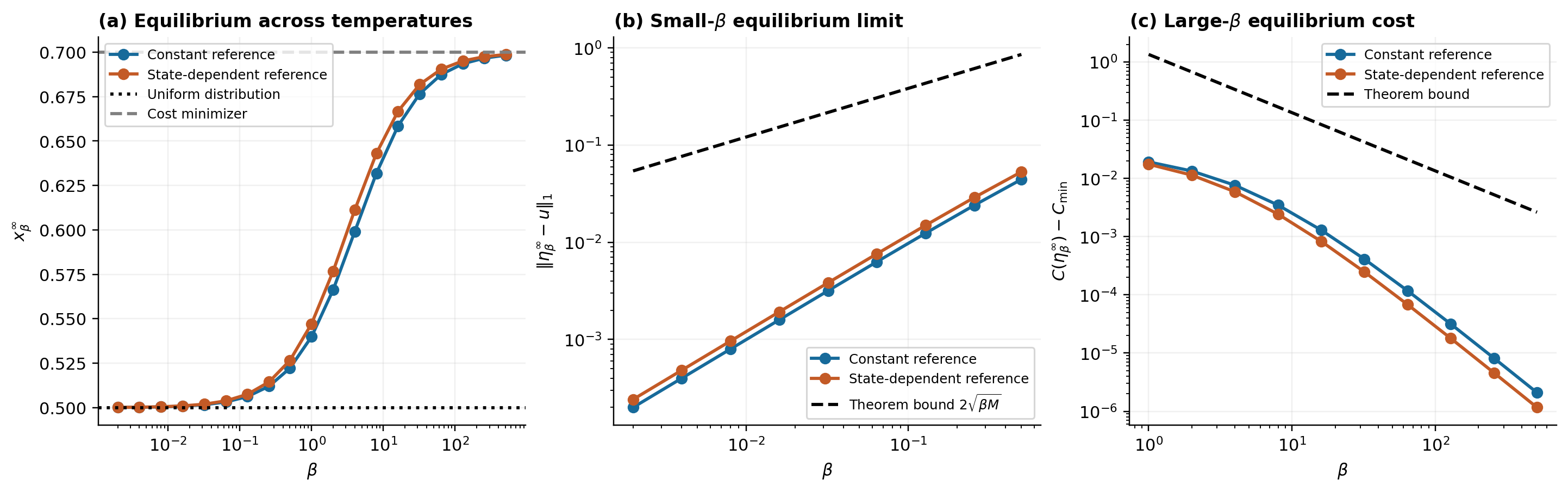}
\color{black}
\caption{\color{black}\textcolor{black}{Extreme-temperature limits of the stable equilibria.
\textbf{(a)} Equilibrium coordinate over the full temperature sweep.
\textbf{(b)} Distance from the uniform distribution and its theorem
bound. \textbf{(c)} Equilibrium running-cost excess and its theorem
bound. These are equilibrium conclusions; no uniform large-$\beta$
decay rate is asserted.}}
\label{fig:temperature-equilibria}
\end{figure}

The accompanying script also retains numerical checks of the common
small-$\beta$ stability envelope and the dissipation identity for the
state-dependent model, without a separate figure.

} 
\color{black}

\section{Viscosity well-posedness}\label{app:viscosity}

\subsection{Viscosity characterization on the closed simplex}

We first establish boundedness of the value function and the dynamic
programming principle. We then prove the viscosity inequalities for
the upper and lower semicontinuous envelopes of the value function on
the closed simplex. Finally, Kraaij's comparison theorem identifies
the two envelopes, yielding continuity and uniqueness.

Since the entropy term is nonnegative and the zero-rate control is
admissible,
\[
-\frac{\|C\|_\infty}{\rho}
\le V(\eta)
\le
\frac1\rho
\left(
\|C\|_\infty+\frac{(d-1)q_{\max}}{\beta}
\right),
\qquad \eta\in\D.
\]
In particular, \Cref{ass:viscosity} implies that $V$ is finite and bounded.

\begin{lemma}[Dynamic programming principle]\label{lem:dpp}
For every $h>0$ and $\eta_0\in\D$,
\begin{equation}\label{eq:dpp}
V(\eta_0)=\inf_{a\in\mathcal A(\eta_0)}\left\{\int_0^h e^{-\rho t}f(\eta(t),a(t))\dd t+e^{-\rho h}V(\eta(h))\right\}.
\end{equation}
\end{lemma}

\begin{proof}
Fix $h>0$. We first record the two closure properties of the admissible
class that will be used below. If $a\in\mathcal A(\eta_0)$, define its
time-shifted tail by
\[
a^{h}(s):=a(h+s),\qquad s\geq0.
\]
Since $a$ is measurable and locally integrable, so is $a^{h}$. Indeed,
for every $T<\infty$,
\[
\int_0^T\sum_{i\neq j}a^h_{ij}(s)\,ds
=
\int_h^{h+T}\sum_{i\neq j}a_{ij}(t)\,dt
<\infty.
\]
If $\eta(\cdot)$ is the trajectory generated by $a$ from $\eta_0$, then
$s\mapsto\eta(h+s)$ is the trajectory generated by $a^h$ from
$\eta(h)$. Hence
\[
a^h\in\mathcal A(\eta(h)).
\]

Similarly, if $a\in\mathcal A(\eta_0)$ and
$\widehat a\in\mathcal A(\eta(h))$, define their concatenation at time
$h$ by
\[
(a\oplus_h\widehat a)(t)
:=
\begin{cases}
a(t), & 0\leq t<h,\\
\widehat a(t-h), & t\geq h.
\end{cases}
\]
This control is measurable and generator-valued. Moreover, for every
$T<\infty$,
\[
\int_0^T\sum_{i\neq j}(a\oplus_h\widehat a)_{ij}(t)\,dt
\leq
\int_0^{\min\{T,h\}}\sum_{i\neq j}a_{ij}(t)\,dt
+
\mathbf 1_{\{T>h\}}
\int_0^{T-h}\sum_{i\neq j}\widehat a_{ij}(s)\,ds
<\infty.
\]
By uniqueness of the linear state equation, its trajectory agrees with
the trajectory generated by $a$ on $[0,h]$ and, after time $h$, with
the trajectory generated by $\widehat a$ from the common state
$\eta(h)$. Thus
$a\oplus_h\widehat a\in\mathcal A(\eta_0)$.

We now prove \eqref{eq:dpp}. Let $a\in\mathcal A(\eta_0)$ and denote
its trajectory by $\eta(\cdot)$. Since $f\ge-\|C\|_\infty$, every discounted cost has finite negative
part, so the following splitting of the extended-real-valued integral
is legitimate. Splitting its cost at time $h$ and
using the change of variables $t=h+s$ in the tail gives
\[
J(\eta_0;a)
=
\int_0^h e^{-\rho t}
f(\eta(t),a(t))\,dt
+
e^{-\rho h}
J(\eta(h);a^h).
\]
Since $a^h\in\mathcal A(\eta(h))$, the definition of the value function
implies
\[
J(\eta(h);a^h)\geq V(\eta(h)).
\]
Therefore
\[
J(\eta_0;a)
\geq
\int_0^h e^{-\rho t}
f(\eta(t),a(t))\,dt
+
e^{-\rho h}V(\eta(h)).
\]
Taking the infimum over $a\in\mathcal A(\eta_0)$ yields
\[
V(\eta_0)
\geq
\inf_{a\in\mathcal A(\eta_0)}
\left\{
\int_0^h e^{-\rho t}
f(\eta(t),a(t))\,dt
+
e^{-\rho h}V(\eta(h))
\right\}.
\]

For the reverse inequality, denote the infimum on the right-hand side
by $R_h(\eta_0)$. The value function is finite: it is bounded below
because the entropy term is nonnegative and $C$ is bounded, while the
zero-rate control is admissible and has finite cost. Hence
$R_h(\eta_0)<\infty$. Fix $\varepsilon>0$ and choose
$a^\varepsilon\in\mathcal A(\eta_0)$, with trajectory
$\eta^\varepsilon$, such that
\[
\int_0^h e^{-\rho t}
f(\eta^\varepsilon(t),a^\varepsilon(t))\,dt
+
e^{-\rho h}V(\eta^\varepsilon(h))
\leq
R_h(\eta_0)+\varepsilon.
\]
By the definition of $V(\eta^\varepsilon(h))$, there exists
$\widehat a^\varepsilon\in
\mathcal A(\eta^\varepsilon(h))$ such that
\[
J(\eta^\varepsilon(h);\widehat a^\varepsilon)
\leq
V(\eta^\varepsilon(h))+\varepsilon.
\]
Concatenate the two controls:
\[
\widetilde a^\varepsilon
:=
a^\varepsilon\oplus_h\widehat a^\varepsilon.
\]
By the closure property proved above,
$\widetilde a^\varepsilon\in\mathcal A(\eta_0)$. Consequently,
\[
\begin{aligned}
V(\eta_0)
&\leq
J(\eta_0;\widetilde a^\varepsilon)\\
&=
\int_0^h e^{-\rho t}
f(\eta^\varepsilon(t),a^\varepsilon(t))\,dt
+
e^{-\rho h}
J(\eta^\varepsilon(h);\widehat a^\varepsilon)\\
&\leq
\int_0^h e^{-\rho t}
f(\eta^\varepsilon(t),a^\varepsilon(t))\,dt
+
e^{-\rho h}V(\eta^\varepsilon(h))
+
e^{-\rho h}\varepsilon\\
&\leq
R_h(\eta_0)+(1+e^{-\rho h})\varepsilon.
\end{aligned}
\]
Letting $\varepsilon\downarrow0$ gives
$V(\eta_0)\leq R_h(\eta_0)$. Together with the opposite inequality,
this proves \eqref{eq:dpp}.

Notice that no optimal control is required. We use only
$\varepsilon$-optimal controls at two fixed states, so no measurable
selection argument is involved.
\end{proof}

\subsection{Proof of \Cref{thm:viscosity}\textcolor{black}{: Closed-simplex viscosity characterization}}

Continuity of $V$ is not known at this stage, so we first work with its
semicontinuous envelopes
\[
V^*(\eta):=\limsup_{\zeta\to\eta}V(\zeta),
\qquad
V_*(\eta):=\liminf_{\zeta\to\eta}V(\zeta).
\]
Step~1 proves that $V^*$ is a subsolution and Step~2 proves that
$V_*$ is a supersolution, both on the closed simplex. Step~3 applies
comparison to obtain $V^*\le V_*$. Since the reverse inequality
always holds, the envelopes coincide with $V$, proving continuity,
comparison, and uniqueness.

\medskip
\noindent\textbf{Step 1: Subsolution property of $V^*$.} Let $V^*-\phi$ have a local maximum at $\eta_0\in\D$ relative to $\D$, and set $m=V^*(\eta_0)-\phi(\eta_0)$. Choose $\eta_n\to\eta_0$ with $V(\eta_n)\to V^*(\eta_0)$ and put $\delta_n=m-(V(\eta_n)-\phi(\eta_n))$. Then $\delta_n\ge0$ and $\delta_n\to0$. Choose $h_n\downarrow0$ with $\delta_n/h_n\to0$. Let $\bar a$ be a finite rate matrix attaining the pointwise infimum defining $H(\eta_0,D\phi(\eta_0))$; it exists by the scalar entropy minimizations underlying \eqref{eq:hamiltonian}. Use the constant control $\bar a$ in \eqref{eq:dpp}. Its trajectories stay in $\D$ and remain in the touching neighborhood for $n$ large. The maximality of $V^*-\phi$ gives
\[
V(\eta_n(h_n))-V(\eta_n)\le \phi(\eta_n(h_n))-\phi(\eta_n)+\delta_n.
\]
Substitution into \eqref{eq:dpp}, division by $h_n$, and passage to the limit give
\[
\rho V^*(\eta_0)\le f(\eta_0,\bar a)+D\phi(\eta_0)\cdot \eta_0\bar a=H(\eta_0,D\phi(\eta_0)).
\]
Thus $V^*$ is a subsolution even at boundary points.

\medskip
\noindent\textbf{Step 2: Supersolution property of $V_*$.}
We next prove that $V_*$ is a
viscosity supersolution on $\D$. Let $V_*-\phi$ have a local minimum at $\eta_0\in\D$, set $m=V_*(\eta_0)-\phi(\eta_0)$, and choose $\eta_n\to\eta_0$ with $V(\eta_n)\to V_*(\eta_0)$. Then $\delta_n:=V(\eta_n)-\phi(\eta_n)-m\ge0$ for $n$ large and $\delta_n\to0$. Choose $h_n\downarrow0$ so that $\delta_n/h_n\to0$, and from \eqref{eq:dpp} choose a control $a^n$ whose right-hand side is at most $V(\eta_n)+h_n^2$.

For $h_n\le1$, near optimality gives
\[
\frac1\beta
\int_0^{h_n}e^{-\rho t}
\sum_{i\ne j}\eta_i^n(t)
\ell\bigl(a_{ij}^n(t),q_{ij}(\eta^n(t))\bigr)\,\dd t
\le
2\|V\|_\infty+1+\frac{\|C\|_\infty}{\rho}.
\]
Since $e^{-\rho t}\ge e^{-\rho}$ on $[0,h_n]$, this implies
\[
\int_0^{h_n}
\sum_{i\ne j}\eta_i^n(t)
\ell\bigl(a_{ij}^n(t),q_{ij}(\eta^n(t))\bigr)\,\dd t
\le K
\]
for a constant $K$ independent of $n$.

The entropy dual inequality
\[
\ell(a,q)\ge \gamma a-q(e^\gamma-1),\qquad \gamma>0,
\]
therefore yields
\[
\int_0^{h_n}\sum_{i\ne j}\eta_i^n(t)a^n_{ij}(t)\dd t
\le \frac{K}{\gamma}+\frac{(d-1)q_{\max}h_n(e^\gamma-1)}{\gamma}.
\]
With $\gamma=\frac12\log(1/h_n)$, the right-hand side tends to zero. Since the $\ell^1$-norm of the drift in \eqref{eq:state-equation} is at most twice the total flux, $\sup_{t\le h_n}\|\eta^n(t)-\eta_n\|_1\to0$. Hence the trajectory remains in the touching neighborhood.

There $V(\zeta)\ge V_*(\zeta)\ge\phi(\zeta)+m$. Applying this at the terminal state in the nearly optimal form of \eqref{eq:dpp} and using the chain rule for $\phi$ gives
\[
\delta_n+h_n^2+(1-e^{-\rho h_n})m
\ge \int_0^{h_n}e^{-\rho t}\bigl[f(\eta^n(t),a^n(t))+D\phi(\eta^n(t))\cdot \eta^n(t)a^n(t)-\rho\phi(\eta^n(t))\bigr]\dd t.
\]
By the definition of the Hamiltonian, the integrand is bounded below after replacing the first two terms by $H(\eta^n(t),D\phi(\eta^n(t)))$. Divide by $h_n$ and pass to the limit, using continuity of \eqref{eq:hamiltonian}, to obtain
\[
\rho V_*(\eta_0)-H(\eta_0,D\phi(\eta_0))\ge0.
\]
Thus $V_*$ is a supersolution on the closed simplex.

\vspace{3pt}
\noindent{\bf Step 3: Comparison, continuity, and uniqueness.} 
Introduce
\[
\mathcal K(\eta,p)
:=
\sum_{i\ne j}\eta_iq_{ij}(\eta)
\bigl(e^{p_j-p_i}-1\bigr).
\]
By \eqref{eq:hamiltonian}, the change of unknown $u=-\beta U$
transforms \eqref{eq:hjb} into
\begin{equation}\label{eq:kraaij-equation}
u(\eta)-\frac1\rho\mathcal K(\eta,Du(\eta))
=
-\frac{\beta}{\rho}C(\eta).
\end{equation}

We verify explicitly the hypotheses of
\cite[Theorem~4]{Kraaij2016}. In Kraaij's notation, define
\[
v(i,j,\eta)
:=
\begin{cases}
\eta_iq_{ij}(\eta),&i\ne j,\\
0,&i=j.
\end{cases}
\]
Then
\[
\mathcal K(\eta,p)
=
\sum_{i,j=1}^d
v(i,j,\eta)\bigl(e^{p_j-p_i}-1\bigr),
\]
which is precisely the Hamiltonian in
\cite[(2.4)]{Kraaij2016}.

For $i\ne j$, the map $v(i,j,\cdot)$ is continuous and
\[
v(i,j,\eta)>0
\qquad\text{whenever }\eta_i>0.
\]
If $\nu_i=0$, then
\[
v(i,j,\eta)
=
v_{\nu,\dagger}(i,j,\eta_i)
v_{\nu,\ddagger}(i,j,\eta),
\]
where
\[
v_{\nu,\dagger}(i,j,s):=s,
\qquad
v_{\nu,\ddagger}(i,j,\eta):=q_{ij}(\eta).
\]
The first factor is increasing in $s$, while the second is continuous
and satisfies
\[
v_{\nu,\ddagger}(i,j,\nu)
=
q_{ij}(\nu)
\ge q_{\min}>0.
\]
For $i=j$, the rate is identically zero, which is the other alternative
allowed by Kraaij's hypotheses. Thus conditions {\rm (a)}--{\rm (b)} of
\cite[Theorem~2]{Kraaij2016} are satisfied.

Equation \eqref{eq:kraaij-equation} is Kraaij's equation with
\[
\lambda=\frac1\rho>0,
\qquad
h(\eta)=-\frac{\beta}{\rho}C(\eta)\in\mathcal{C}(\D).
\]
The use of intrinsic gradients causes no ambiguity because
\begin{equation}\label{eq:hamiltonian-gauge-invariance}
\mathcal K(\eta,p+c\one)=\mathcal K(\eta,p).
\end{equation}
Kraaij formulates the touching conditions using global extrema on the
compact simplex. Our local-touching definition implies those conditions,
since every global extremum is also a local extremum.

{\color{black}The test-function classes also agree. Let
\[
\mathsf P_{\mathsf H}x
:=
x-\frac{\sum_{i=1}^d x_i-1}{d}\one
\]
be the orthogonal projection onto $\mathsf H$. If $\phi$ is of class
$\mathcal C^1$ on a relative neighborhood $U$ of $\D$ in $\mathsf H$,
then
\[
\widetilde\phi(x)
:=
\phi\bigl(\mathsf P_{\mathsf H}x\bigr),
\qquad
x\in\mathsf P_{\mathsf H}^{-1}(U),
\]
is a $\mathcal C^1$ extension to an ambient neighborhood of $\D$.
Conversely, every ambient $\mathcal C^1$ test function restricts to a
$\mathcal C^1$ function on $\mathsf H$. Normal components of the
ambient gradient are immaterial by
\eqref{eq:hamiltonian-gauge-invariance}.
Hence, after restriction to $\D$, our test-function class agrees with
the class used in \cite[Definition~1]{Kraaij2016}.
}

Multiplication by $-\beta$ interchanges upper and lower
semicontinuity and reverses the viscosity inequalities. Hence
$-\beta V_*$ is a bounded upper semicontinuous subsolution of
\eqref{eq:kraaij-equation}, while $-\beta V^*$ is a bounded lower
semicontinuous supersolution. Since the preceding arguments established
the corresponding inequalities on all of $\D$,
\cite[Theorem~4]{Kraaij2016} gives
\[
-\beta V_*\le-\beta V^*.
\]
Therefore $V_*\ge V^*$. The reverse inequality follows from the
definitions of the semicontinuous envelopes, so
\[
V_*=V^*=V.
\]
Thus $V$ is continuous.

More generally, let $\underline U$ be any bounded upper
semicontinuous subsolution and let $\overline U$ be any bounded lower
semicontinuous supersolution of \eqref{eq:hjb} on $\D$. Then
$-\beta\overline U$ is a subsolution and $-\beta\underline U$ is a
supersolution of \eqref{eq:kraaij-equation}. Comparison gives
\[
-\beta\overline U\le-\beta\underline U,
\]
or equivalently
\[
\underline U\le\overline U.
\]
This proves comparison and uniqueness.

\hfill\qed

{\color{black}
\subsection{Monotonicity and limiting values}

}

\color{black}

\subsection{Proof of \Cref{thm:beta-limits}\textcolor{black}{: Monotonicity and inverse-temperature limits}}

{\color{black}
The proof has three steps. We first prove monotonicity, then obtain
the small-$\beta$ limit by controlling entropy, fluxes, and trajectories,
and finally obtain the large-$\beta$ limit using a fast control toward
a minimizer of $C$.
}

For an admissible control $a$ with associated trajectory $\eta$, write
\[
J_C(\eta_0;a)
:=
\int_0^\infty e^{-\rho t}C(\eta(t))\,dt
\]
for its discounted running cost, and
\[
E(\eta_0;a)
:=
\int_0^\infty e^{-\rho t}
\sum_{i\ne j}
\eta_i(t)\ell\bigl(a_{ij}(t),q_{ij}(\eta(t))\bigr)\,dt
\]
for its discounted entropy expenditure, where
\[
\ell(x,r)=x\log(x/r)-x+r.
\]
Thus
\[
J_\beta(\eta_0;a)
=
J_C(\eta_0;a)+\frac{1}{\beta}E(\eta_0;a).
\]
We use the usual conventions at zero. In particular,
$\ell(0,r)=r$ and $\ell(0,0)=0$.

\medskip
\noindent\textcolor{black}{\textbf{Step 1: Monotonicity.}}
Since the entropy expenditure is nonnegative, increasing $\beta$
can only decrease the cost of any fixed control. More precisely,
if $0<\beta_1<\beta_2$, then
\[
J_{\beta_1}(\eta_0;a)\ge J_{\beta_2}(\eta_0;a)
\]
for every admissible control $a$.
Taking infima over the same admissible class gives
\[
V_{\beta_1}(\eta_0)\ge V_{\beta_2}(\eta_0).
\]

\medskip
\noindent\textcolor{black}{\textbf{Step 2: The limit as $\beta\downarrow0$.}}
Assume that the reference rates are Lipschitz. Then
$F(\eta):=\eta q(\eta)$ is Lipschitz on the simplex, so the
reference equation has a unique solution
$\eta^q(\cdot;\eta_0)$ from every initial state.

{\color{black}
We first control entropy and fluxes to obtain trajectory convergence;
we then compare discounted costs to conclude the small-$\beta$ limit.
}

\smallskip
\noindent\textcolor{black}{\textbf{Step 2.1: Bound the entropy expenditure.}}
The control
\[
a(t)=q(\eta^q(t;\eta_0))
\]
is admissible and has zero entropy expenditure. Consequently,
\[
V_\beta(\eta_0)\le V^q(\eta_0).
\]
For each $\beta>0$ and $\eta_0\in\D$, choose an admissible
control $a^\beta$ such that
\[
J_\beta(\eta_0;a^\beta)
\le V_\beta(\eta_0)+\beta,
\]
and denote its trajectory by $\eta^\beta$.
We use this same control throughout the argument.

By the decomposition of $J_\beta$,
\[
\frac{1}{\beta}E(\eta_0;a^\beta)
\le
V_\beta(\eta_0)-J_C(\eta_0;a^\beta)+\beta
\le
V^q(\eta_0)-J_C(\eta_0;a^\beta)+\beta.
\]
Both terms in the last difference are discounted running costs.
Writing
\[
C_{\min}:=\min_{\D}C,
\qquad
C_{\max}:=\max_{\D}C,
\]
we have
\[
V^q(\eta_0)\le\frac{C_{\max}}{\rho},
\qquad
J_C(\eta_0;a^\beta)\ge\frac{C_{\min}}{\rho}.
\]
It follows that
\[
E(\eta_0;a^\beta)
\le
\beta\left(
\frac{C_{\max}-C_{\min}}{\rho}+\beta
\right).
\]
Thus the entropy expenditure tends to zero uniformly in the
initial state.

\smallskip
\noindent\textcolor{black}{\textbf{Step 2.2: Control the integrated flux error.}}
Along the controlled trajectory, define
\[
w_{ij}(t):=\eta_i^\beta(t)a_{ij}^\beta(t),
\qquad
r_{ij}(t):=\eta_i^\beta(t)q_{ij}(\eta^\beta(t)).
\]
Here $w_{ij}$ is the actual population flux, while $r_{ij}$ is
the reference flux evaluated at the same state $\eta^\beta(t)$.
The entropy can be expressed directly in terms of these fluxes,
since
\[
\eta_i^\beta
\ell\bigl(a_{ij}^\beta,q_{ij}(\eta^\beta)\bigr)
=
\ell(w_{ij},r_{ij}).
\]

We use the elementary inequality
\[
(\sqrt{x}-\sqrt{r})^2\le\ell(x,r).
\]
Indeed, for $x,r>0$, setting $u=\sqrt{x/r}$ gives
\[
\ell(x,r)-(\sqrt{x}-\sqrt{r})^2
=
2ru\bigl(u\log u-u+1\bigr)\ge0.
\]
The cases involving zero follow by continuity and the entropy
conventions.

To control the ordinary flux difference, write
\[
x-r
=
(\sqrt{x}-\sqrt{r})^2
+
2\sqrt{r}(\sqrt{x}-\sqrt{r}).
\]
Taking absolute values and applying the preceding inequality yields
\[
|x-r|
\le
\ell(x,r)+2\sqrt{r\,\ell(x,r)}.
\]

Fix $T<\infty$. Integrating this bound over time and summing
over the directed edges, Cauchy--Schwarz gives
\[
\begin{aligned}
\int_0^T\sum_{i\ne j}|w_{ij}(t)-r_{ij}(t)|\,dt
&\le
\int_0^T\sum_{i\ne j}\ell(w_{ij}(t),r_{ij}(t))\,dt\\
&\quad+
2\left(
\int_0^T\sum_{i\ne j}r_{ij}(t)\,dt
\right)^{1/2}
\left(
\int_0^T\sum_{i\ne j}\ell(w_{ij}(t),r_{ij}(t))\,dt
\right)^{1/2}.
\end{aligned}
\]
On $[0,T]$, removing the discount costs at most a factor
$e^{\rho T}$, so
\[
\int_0^T\sum_{i\ne j}\ell(w_{ij}(t),r_{ij}(t))\,dt
\le e^{\rho T}E(\eta_0;a^\beta).
\]
Also, boundedness of the reference rates gives
\[
\sum_{i\ne j}r_{ij}(t)
\le
(d-1)q_{\max}\sum_i\eta_i^\beta(t)
=
(d-1)q_{\max}.
\]
Therefore
\[
\begin{aligned}
\int_0^T\sum_{i\ne j}|w_{ij}(t)-r_{ij}(t)|\,dt
&\le e^{\rho T}E(\eta_0;a^\beta)\\
&\quad+
2\sqrt{
(d-1)q_{\max}T\,e^{\rho T}E(\eta_0;a^\beta)
}
\longrightarrow0.
\end{aligned}
\]
For each fixed $T$, this convergence is uniform in $\eta_0$.

\smallskip
\noindent\textcolor{black}{\textbf{Step 2.3: Prove finite-time trajectory convergence.}}
The preceding estimate says that the controlled trajectory
approximately satisfies the reference equation. Indeed,
each edge flux contributes once to an outflow and once to an
inflow, and hence
\[
\|\dot\eta^\beta(t)-F(\eta^\beta(t))\|_1
\le
2\sum_{i\ne j}|w_{ij}(t)-r_{ij}(t)|
\]
for almost every $t$.

The reference trajectory satisfies
$\dot\eta^q=F(\eta^q)$ and has the same initial state.
Subtracting the two integral equations and using the Lipschitz
constant $L_F$ of $F$, we obtain, for $0\le t\le T$,
\[
\begin{aligned}
\|\eta^\beta(t)-\eta^q(t;\eta_0)\|_1
&\le
L_F\int_0^t
\|\eta^\beta(s)-\eta^q(s;\eta_0)\|_1\,ds\\
&\quad+
2\int_0^T\sum_{i\ne j}|w_{ij}(s)-r_{ij}(s)|\,ds.
\end{aligned}
\]
Gronwall's inequality therefore gives
\[
\sup_{0\le t\le T}
\|\eta^\beta(t)-\eta^q(t;\eta_0)\|_1
\le
2e^{L_FT}
\int_0^T\sum_{i\ne j}|w_{ij}(t)-r_{ij}(t)|\,dt
\longrightarrow0,
\]
uniformly in $\eta_0$.

\smallskip
\noindent\textcolor{black}{\textbf{Step 2.4: Pass to discounted running costs.}}
For every fixed initial state and every fixed time,
continuity of $C$ gives
\[
C(\eta^\beta(t))
\longrightarrow C(\eta^q(t;\eta_0)).
\]
Moreover,
\[
e^{-\rho t}
|C(\eta^\beta(t))-C(\eta^q(t;\eta_0))|
\le
2\|C\|_\infty e^{-\rho t}.
\]
The right-hand side is integrable on $[0,\infty)$, so dominated
convergence yields
\[
J_C(\eta_0;a^\beta)\longrightarrow V^q(\eta_0).
\]

This convergence is uniform in the initial state. To see this
explicitly, fix $T<\infty$. Uniform continuity of $C$ and the
uniform trajectory estimate imply
\[
\sup_{\eta_0\in\D}\sup_{0\le t\le T}
|C(\eta^\beta(t))-C(\eta^q(t;\eta_0))|
\longrightarrow0.
\]
Consequently,
\[
\begin{aligned}
\sup_{\eta_0\in\D}
|J_C(\eta_0;a^\beta)-V^q(\eta_0)|
&\le
\frac{1}{\rho}
\sup_{\eta_0\in\D}\sup_{0\le t\le T}
|C(\eta^\beta(t))-C(\eta^q(t;\eta_0))|\\
&\quad+
\frac{2\|C\|_\infty}{\rho}e^{-\rho T}.
\end{aligned}
\]
First sending $\beta\downarrow0$ with $T$ fixed, and then
$T\to\infty$, proves the asserted uniform convergence of the
running costs.

\smallskip
\noindent\textcolor{black}{\textbf{Step 2.5: Squeeze the value functions.}}
Near optimality and nonnegativity of the entropy term imply
\[
V_\beta(\eta_0)
\ge J_\beta(\eta_0;a^\beta)-\beta
\ge J_C(\eta_0;a^\beta)-\beta.
\]
Combining this with the reference-control upper bound,
\[
J_C(\eta_0;a^\beta)-\beta
\le V_\beta(\eta_0)\le V^q(\eta_0).
\]
The left-hand side converges uniformly to $V^q$, so
\[
\|V_\beta-V^q\|_\infty\longrightarrow0.
\]
Notice that we did not need to establish convergence of
$E(\eta_0;a^\beta)/\beta$ separately: its nonnegativity is
enough for this final comparison.

\medskip
\noindent\textcolor{black}{\textbf{Step 3: The limit as $\beta\to\infty$.}}
Since the running cost is bounded below by $C_{\min}$ and
the entropy term is nonnegative, every admissible control satisfies
\[
J_\beta(\eta_0;a)\ge
\int_0^\infty e^{-\rho t}C_{\min}\,dt
=
\frac{C_{\min}}{\rho}.
\]
Thus
\[
V_\beta(\eta_0)\ge\frac{C_{\min}}{\rho}.
\]
To obtain the matching upper bound, we construct a control
that drives the population increasingly rapidly toward a
minimizer of $C$, while keeping its entropy penalty small.

\smallskip
\noindent\textcolor{black}{\textbf{Step 3.1: Choose a control with an explicit trajectory.}}
Let $\bar\eta\in\D$ satisfy $C(\bar\eta)=C_{\min}$.
For $M\ge1$, define the constant generator by
\[
a_{ij}^M=M\bar\eta_j,\qquad i\ne j,
\]
with diagonal entries chosen so that each row sums to zero.
This is a bounded admissible control for each fixed $M$.
Its trajectory satisfies
\[
\begin{aligned}
\dot\eta_i^M
&=
\sum_{j\ne i}\eta_j^M M\bar\eta_i
-
\eta_i^M\sum_{j\ne i}M\bar\eta_j\\
&=
M\bar\eta_i(1-\eta_i^M)
-
M\eta_i^M(1-\bar\eta_i)\\
&=
M(\bar\eta_i-\eta_i^M).
\end{aligned}
\]
Solving this equation gives
\[
\eta^M(t)
=
\bar\eta+e^{-Mt}(\eta_0-\bar\eta).
\]
In particular,
\[
\|\eta^M(t)-\bar\eta\|_1\le2e^{-Mt},
\]
uniformly over the initial state.

\smallskip
\noindent\textcolor{black}{\textbf{Step 3.2: Approach the minimum running cost.}}
For every fixed $t>0$, the trajectory approaches $\bar\eta$
as $M\to\infty$, and hence
\[
C(\eta^M(t))\longrightarrow C_{\min}.
\]
The bound
\[
e^{-\rho t}|C(\eta^M(t))-C_{\min}|
\le2\|C\|_\infty e^{-\rho t}
\]
allows us to apply dominated convergence. Therefore
\[
J_C(\eta_0;a^M)\longrightarrow\frac{C_{\min}}{\rho}.
\]
The initial condition at $t=0$ does not affect this argument,
since a single time has zero measure.

The convergence is also uniform in $\eta_0$. Indeed, fix
$\varepsilon>0$. The contribution of $[0,\varepsilon]$ to the
running-cost difference is at most
$2\|C\|_\infty\varepsilon$. For all $t\ge\varepsilon$,
\[
\|\eta^M(t)-\bar\eta\|_1\le2e^{-M\varepsilon},
\]
so uniform continuity of $C$ makes the running-cost difference
uniformly small on this entire interval. Its discounted integral
is bounded by that uniform difference divided by $\rho$.
First letting $M\to\infty$, and then $\varepsilon\downarrow0$,
gives
\[
\sup_{\eta_0\in\D}
\left|J_C(\eta_0;a^M)-\frac{C_{\min}}{\rho}\right|
\longrightarrow0.
\]

\smallskip
\noindent\textcolor{black}{\textbf{Step 3.3: Choose the speed to make the entropy penalty vanish.}}
The rates increase with $M$, so the entropy expenditure need
not remain bounded. We only require it to grow more slowly
than $\beta$.

Set
\[
c_q:=\max\{0,\log(1/q_{\min})\}.
\]
For $0<x\le M$ and $q_{\min}\le r\le q_{\max}$,
\[
\ell(x,r)
=x\log(x/r)-x+r
\le x(\log M+c_q)+q_{\max}.
\]
The same inequality holds at $x=0$, since $\ell(0,r)=r$.
Applying it to $a_{ij}^M=M\bar\eta_j$ gives
\[
\begin{aligned}
\sum_{i\ne j}\eta_i^M(t)
\ell(a_{ij}^M,q_{ij}(\eta^M(t)))
&\le
M(\log M+c_q)
\sum_{i\ne j}\eta_i^M(t)\bar\eta_j
+(d-1)q_{\max}\\
&\le M(\log M+c_q)+(d-1)q_{\max}.
\end{aligned}
\]
After discounting and integrating,
\[
E(\eta_0;a^M)
\le
\frac{M(\log M+c_q)+(d-1)q_{\max}}{\rho},
\]
uniformly in the initial state.

Now choose $M=\sqrt{\beta}$ for $\beta\ge1$.
Then $M\to\infty$, while
\[
\frac{E(\eta_0;a^{\sqrt{\beta}})}{\beta}
\le
\frac{
\sqrt{\beta}\bigl(\tfrac12\log\beta+c_q\bigr)
+(d-1)q_{\max}
}{\beta\rho}
\longrightarrow0.
\]
Thus the population approaches the minimizing state increasingly
rapidly, but the resulting entropy expenditure grows slowly
enough that its contribution to the cost vanishes.

Finally, using this control as a competitor,
\[
\frac{C_{\min}}{\rho}
\le V_\beta(\eta_0)
\le
J_C(\eta_0;a^{\sqrt{\beta}})
+\frac{E(\eta_0;a^{\sqrt{\beta}})}{\beta}.
\]
The right-hand side converges uniformly to $C_{\min}/\rho$.
This proves
\[
\left\|V_\beta-\frac{C_{\min}}{\rho}\right\|_\infty
\longrightarrow0.
\]
\hfill\qed

\color{black}

\section{Semiconvexity, semiconcavity, and Gibbs feedback}\label{app:regularity}

\subsection{Regularity and optimal Gibbs feedback}

\begin{proposition}[Global semiconvexity]\label{prop:semiconvexity}
Under \Cref{ass:regularity}, $V+\Phi_K$ is convex on $\D$. Equivalently, $V$ is globally $K$-semiconvex.
\end{proposition}

\begin{proof}
Let $\widehat J_K$ and $\widehat V_K$ denote the discounted cost and value obtained by replacing $L$ with \eqref{eq:modified-flux-lagrangian}. 

{\color{black}
We first verify that the modified extended-real-valued cost is well
defined. For each directed edge set
\[
c_{ij}(\eta):=K(\eta_i-\eta_j).
\]
Since
\[
-D\Phi_K(\eta)\cdot b(w)
=
\sum_{i\ne j}c_{ij}(\eta)w_{ij}
\]
and $|c_{ij}(\eta)|\le K$, scalar entropy minimization gives, for
$r,w\ge0$,
\[
\frac1\beta\ell(w,r)+c_{ij}(\eta)w
\ge
\inf_{z\ge0}
\left\{
\frac1\beta\ell(z,r)+c_{ij}(\eta)z
\right\}
=
\frac r\beta\bigl(1-e^{-\beta c_{ij}(\eta)}\bigr)
\ge
-\frac r\beta\bigl(e^{\beta K}-1\bigr).
\]
Moreover,
\[
\sum_{i\ne j}r_{ij}(\eta)
=
\sum_i\eta_i\sum_{j\ne i}q_{ij}(\eta)
\le
(d-1)q_{\max}.
\]
Since $\rho\Phi_K\ge0$, it follows that
\[
\widehat L_K(\eta,w)
\ge
-\|C\|_\infty
-\frac{(d-1)q_{\max}}{\beta}
\bigl(e^{\beta K}-1\bigr).
\]
Thus the discounted negative part of the modified running cost is
finite, so $\widehat J_K$ is well defined as an extended-real-valued
functional.

Write $J_T$ and $\widehat J_{K,T}$ for the respective costs truncated
at time $T$. For every $T<\infty$, the discounted chain rule and
$\dot\eta=b(w)$ give
\[
\widehat J_{K,T}(\eta_0;a)
=
J_T(\eta_0;a)
+\Phi_K(\eta_0)
-e^{-\rho T}\Phi_K(\eta(T)).
\]
Both running costs are bounded below uniformly, and $\Phi_K$ is bounded
on $\D$. Letting $T\to\infty$ therefore yields, also when the cost is
$+\infty$,
\[
\widehat J_K(\eta_0;a)
=
J(\eta_0;a)+\Phi_K(\eta_0).
\]
Taking infima gives
\[
\widehat V_K=V+\Phi_K.
\]
}

Take two admissible pairs $(\eta^0,a^0)$ and $(\eta^1,a^1)$ and write $w^k_{ij}=\eta_i^k a^k_{ij}$. For $\theta\in[0,1]$, set $\eta^\theta=(1-\theta)\eta^0+\theta\eta^1$ and $w^\theta=(1-\theta)w^0+\theta w^1$. The flux balance law is linear, so $\dot\eta^\theta=b(w^\theta)$. Define $a^\theta_{ij}=w^\theta_{ij}/\eta_i^\theta$ whenever $\eta_i^\theta>0$, and set the unused row to zero when $\eta_i^\theta=0$. 

{\color{black}
For $0<\theta<1$, whenever $\eta_i^\theta(t)>0$, set
\[
\lambda_i(t)
:=
\frac{(1-\theta)\eta_i^0(t)}{\eta_i^\theta(t)}
\in[0,1].
\]
Then
\[
a_{ij}^\theta(t)
=
\lambda_i(t)a_{ij}^0(t)
+
\bigl(1-\lambda_i(t)\bigr)a_{ij}^1(t),
\]
and hence
\[
0\le a_{ij}^\theta(t)
\le a_{ij}^0(t)+a_{ij}^1(t),  \quad i\neq j
\]
Thus $a^\theta$ is locally integrable. If
$\eta_i^\theta(t)=0$, then
$\eta_i^0(t)=\eta_i^1(t)=0$, and consequently
\[
w_{ij}^0(t)=w_{ij}^1(t)=w_{ij}^\theta(t)=0
\qquad (j\ne i).
\]
Setting that unused row of $a^\theta$ equal to zero is therefore
consistent with $w^\theta=\eta^\theta a^\theta$. The endpoint cases
$\theta=0,1$ are immediate. Thus $a^\theta$ is admissible and generates
the interpolated trajectory.
}

 Joint convexity of \eqref{eq:modified-flux-lagrangian} gives convexity of the corresponding modified costs. Applying this to $\varepsilon$-optimal controls and letting $\varepsilon\downarrow0$ proves convexity of $\widehat V_K=V+\Phi_K$.
\end{proof}

\begin{proposition}[Local semiconcavity]\label{prop:semiconcavity}
Under \Cref{ass:regularity}, for every compact $K\subset\Dint$ there
exist constants $r_K>0$ and $C_K<\infty$ such that, whenever
$\eta\in K$, $h\in\T$, and $\|h\|_\infty\le r_K$, one has
\begin{equation}\label{eq:semiconcavity}
V(\eta+h)+V(\eta-h)-2V(\eta)
\le C_K\|h\|_\infty^2.
\end{equation}
\end{proposition}

\begin{proof}
Since $K\subset\Dint$ is compact, define
\[
\kappa_K
:=
\min_{\zeta\in K}\min_{i\in[d]}\zeta_i>0,
\qquad
r_K:=\frac{\kappa_K}{2}.
\]
Fix $\eta\in K$ and $h\in\T$ with
$\|h\|_\infty\le r_K$. In particular,
\[
\eta_i\pm h_i\ge\frac{\kappa_K}{2}>0,
\]
so $\eta\pm h\in\Dint$.

Let $0<\varepsilon\le1$, and let $a^\varepsilon$ be an
$\varepsilon$-optimal admissible control from $\eta$. Denote its
trajectory by $\eta^\varepsilon$. Thus
\[
J(\eta;a^\varepsilon)\le V(\eta)+\varepsilon.
\]
By \Cref{def:admissible}, the locally integrable generator
$a^\varepsilon$ generates the stochastic evolution family
$P^{a^\varepsilon}(s,t)$ from \eqref{eq:evolution-family}. Hence,
viewing probability vectors as row vectors,
\[
\eta^\varepsilon(t)
=
\eta P^{a^\varepsilon}(0,t).
\]

\smallskip
\noindent{\bf Relative propagation of the perturbation.} Apply the same open-loop rate control $a^\varepsilon$ to the initial
conditions $\eta+h$ and $\eta-h$. The corresponding trajectories are
\[
\eta^\pm(t)
=
(\eta\pm h)P^{a^\varepsilon}(0,t)
=
\eta^\varepsilon(t)\pm\delta\eta(t),
\qquad
\delta\eta(t)
:=
hP^{a^\varepsilon}(0,t).
\]

Write
\[
h_i=\eta_i\xi_i,
\qquad i\in[d].
\]
Then
\[
\|\xi\|_\infty
\le
\frac{\|h\|_\infty}{\kappa_K}
\le\frac12.
\]
Since $P^{a^\varepsilon}_{ki}(0,t)\ge0$, for each $i$ and $t\ge0$,
\begin{align}
|\delta\eta_i(t)|
&=
\left|
\sum_{k=1}^d
h_kP^{a^\varepsilon}_{ki}(0,t)
\right| \notag\\
&=
\left|
\sum_{k=1}^d
\eta_k\xi_kP^{a^\varepsilon}_{ki}(0,t)
\right| \notag\\
&\le
\|\xi\|_\infty
\sum_{k=1}^d
\eta_kP^{a^\varepsilon}_{ki}(0,t) \notag\\
&=
\|\xi\|_\infty\eta_i^\varepsilon(t).
\label{eq:relative-perturbation}
\end{align}
Set
\[
\vartheta:=\|\xi\|_\infty\le\frac12.
\]
It follows from \eqref{eq:relative-perturbation} that
\[
(1-\vartheta)\eta_i^\varepsilon(t)
\le
\eta_i^\pm(t)
\le
(1+\vartheta)\eta_i^\varepsilon(t).
\]
Moreover, since
$P^{a^\varepsilon}(0,t)\one=\one$ and $h\one=0$,
\[
\sum_i\delta\eta_i(t)=0.
\]
Thus $\eta^\pm(t)\in\D$ for every $t\ge0$, and the same control
$a^\varepsilon$ is admissible from both perturbed initial conditions.
We shall also use
\[
\|\delta\eta(t)\|_1
\le
\vartheta\sum_i\eta_i^\varepsilon(t)
=
\vartheta,
\qquad
\|\delta\eta(t)\|_2\le \vartheta.
\]

\smallskip
\noindent{\bf Entropy, flux, and perturbed-cost integrability.} We next record the entropy and flux bounds that will also ensure that
the two perturbed costs are finite. From $\varepsilon$-optimality and
the boundedness of $C$,
\[
\begin{aligned}
&\int_0^\infty e^{-\rho t}
\sum_{i\ne j}\eta_i^\varepsilon(t)
\ell\bigl(
a_{ij}^\varepsilon(t),
q_{ij}(\eta^\varepsilon(t))
\bigr)\,\dd t\\
&\qquad\le
\beta
\left(
\|V\|_\infty+1+\frac{\|C\|_\infty}{\rho}
\right)
=:M_{\mathrm{ent}}.
\end{aligned}
\]
The entropy dual inequality with parameter $1$ gives
\[
\ell(a,q)\ge a-(e-1)q.
\]
Consequently,
\[
\begin{aligned}
&\int_0^\infty e^{-\rho t}
\sum_{i\ne j}
\eta_i^\varepsilon(t)a_{ij}^\varepsilon(t)\,\dd t\\
&\qquad\le
M_{\mathrm{ent}}
+
\frac{(e-1)(d-1)q_{\max}}{\rho}
=:M_{\mathrm{flux}}<\infty.
\end{aligned}
\]

We now verify that the perturbed costs are finite. For
\[
q^0_{ij}(t)
:=
q_{ij}(\eta^\varepsilon(t)),
\qquad
q^\pm_{ij}(t)
:=
q_{ij}(\eta^\pm(t)),
\]
the identity
\[
\ell(a,q^\pm)
=
\ell(a,q^0)
+
a\log\frac{q^0}{q^\pm}
+
q^\pm-q^0
\]
and the bounds
$q_{\min}\le q^0,q^\pm\le q_{\max}$ give
\[
\ell(a,q^\pm)
\le
\ell(a,q^0)+c_q(1+a)
\]
for a constant $c_q$ depending only on $q_{\min}$ and $q_{\max}$.
Since
\[
\eta_i^\pm(t)
\le
(1+\vartheta)\eta_i^\varepsilon(t),
\]
the preceding entropy and flux estimates imply
\[
\int_0^\infty e^{-\rho t}
\sum_{i\ne j}\eta_i^\pm(t)
\ell\bigl(
a_{ij}^\varepsilon(t),
q_{ij}(\eta^\pm(t))
\bigr)\,\dd t
<\infty.
\]
Since $C$ is bounded, it follows that
\[
J(\eta+h;a^\varepsilon)<\infty,
\qquad
J(\eta-h;a^\varepsilon)<\infty.
\]

\smallskip
\noindent{\bf Cost comparison and second-difference estimate.} We may therefore compare the three costs without any
$\infty-\infty$ ambiguity. Since $a^\varepsilon$ is an admissible
competitor from $\eta+h$ and $\eta-h$,
\[
V(\eta+h)\le J(\eta+h;a^\varepsilon),
\qquad
V(\eta-h)\le J(\eta-h;a^\varepsilon).
\]
On the other hand, $\varepsilon$-optimality gives
\[
-2V(\eta)
\le
-2J(\eta;a^\varepsilon)+2\varepsilon.
\]
Consequently,
\[
\begin{aligned}
&V(\eta+h)+V(\eta-h)-2V(\eta)\\
&\qquad\le
\int_0^\infty e^{-\rho t}
\Big[
f(\eta^\varepsilon(t)+\delta\eta(t),a^\varepsilon(t))
+
f(\eta^\varepsilon(t)-\delta\eta(t),a^\varepsilon(t))\\
&\hspace{48mm}
-
2f(\eta^\varepsilon(t),a^\varepsilon(t))
\Big]\dd t
+
2\varepsilon.
\end{aligned}
\]

Fix $t\ge0$ and abbreviate
\[
x:=\eta^\varepsilon(t),
\qquad
z:=\delta\eta(t),
\qquad
a:=a^\varepsilon(t).
\]
We estimate the symmetric second difference of $f$ with respect to
the state variable, holding $a$ fixed.

We repeatedly use the following consequence of
$\mathcal{C}^{1,1}$ regularity. If $\varphi$ has Lipschitz gradient
with Lipschitz constant $L_\varphi$, then
\[
\begin{aligned}
&\varphi(x+z)+\varphi(x-z)-2\varphi(x)\\
&\qquad=
\int_0^1
\bigl(
D\varphi(x+sz)-D\varphi(x-sz)
\bigr)\cdot z\,\dd s,
\end{aligned}
\]
and therefore
\[
\left|
\varphi(x+z)+\varphi(x-z)-2\varphi(x)
\right|
\le
L_\varphi\|z\|_2^2.
\]
Since $\|z\|_2\le \vartheta$, the contribution of the state cost $C$ is
bounded in absolute value by $c\vartheta^2$, where here and below $c$ denotes
a finite constant depending only on the model data.

For each directed edge $i\ne j$, expand its entropy contribution as
\[
x_i
\left[
a_{ij}\log\frac{a_{ij}}{q_{ij}(x)}
-a_{ij}+q_{ij}(x)
\right]
=
x_i a_{ij}(\log a_{ij}-1)
-
a_{ij}x_i\log q_{ij}(x)
+
x_iq_{ij}(x),
\]
with the usual convention $0\log0=0$. For fixed $a_{ij}$, the first
term on the right-hand side is affine in $x$, and hence its symmetric
second difference is zero.

Define
\[
\varphi_{ij}(x):=\log q_{ij}(x),
\qquad
g_{ij}(x):=-x_i\varphi_{ij}(x).
\]
Since $q_{ij}\ge q_{\min}>0$ and
$q_{ij}\in\mathcal{C}^{1,1}$, the function $\varphi_{ij}$ belongs to
$\mathcal{C}^{1,1}$. Thus there are constants $B,L<\infty$,
independent of $i,j$, such that
\[
\|D\varphi_{ij}(x)\|\le B
\]
and
\[
\|D\varphi_{ij}(x)-D\varphi_{ij}(y)\|
\le
L\|x-y\|_2,
\qquad x,y\in\D.
\]
Algebraically,
\[
\begin{aligned}
&g_{ij}(x+z)+g_{ij}(x-z)-2g_{ij}(x)\\
&\quad=
-x_i
\bigl[
\varphi_{ij}(x+z)
+
\varphi_{ij}(x-z)
-
2\varphi_{ij}(x)
\bigr]\\
&\qquad
-z_i
\bigl[
\varphi_{ij}(x+z)
-
\varphi_{ij}(x-z)
\bigr].
\end{aligned}
\]
The first bracket is bounded by $L\|z\|_2^2$, while
\[
|\varphi_{ij}(x+z)-\varphi_{ij}(x-z)|
\le
2B\|z\|_2.
\]
Using $|z_i|\le \vartheta x_i$ and $\|z\|_2\le \vartheta$, we obtain
\[
\begin{aligned}
&
|g_{ij}(x+z)+g_{ij}(x-z)-2g_{ij}(x)|\\
&\qquad\le
Lx_i\vartheta^2+2B(\vartheta x_i)\vartheta
\le
c\vartheta^2x_i.
\end{aligned}
\]
Multiplication by $a_{ij}\ge0$ therefore gives
\[
a_{ij}
|g_{ij}(x+z)+g_{ij}(x-z)-2g_{ij}(x)|
\le
c\vartheta^2x_i a_{ij}.
\]

Finally, define
\[
R_{ij}(x):=x_iq_{ij}(x).
\]
Then $R_{ij}\in\mathcal{C}^{1,1}$ and
\[
|R_{ij}(x+z)+R_{ij}(x-z)-2R_{ij}(x)|
\le
c\|z\|_2^2
\le
c\vartheta^2.
\]
Summing over the directed edges and recalling the factor $1/\beta$,
we conclude that
\[
\begin{aligned}
&
f(x+z,a)+f(x-z,a)-2f(x,a)\\
&\qquad\le
c\vartheta^2
\left(
1+\sum_{i\ne j}x_i a_{ij}
\right).
\end{aligned}
\]

Integrating this estimate and using the previously established flux
bound gives
\[
\begin{aligned}
V(\eta+h)+V(\eta-h)-2V(\eta)
&\le
c\vartheta^2
\left(
\frac1\rho+M_{\mathrm{flux}}
\right)
+
2\varepsilon\\
&\le
C_K\|h\|_\infty^2+2\varepsilon,
\end{aligned}
\]
where we used
\[
\vartheta
=
\|\xi\|_\infty
\le
\frac{\|h\|_\infty}{\kappa_K}.
\]
The constant $C_K$ is independent of $\eta$, $h$, and
$\varepsilon$. Letting $\varepsilon\downarrow0$ proves
\eqref{eq:semiconcavity}.
\end{proof}

By \Cref{prop:semiconvexity,prop:semiconcavity}, $V$ is locally both semiconvex and semiconcave on $\Dint$; hence $V\in\mathcal{C}^{1,1}_{\mathrm{loc}}(\Dint)$, equivalently its intrinsic gradient is locally Lipschitz. This is the standard finite-dimensional semiconvexity/semiconcavity criterion; see, for example, \cite{CannarsaSinestrari2004}.

\subsection{Proof of \Cref{thm:gibbs}\textcolor{black}{: Global regularity and Gibbs feedback}}
At every $\eta\in\Dint$, the regularity just proved makes
\eqref{eq:hjb} classical. We divide the proof into five steps:
identification of the Gibbs minimizer, uniform bounds on the intrinsic
gradient differences and rates, global well-posedness of the
closed-loop equation, verification, and uniform regularity up to the
boundary.

\medskip
\noindent\textbf{Step 1: Pointwise minimization and Gibbs formula.}

For fixed $p=DV(\eta)$, the minimization separates over directed edges. The edge objective is
\[
a(p_j-p_i)+\frac1\beta\left(a\log\frac{a}{q_{ij}(\eta)}-a+q_{ij}(\eta)\right),\qquad a\ge0.
\]
For $a>0$, the first and second derivatives of the edge objective are
\[
p_j-p_i+\frac1\beta\log\frac{a}{q_{ij}(\eta)}
\qquad\text{and}\qquad
\frac1{\beta a}>0.
\]
The first derivative tends to $-\infty$ as $a\downarrow0$ and to
$+\infty$ as $a\to\infty$. Hence the unique critical point is the
global minimizer on $[0,\infty)$. Solving the first-order condition gives exactly \eqref{eq:gibbs-feedback}.

\medskip
\noindent\textbf{Step 2: Uniform gradient and rate bounds.}
We next obtain a global bound on the intrinsic gradient. At the minimizer, the contribution of the directed edge $(i,j)$ to the Hamiltonian is $(q_{ij}(\eta)-a^*_{ij}(\eta))/\beta$. Hence the classical HJB equation may be written as

$$
\rho V(\eta)
=
C(\eta)
+
\frac{1}{\beta}\sum_{i\ne j}\eta_i
\bigl(q_{ij}(\eta)-a^*_{ij}(\eta)\bigr).
$$

Rearranging the preceding identity gives
\[
\begin{aligned}
\sum_{i\ne j}\eta_i a^*_{ij}(\eta)
&=
\sum_{i\ne j}\eta_iq_{ij}(\eta)
+
\beta\bigl(C(\eta)-\rho V(\eta)\bigr)\\
&\le
(d-1)q_{\max}
+
\beta\bigl(\|C\|_\infty+\rho\|V\|_\infty\bigr)
=:B.
\end{aligned}
\]
Fix $\eta\in\Dint$ and choose an index $k$ such that $\eta_k\ge 1/d$. Since every summand in the preceding bound is nonnegative, for each $i\ne k$,

$$
\eta_k a^*_{ki}(\eta)\le B,
\qquad\text{and therefore}\qquad
a^*_{ki}(\eta)\le dB.
$$

Using \eqref{eq:gibbs-feedback} for the edge $(k,i)$ and the lower bound $q_{ki}\ge q_{\min}$, we obtain

$$
q_{\min}
\exp\!\left[-\beta\bigl(\partial_iV-\partial_kV\bigr)\right]
\le dB.
$$

Thus $\partial_iV-\partial_kV$ admits a uniform lower bound, independent of $\eta$ and $i$.

For the opposite bound, consider the feasible point
\[
\zeta:=\eta+\frac{\eta_k}{2}(e_i-e_k)\in\D.
\]
Global $K$-semiconvexity gives
\[
V(\zeta)
\ge
V(\eta)
+
\frac{\eta_k}{2}\bigl(\partial_iV(\eta)-\partial_kV(\eta)\bigr)
-\frac{K\eta_k^2}{4}.
\]
Consequently,
\[
\partial_iV(\eta)-\partial_kV(\eta)
\le
\frac{2}{\eta_k}\bigl(V(\zeta)-V(\eta)\bigr)+\frac{K\eta_k}{2}
\le
2d\,\osc_{\D}V+\frac K2.
\]

Hence $\partial_iV-\partial_kV$ is bounded uniformly from both above and below. Since

$$
\partial_iV-\partial_jV
=
(\partial_iV-\partial_kV)
-
(\partial_jV-\partial_kV),
$$

all intrinsic gradient differences $\partial_iV-\partial_jV$ are uniformly bounded on $\Dint$. Substituting these bounds into \eqref{eq:gibbs-feedback}, and using

$$
0<q_{\min}\le q_{ij}(\eta)\le q_{\max}<\infty,
$$

yields the constants $0<\underline a\le \overline a<\infty$ asserted in \eqref{eq:global-rate-bounds}.

\medskip
\noindent\textbf{Step 3: Global closed-loop well-posedness.} Since $DV$ and $q$ are locally Lipschitz on $\Dint$, the closed-loop
vector field is locally Lipschitz there. Let $[0,T_{\max})$ be the
maximal interval of existence. By \eqref{eq:global-rate-bounds},
\[
\dot\eta_i(t)
\ge
-(d-1)\overline a\,\eta_i(t),
\]
and hence
\[
\eta_i(t)
\ge
\eta_i(0)e^{-(d-1)\overline a t},
\qquad t<T_{\max}.
\]
The closed-loop vector field is bounded because the rates are
uniformly bounded. Therefore, if $T_{\max}<\infty$, the trajectory is
Lipschitz and has a limit $\bar\eta$ as $t\uparrow T_{\max}$. The
preceding estimate gives
\[
\bar\eta_i
\ge
\eta_i(0)e^{-(d-1)\overline aT_{\max}}>0
\]
for every $i$, so $\bar\eta\in\Dint$. Local existence from $\bar\eta$
extends the solution beyond $T_{\max}$, contradicting maximality.
Thus $T_{\max}=\infty$.

\medskip
\noindent\textbf{Step 4: Verification and uniqueness.}

Let $a$ be any admissible control from $\eta_0\in\Dint$. If
$J(\eta_0;a)=+\infty$, the verification inequality is immediate, so
suppose that the cost is finite. By
\eqref{eq:interior-preservation}, the trajectory is contained in a
compact subset of $\Dint$ on every finite interval, and the chain rule
for $V$ applies.

Define the Hamiltonian gap
\[
G(\eta,a)
:=
f(\eta,a)+DV(\eta)\cdot(\eta a)-\rho V(\eta).
\]
The classical HJB equation and
\eqref{eq:hamiltonian-variational} imply
\[
G(\eta,a)\ge0.
\]
For every $T<\infty$, the chain rule gives
\[
\begin{aligned}
\int_0^T e^{-\rho t}
G(\eta(t),a(t))\,\dd t
&=
\int_0^T e^{-\rho t}
f(\eta(t),a(t))\,\dd t\\
&\quad
+e^{-\rho T}V(\eta(T))-V(\eta_0).
\end{aligned}
\]
Since $V$ is bounded, letting $T\to\infty$ gives
\[
J(\eta_0;a)-V(\eta_0)
=
\int_0^\infty e^{-\rho t}
G(\eta(t),a(t))\,\dd t
\ge0.
\]

Along the closed-loop trajectory $\eta^*$, define the time-dependent
control $\alpha^*(t):=a^*(\eta^*(t))$. The uniform rate bounds
\eqref{eq:global-rate-bounds} and the bounds on $q$ imply that
$\alpha^*$ is admissible and has bounded running cost, hence finite
discounted cost. Since it attains the Hamiltonian minimum along this
trajectory, its Hamiltonian gap vanishes identically, and the preceding
identity gives
\[
J(\eta_0;\alpha^*)=V(\eta_0).
\]
Hence the Gibbs feedback is optimal.

If $a$ is any other optimal admissible control from $\eta_0$, then
\[
0
=
J(\eta_0;a)-V(\eta_0)
=
\int_0^\infty e^{-\rho t}
G(\eta(t),a(t))\,\dd t.
\]
Therefore
\[
G(\eta(t),a(t))=0
\qquad\text{for almost every }t.
\]
Thus $a(t)$ attains the Hamiltonian minimum almost everywhere. Since
every coordinate of $\eta(t)$ is positive and every edge objective is
strictly convex, the minimizer is unique. Consequently,
\[
a(t)=a^*(\eta(t))
\qquad\text{for almost every }t.
\]
Uniqueness of the closed-loop ODE then also gives uniqueness of the
optimal trajectory.

\medskip
\noindent\textbf{Step 5: Uniform regularity up to the boundary.}

The local semiconcavity result was used above to construct and verify
the Gibbs feedback. We now use the uniform upper bound on the optimal
rates to obtain a semiconcavity estimate whose constant does not
deteriorate near the boundary.

For every rate matrix $a$ whose off-diagonal entries satisfy
$0\le a_{ij}\le\overline a$, the map $\eta\mapsto f(\eta,a)$ has a
$\mathcal{C}^{1,1}$ constant that is uniform over all such $a$.
Indeed, for each directed edge,
\[
\eta_i\ell\bigl(a_{ij},q_{ij}(\eta)\bigr)
=
\eta_i a_{ij}(\log a_{ij}-1)
-a_{ij}\eta_i\log q_{ij}(\eta)
+\eta_iq_{ij}(\eta),
\]
with $0\log0=0$. The first term is affine in $\eta$. Since
$q_{ij}\in\mathcal{C}^{1,1}$, $q_{ij}\ge q_{\min}>0$, and
$a_{ij}\le\overline a$, the gradients of the remaining terms are
uniformly Lipschitz. Together with $C\in\mathcal{C}^{1,1}$, this gives
a constant $L_f<\infty$ such that
\begin{equation}\label{eq:uniform-running-semiconcavity}
\left|
f(x+z,a)+f(x-z,a)-2f(x,a)
\right|
\le
L_f\|z\|_2^2
\end{equation}
whenever $x,x+z,x-z\in\D$ and the off-diagonal entries of $a$ lie in
$[0,\overline a]$.

Fix $\eta\in\Dint$ and $h\in\T$ such that
$\eta+h,\eta-h\in\D$. Let $\eta^*(\cdot)$ be the optimal
closed-loop trajectory starting from $\eta$, and regard
$a(t):=a^*(\eta^*(t))$ as a fixed open-loop control. By
\eqref{eq:global-rate-bounds}, its off-diagonal entries lie in
$[0,\overline a]$ uniformly in $t,i,j$, and by Step~4,
$V(\eta)=J(\eta;a)$.

Apply the same time-dependent rate matrix $a(t)$ to the initial
conditions $\eta+h$ and $\eta-h$. If $P(0,t)$ denotes the stochastic
evolution matrix generated by $a$, then the resulting trajectories are
\[
\eta^\pm(t)
=
(\eta\pm h)P(0,t)
=
\eta^*(t)\pm z(t),
\qquad
z(t):=hP(0,t).
\]
Since $P(0,t)$ is stochastic,
\[
\|z(t)\|_2
\le
\|z(t)\|_1
\le
\|h\|_1
\le
\sqrt d\,\|h\|_2.
\]
The bounded control $a$ has finite cost from all three initial
conditions. Hence its optimality from $\eta$ and admissibility from
$\eta\pm h$ give
\[
\begin{aligned}
V(\eta+h)+V(\eta-h)-2V(\eta)
&\le
\int_0^\infty e^{-\rho t}
\Big[
f(\eta^*(t)+z(t),a(t))\\
&\hspace{4.7em}
+f(\eta^*(t)-z(t),a(t))
-2f(\eta^*(t),a(t))
\Big]\dd t.
\end{aligned}
\]
Using \eqref{eq:uniform-running-semiconcavity} and the preceding bound
on $z(t)$ yields
\begin{equation}\label{eq:global-semiconcavity}
V(\eta+h)+V(\eta-h)-2V(\eta)
\le
\frac{dL_f}{\rho}\|h\|_2^2.
\end{equation}
Crucially, the constant is independent of the distance from $\eta$ to
$\partial\D$.

It remains to extend \eqref{eq:global-semiconcavity} to
$\eta\in\partial\D$. Put
$\eta^{\mathrm{unif}}:=d^{-1}\one$ and, for
$\varepsilon\in(0,1)$, define
\[
\eta_\varepsilon
:=
(1-\varepsilon)\eta+\varepsilon\eta^{\mathrm{unif}},
\qquad
h_\varepsilon:=(1-\varepsilon)h.
\]
Then $\eta_\varepsilon$ and
$\eta_\varepsilon\pm h_\varepsilon$ belong to $\Dint$. Applying
\eqref{eq:global-semiconcavity} to
$(\eta_\varepsilon,h_\varepsilon)$, letting
$\varepsilon\downarrow0$, and using continuity of $V$ proves the same
estimate for every $\eta\in\D$ and $h\in\T$ such that
$\eta\pm h\in\D$.

By \Cref{prop:semiconvexity}, $V$ is globally $K$-semiconvex.
Combining this lower second-order bound with
\eqref{eq:global-semiconcavity} shows that there is a constant
$L_V<\infty$ such that
\begin{equation}\label{eq:global-gradient-lipschitz}
\|D_{\T}V(\eta)-D_{\T}V(\zeta)\|_2
\le
L_V\|\eta-\zeta\|_2,
\qquad
\eta,\zeta\in\Dint.
\end{equation}
Since $\Dint$ is dense in $\D$, the intrinsic gradient extends uniquely
to the boundary. Indeed, if
$(\eta^n)_{n\ge1}\subset\Dint$ converges to
$\eta\in\partial\D$, then \eqref{eq:global-gradient-lipschitz} shows
that $(D_{\T}V(\eta^n))_{n\ge1}$ is Cauchy in $\T$. The limit is
independent of the approximating sequence, again by
\eqref{eq:global-gradient-lipschitz}. We therefore define
\[
D_{\T}V(\eta)
:=
\lim_{\substack{\zeta\to\eta\\ \zeta\in\Dint}}
D_{\T}V(\zeta),
\qquad
\eta\in\partial\D.
\]
Passing to the limit in \eqref{eq:global-gradient-lipschitz} shows that
the same Lipschitz estimate holds for every $\eta,\zeta\in\D$.

It remains to verify that this extension is the intrinsic derivative
of $V$ at boundary points. If $\eta,\zeta\in\Dint$, the segment joining
them lies in $\Dint$, and the fundamental theorem of calculus together
with \eqref{eq:global-gradient-lipschitz} gives
\[
\left|
V(\zeta)-V(\eta)
-
D_{\T}V(\eta)\cdot(\zeta-\eta)
\right|
\le
\frac{L_V}{2}\|\zeta-\eta\|_2^2.
\]
For general $\eta,\zeta\in\D$, approximate both points from the
interior by convex combinations with $\eta^{\mathrm{unif}}$ and pass
to the limit using continuity of $V$ and of the extended gradient.
Thus the same quadratic expansion holds on the closed simplex. Hence
$V$ is intrinsically differentiable on $\D$, its intrinsic gradient is
globally Lipschitz, and $V$ is intrinsically
$\mathcal{C}^{1,1}$ on the closed simplex.

Since the intrinsic gradient extends continuously to the closed simplex,
the Gibbs formula \eqref{eq:gibbs-feedback} defines a continuous rate
matrix there. Passing to the limit from interior states in the classical
HJB equation shows that this matrix attains the Hamiltonian minimum at
every boundary state as well. For each positive-mass state, strict
convexity still gives a unique minimizing row. For a zero-mass state,
the entire row contribution to the Hamiltonian vanishes, so its
off-diagonal rates may be chosen arbitrarily in $[0,\infty)$, with the
diagonal determined by the zero-row-sum condition. This completes the proof.
\hfill\qed

{\color{black}
\subsection{Proof of \Cref{thm:uniform-small-beta-feedback}: Uniform small-$\beta$ convergence of the Gibbs feedback}

\begin{proof}
Comparison with the reference control and nonnegativity of the
entropy penalty give the uniform bounds
\begin{equation}\label{eq:temperature-value-bounds}
\frac{C_{\min}}\rho\le V_\beta\le\frac{C_{\max}}\rho,
\qquad \osc_{\D}V_\beta\le M:=\frac{\osc_{\D}C}{\rho},
\qquad |C-\rho V_\beta|\le\osc_{\D}C.
\end{equation}
Fix $\eta\in\D$, write $g=\beta D_{\T}V_\beta(\eta)$, and let
$\bar g=\sum_i\eta_i g_i$. It suffices to control the largest
component above this average and the average above the smallest.

Choose $j$ with $g_j=\max_i g_i$, and apply semiconvexity between
$\eta$ and $y=\eta+t(e_j-\eta)$, where $0<t\le1$. Since
$\|e_j-\eta\|_2^2\le2$, the linear term is bounded by the
oscillation of $V_\beta$ plus $K_\beta t^2$. Dividing by $t$
and multiplying by $\beta$ gives
\begin{equation}\label{eq:short-segment-scaled-gradient}
0\le\max_i g_i-\bar g
\le\frac{\beta M}{t}+\beta K_\beta t.
\end{equation}
For $\beta\le\min\{\beta_0,1\}$, take $t=\sqrt\beta$.
By \eqref{eq:small-beta-semiconvexity-condition}, the right-hand
side is at most $\delta_\beta:=(M+L)\sqrt\beta$, uniformly in $\eta$.
The short segment is the only change needed in the semiconvexity
argument; it avoids requiring $\beta K_\beta$ to vanish.

Substituting the Gibbs minimizer into the HJB equation yields
\begin{equation}\label{eq:scaled-gradient-hjb-balance}
\sum_{i\ne j}\eta_iq_{ij}(\eta)(e^{g_i-g_j}-1)
=\beta(C(\eta)-\rho V_\beta(\eta)).
\end{equation}
Let $P_\beta$ be the sum of the positive terms on the left and
$N_\beta$ the absolute sum of the negative terms. The elementary
inequality $1-e^{-r}\le r$ for $r\ge0$ gives
$N_\beta\le(d-1)q_{\max}(\max_i g_i-\bar g)$.
Using \eqref{eq:temperature-value-bounds} and \eqref{eq:scaled-gradient-hjb-balance}
therefore bounds $P_\beta$ by
$(d-1)q_{\max}\delta_\beta+\beta\osc_{\D}C$.
If $g_m=\min_i g_i$, all edges into $m$ contribute to $P_\beta$.
Keeping these edges and using $e^r-1\ge r$ shows that
$P_\beta\ge q_{\min}(\bar g-g_m)$. Consequently,
\[
\max_i g_i-\min_i g_i
\le\left(1+\frac{(d-1)q_{\max}}{q_{\min}}\right)\delta_\beta
 +\frac{\beta\osc_{\D}C}{q_{\min}}
=:A_\beta.
\]
This estimate holds on the closed simplex: the HJB identity and
intrinsic gradient extend there by \Cref{thm:gibbs}.
The Gibbs formula now bounds every feedback error by
$q_{\max}(e^{A_\beta}-1)$, proving
\eqref{eq:uniform-feedback-convergence} with order $O(\sqrt\beta)$.
If $K_\beta$ is uniformly bounded, take $t=1$ in
\eqref{eq:short-segment-scaled-gradient}; the same reasoning gives
$\delta_\beta=O(\beta)$ and the sharper asserted rate.
\end{proof}

\color{black}

\subsection{Proof of \Cref{cor:small-beta-trajectories}: Small-$\beta$ convergence of the closed-loop dynamics}

For each $\beta$, the Gibbs rates are globally Lipschitz and bounded
on $\D$ by \Cref{thm:gibbs}. If $\eta_i=0$, then
\[
b_i^\beta(\eta)
=
\sum_{j\ne i}\eta_j
a_{ji}^{*,\beta}(\eta)
\ge0,
\]
while $\sum_i b_i^\beta(\eta)=0$. Thus the closed-loop vector field is
globally Lipschitz on $\D$, tangent to its affine hull, and inward
pointing at its boundary. It therefore generates a unique global
trajectory in $\D$ from every $\eta_0\in\D$.
The same is true for the reference dynamics. Set
\[
\varepsilon_\beta:=\max_{i\ne j}\sup_{\eta\in\D}
|a_{ij}^{*,\beta}(\eta)-q_{ij}(\eta)|,
\qquad
b^q(\eta):=\sum_{i\ne j}\eta_iq_{ij}(\eta)(e_j-e_i),
\]
and define $b^\beta$ by replacing $q$ with $a^{*,\beta}$.
The theorem gives $\varepsilon_\beta\to0$ and immediately proves
the first assertion. Also,
\[
\sup_{\eta\in\D}\|b^\beta(\eta)-b^q(\eta)\|_1
\le2(d-1)\varepsilon_\beta.
\]
Let $L_q$ be a Lipschitz constant of $b^q$ in the $\ell^1$ norm.
Gronwall's inequality yields, uniformly in $\eta_0\in\D$,
\[
\sup_{0\le t\le T}
\|\eta^\beta(t;\eta_0)-\eta^q(t;\eta_0)\|_1
\le2(d-1)\varepsilon_\beta T e^{L_qT},
\]
which proves \eqref{eq:small-beta-finite-horizon-trajectories}.
For \eqref{eq:small-beta-discounted-trajectories}, fix $T<\infty$
and split the integral at $T$. Since the $\ell^1$ diameter of
$\D$ is two,
\[
\begin{aligned}
&\sup_{\eta_0\in\D}\int_0^\infty e^{-\rho t}
\|\eta^\beta(t;\eta_0)-\eta^q(t;\eta_0)\|_1\,dt\\
&\qquad\le
\frac{1-e^{-\rho T}}{\rho}
\sup_{\eta_0\in\D}\sup_{0\le t\le T}
\|\eta^\beta(t;\eta_0)-\eta^q(t;\eta_0)\|_1
+\frac{2e^{-\rho T}}{\rho}.
\end{aligned}
\]
For each fixed $T$, the first term tends to zero as $\beta\downarrow0$
by the finite-horizon convergence. Taking first the limit superior
as $\beta\downarrow0$ and then letting $T\to\infty$ proves
\eqref{eq:small-beta-discounted-trajectories}.

For the rate assertion, let $L$ be a common Lipschitz constant
of the $q_{ij}$. Adding and subtracting
$q_{ij}(\eta^\beta(t;\eta_0))$ gives
\[
\begin{aligned}
&|a_{ij}^{*,\beta}(\eta^\beta(t;\eta_0))
-q_{ij}(\eta^q(t;\eta_0))|\\
&\qquad\le
|a_{ij}^{*,\beta}(\eta^\beta(t;\eta_0))
-q_{ij}(\eta^\beta(t;\eta_0))|
+|q_{ij}(\eta^\beta(t;\eta_0))
-q_{ij}(\eta^q(t;\eta_0))|\\
&\qquad\le\varepsilon_\beta
+L\|\eta^\beta(t;\eta_0)-\eta^q(t;\eta_0)\|_1.
\end{aligned}
\]
Multiplying by $e^{-\rho t}$, integrating, and summing over
$i\ne j$, we obtain
\[
\begin{aligned}
&\sup_{\eta_0\in\D}\int_0^\infty e^{-\rho t}
\sum_{i\ne j}
|a_{ij}^{*,\beta}(\eta^\beta(t;\eta_0))
-q_{ij}(\eta^q(t;\eta_0))|\,dt\\
&\qquad\le
\frac{d(d-1)\varepsilon_\beta}{\rho}
+d(d-1)L\sup_{\eta_0\in\D}\int_0^\infty e^{-\rho t}
\|\eta^\beta(t;\eta_0)-\eta^q(t;\eta_0)\|_1\,dt.
\end{aligned}
\]
Both terms tend to zero as $\beta\downarrow0$, by
$\varepsilon_\beta\to0$ and
\eqref{eq:small-beta-discounted-trajectories}, respectively.
This proves \eqref{eq:small-beta-discounted-rates}.
\hfill\qed}

\section{Global exponential convergence}\label{app:convergence}

\subsection{Free-energy dissipation and global exponential stability}

The entropy-production estimate has two complementary regimes. In a fixed interior region, every edge conductance $q_{ij}(\eta)\sqrt{\eta_i\eta_j}$ is bounded below, so the dissipation controls all chemical-potential differences and hence the tangential gradient of the free energy. Near the boundary these conductances may degenerate through the population masses, but the logarithmic part of the chemical potential diverges. Its exponential appearance in the Gibbs flux compensates for the small mass and produces a uniformly positive contribution from at least one edge. The next two lemmas isolate the strong-convexity and complete-graph identities used in the interior regime; the boundary mechanism is treated directly in the proof of \Cref{thm:convergence}.

\begin{lemma}[Strong convexity and projected PL inequality]\label{lem:pl}
Under \Cref{ass:regularity}, suppose $2\beta K<1$ and write $\mu=1-2\beta K$. Then $\mathcal F$ in \eqref{eq:free-energy} is $\mu$-strongly convex on the affine simplex, has a unique minimizer $\eta^\infty\in\Dint$, and for every $\eta\in\Dint$,
\begin{equation}\label{eq:projected-pl}
\|\Pi_{\T}\psi(\eta)\|_2^2
\ge
2\mu\bigl(\mathcal F(\eta)-\mathcal F(\eta^\infty)\bigr).
\end{equation}
Moreover,
\begin{equation}\label{eq:strong-convexity-gap}
\mathcal F(\eta)-\mathcal F(\eta^\infty)
\ge
\frac\mu2\|\eta-\eta^\infty\|_2^2.
\end{equation}
\end{lemma}

\begin{proof}
On $\Dint$, the entropy $S(\eta)=\sum_i\eta_i\log\eta_i$ has Hessian $\operatorname{diag}(1/\eta_i)$. Since $\eta_i\le1$, its quadratic form on $\T$ dominates the Euclidean norm squared. Thus $S$ is $1$-strongly convex on the affine simplex. Since $V$ is $K$-semiconvex by \Cref{prop:semiconvexity}, the function $2\beta V$ is $2\beta K$-semiconvex, so $\mathcal F$ is $\mu$-strongly convex. Continuity on the compact simplex gives a minimizer and strong convexity gives uniqueness.

The minimizer cannot lie on the boundary. Suppose $\bar\eta\in\partial\D$, let $Z:=\{i:\bar\eta_i=0\}\ne\varnothing$, put $u=(1/d,\ldots,1/d)$, and set $\eta_t=(1-t)\bar\eta+tu$. Convexity of $V+\Phi_K$ gives an $O(t)$ upper bound for $V(\eta_t)-V(\bar\eta)$ as $t\downarrow0$, because the quadratic correction also changes by $O(t)$. For the entropy,
\[
S(\eta_t)-S(\bar\eta)
=\frac{|Z|t}{d}\log\frac{t}{d}+O(t).
\]
The negative $t\log t$ contribution dominates the $O(t)$ terms, so $\mathcal F(\eta_t)<\mathcal F(\bar\eta)$ for all sufficiently small $t>0$, contradicting minimality.

For $h\in\T$, $D\mathcal F(\eta)[h]=\psi(\eta)\cdot h=\Pi_{\T}\psi(\eta)\cdot h$ by \eqref{eq:tangent-projection}. Strong convexity at $\eta$ with comparison point $\eta^\infty$ gives
\[
\mathcal F(\eta)-\mathcal F(\eta^\infty)
\le
\Pi_{\T}\psi(\eta)\cdot(\eta-\eta^\infty)
-\frac\mu2\|\eta-\eta^\infty\|_2^2.
\]
Cauchy--Schwarz and $ar-\mu r^2/2\le a^2/(2\mu)$ prove \eqref{eq:projected-pl}. Applying strong convexity at the minimizer, whose tangential derivative vanishes, gives \eqref{eq:strong-convexity-gap}.
\end{proof}

\begin{lemma}[Complete-graph identity]
For every $z\in\R^d$,
\begin{equation}\label{eq:complete-graph}
\sum_{i\ne j}(z_j-z_i)^2
=
2d\|\Pi_{\T}z\|_2^2.
\end{equation}
\end{lemma}

\begin{proof}
Expanding the left-hand side gives $2d\sum_i z_i^2-2(\sum_i z_i)^2$, while \eqref{eq:tangent-projection} gives $\|\Pi_{\T}z\|_2^2=\sum_i z_i^2-d^{-1}(\sum_i z_i)^2$. Multiplication by $2d$ proves \eqref{eq:complete-graph}.
\end{proof}

\begin{proof}[Proof of \Cref{thm:convergence}]
Write $w_{ij}:=\eta_i a^*_{ij}$ and $x_{ij}:=\psi_j-\psi_i$. By \eqref{eq:gibbs-flux},
$w_{ij}=q_{ij}(\eta)\sqrt{\eta_i\eta_j}\,e^{-x_{ij}/2}$.
Since $\sum_i\dot\eta_i=0$, the constant $1$ in the derivative of $\eta_i\log\eta_i$ does not contribute, and therefore
\[
\frac{\dd}{\dd t}\mathcal F(\eta(t))
=
\frac12\sum_{i\ne j}x_{ij}(w_{ij}-w_{ji}).
\]
By \Cref{ass:symmetry}, $q_{ij}(\eta)=q_{ji}(\eta)$, so $w_{ji}=q_{ij}(\eta)\sqrt{\eta_i\eta_j}\,e^{x_{ij}/2}$. Substitution gives \eqref{eq:dissipation-identity}; nonnegativity follows from the elementary inequality stated after \eqref{eq:dissipation}.

We next identify the equilibria. Since $\eta^\infty\in\Dint$ minimizes $\mathcal F$, its tangential first-order condition is $\Pi_{\T}\psi(\eta^\infty)=0$. Hence the two fluxes on every unoriented edge agree, so $\eta^\infty$ is an equilibrium. Conversely, if $\bar\eta\in\Dint$ is an equilibrium, its constant trajectory has zero free-energy derivative. By \eqref{eq:dissipation-identity}, $\mathcal D(\bar\eta)=0$. Every edge weight is strictly positive, so all chemical potentials are equal and $\Pi_{\T}\psi(\bar\eta)=0$. Strong convexity then implies $\bar\eta=\eta^\infty$.

We prove a global entropy-production inequality in two regions. First suppose $\eta_i\ge\delta$ for every $i$. Since $x(e^{x/2}-e^{-x/2})\ge x^2$, \eqref{eq:dissipation}, the bound $q_{ij}(\eta)\ge q_{\min}$, and \eqref{eq:complete-graph} give
\[
\mathcal D(\eta)
\ge
dq_{\min}\delta\|\Pi_{\T}\psi(\eta)\|_2^2.
\]
By \eqref{eq:projected-pl}, this is at least $2\mu d q_{\min}\delta\bigl(\mathcal F(\eta)-\mathcal F(\eta^\infty)\bigr)$.

It remains to control a boundary neighborhood. Put $M_V=\osc_{\D}V$ and
\[
\delta_{\partial}
:=
\frac1d\exp\bigl(-1-2\beta(dM_V+K)\bigr).
\]
Suppose $\varepsilon:=\min_i\eta_i\le\delta_{\partial}$, choose $i$ with $\eta_i=\varepsilon$, and choose $j$ with $\eta_j\ge1/d$. Applying global $K$-semiconvexity to the feasible point $\eta+\eta_j(e_i-e_j)$ gives
\[
\partial_iV(\eta)-\partial_jV(\eta)
\le
dM_V+K.
\]
Hence, with $x=\psi_j(\eta)-\psi_i(\eta)$,
\[
x
\ge
\log\frac1{d\varepsilon}-2\beta(dM_V+K)
\ge1.
\]
Keeping only the nonnegative $(i,j)$ term in \eqref{eq:dissipation} and using $x(e^{x/2}-e^{-x/2})\ge(1-e^{-1})e^{x/2}$ gives
\[
\mathcal D(\eta)
\ge
\frac{q_{\min}(1-e^{-1})}{2d}
e^{-\beta(dM_V+K)}
=:c_{\partial}>0.
\]
Since $-\log d\le\sum_i\eta_i\log\eta_i\le0$, the free-energy gap is bounded above by $\log d+2\beta M_V$. Thus
\[
\mathcal D(\eta)
\ge
\lambda_{\partial}\bigl(\mathcal F(\eta)-\mathcal F(\eta^\infty)\bigr),
\qquad
\lambda_{\partial}:=\frac{c_{\partial}}{\log d+2\beta M_V}>0.
\]
{\color{black}Equivalently, the two admissible decay constants are
\begin{equation}\label{eq:stability-rate-constants}
\begin{aligned}
\lambda_{\mathrm{int}}
 &=2(1-2\beta K)q_{\min}e^{-1-2\beta(dM_V+K)},\\
\lambda_{\partial}
 &=\frac{q_{\min}(1-e^{-1})e^{-\beta(dM_V+K)}}
 {2d(\log d+2\beta M_V)}.
\end{aligned}
\end{equation}
Taking $\lambda=\min\{\lambda_{\mathrm{int}},\lambda_{\partial}\}$
gives the entropy-production inequality throughout $\Dint$.} Combining it with \eqref{eq:dissipation-identity} and Gronwall's inequality gives \eqref{eq:free-energy-decay}. Finally, \eqref{eq:strong-convexity-gap} and \eqref{eq:free-energy-decay} give \eqref{eq:state-decay}. This proves the theorem.
\end{proof}

\subsection{Proof of \Cref{thm:uniform-small-beta-stability}: Uniform exponential convergence for small inverse temperature}

\begin{proof}
Use $M$ from \eqref{eq:temperature-value-bounds}. Choose
$\beta_*\le\beta_0$ sufficiently small that $d\beta M\le1/2$
for $0<\beta\le\beta_*$. The assumed margin gives
$\mu_\beta:=1-2\beta K_\beta\ge\varepsilon$ and
$\theta_\beta:=\beta(dM+K_\beta)\le1$. In particular,
\Cref{thm:convergence} applies at each such $\beta$ and supplies
the unique equilibrium, strong convexity, and the dissipation identity.

It remains only to make its constants uniform. In
\eqref{eq:stability-rate-constants}, replace $M_V$ by its uniform
upper bound $M$. The interior constant is at least
$2\varepsilon q_{\min}e^{-3}$, while the boundary constant is at
least $q_{\min}(1-e^{-1})e^{-1}/[2d(\log d+1)]$.
Their minimum is the stated $\lambda_*$, which proves
\eqref{eq:uniform-small-beta-entropy-production} without repeating
the two-region argument.

Write $G_\beta(\eta)=\mathcal F_\beta(\eta)
-\mathcal F_\beta(\eta_\beta^\infty)$. The dissipation identity
gives $\frac{\dd}{\dd t}G_\beta\le-\lambda_*G_\beta$ along the
trajectory, so Gronwall proves \eqref{eq:uniform-small-beta-energy-decay}.
Finally, the entropy range and \eqref{eq:strong-convexity-gap} give
\[
\frac{\varepsilon}{2}\|\eta-\eta_\beta^\infty\|_2^2
\le G_\beta(\eta)\le\log d+2\beta M\le\log d+1.
\]
Combining these bounds with the energy decay and taking square roots
proves \eqref{eq:uniform-small-beta-state-decay} with the stated $B_*$.
\end{proof}

\subsection{Proof of \Cref{cor:beta-stability-limits}: Extreme-temperature limits of the stable equilibria}

\begin{proof}
Write $S(\eta)=\sum_i\eta_i\log\eta_i$. For the small-$\beta$
conclusion, compare the free energy at its minimizer with its value
at $u$. Since $S(u)=-\log d$, this gives
\[
\KL(\eta_\beta^\infty\Vert u)
=S(\eta_\beta^\infty)-S(u)
\le2\beta\bigl(V_\beta(u)-V_\beta(\eta_\beta^\infty)\bigr).
\]
The oscillation bound \eqref{eq:temperature-value-bounds} and
Pinsker's inequality prove the two asserted equilibrium estimates.
The common exponential rate follows directly from
\Cref{thm:uniform-small-beta-stability}.

For the reference dynamics, put $c_0=q_{\min}+(d-1)q_{\max}$.
The inequality $\dot\eta_i^q\ge q_{\min}-c_0\eta_i^q$ implies
strict positivity of every coordinate for $t>0$, even from a
boundary initial state. Symmetry and pairing of the directed edges give
\[
-\frac{\dd}{\dd t}\KL(\eta^q(t)\Vert u)
=\sum_{i<j}q_{ij}(\eta^q(t))
(\eta_i^q(t)-\eta_j^q(t))
(\log\eta_i^q(t)-\log\eta_j^q(t))
\ge dq_{\min}\KL(\eta^q(t)\Vert u).
\]
For the last inequality, replace each rate by $q_{\min}$, expand
the pair sum as $dS(\eta^q)-\sum_i\log\eta_i^q$, and use
$\sum_i\log\eta_i^q\le-d\log d$.
Apply Gronwall on $[\varepsilon,t]$ and then let
$\varepsilon\downarrow0$, using continuity of relative entropy on
the closed simplex. Symmetry makes $u$ an equilibrium, strictly
positive inward drift excludes boundary equilibria, and zero
entropy dissipation forces any interior equilibrium to equal $u$.

For the large-$\beta$ conclusion, choose $\bar\eta\in\mathcal M$.
The zero-rate control keeps the population there, giving
$V_\beta(\bar\eta)\le C_{\min}/\rho+Q/(\beta\rho)$.
Comparison of the free energy at $\eta_\beta^\infty$ and $\bar\eta$
then yields
\begin{equation}\label{eq:equilibrium-value-excess}
2\beta\left(V_\beta(\eta_\beta^\infty)-\frac{C_{\min}}\rho\right)
\le S(\bar\eta)-S(\eta_\beta^\infty)+\frac{2Q}\rho
\le\log d+\frac{2Q}\rho.
\end{equation}
At the equilibrium, stationarity of the free energy gives
$\beta(\partial_jV_\beta-\partial_iV_\beta)
=\frac12\log(\eta_{\beta,i}^\infty/\eta_{\beta,j}^\infty)$.
Insert this into the Gibbs formula and pair the symmetric terms in
the HJB equation to obtain
\[
\rho V_\beta(\eta_\beta^\infty)
=C(\eta_\beta^\infty)
+\frac1\beta\sum_{i<j}q_{ij}(\eta_\beta^\infty)
\left(\sqrt{\eta_{\beta,i}^\infty}
-\sqrt{\eta_{\beta,j}^\infty}\right)^2.
\]
The last term is nonnegative. Thus
$C(\eta_\beta^\infty)-C_{\min}$ is bounded by
$\rho V_\beta(\eta_\beta^\infty)-C_{\min}$, and
\eqref{eq:equilibrium-value-excess} proves the stated cost bound.
Continuity of $C$ and compactness of $\D$ imply convergence in
distance to $\mathcal M$: otherwise a subsequence staying a positive
distance away would have a limit that both minimizes $C$ and lies
outside $\mathcal M$. This also gives the unique-minimizer conclusion.
\end{proof}

\section{Proof of \Cref{thm:planner-to-MFG}: The potential-MFG correspondence}\label{app:potential-mfg}

Let $p(t):=D_{\T}V(\eta^*(t))$, represented as a column vector, and
retain the time-dependent rates $\alpha^*(t)$ from the theorem.
By \Cref{thm:gibbs}, the intrinsic gradient is uniformly bounded on
$\D$ and locally Lipschitz in the interior. Since $\eta^*$ is locally
absolutely continuous and stays in a compact subset of $\Dint$ on
each finite time interval, $p$ is bounded and locally absolutely
continuous. The rates $\alpha^*$ are uniformly bounded by
\eqref{eq:global-rate-bounds}.

{\color{black}
These bounds provide the regularity needed for a first-order DPP
sensitivity argument along the optimal planner trajectory $\eta^*$.
Since the available regularity of $V$ is $\mathcal{C}^{1,1}$, we derive
the adjoint equation by differentiating the perturbed DPP once with
respect to the initial state, while holding the open-loop rates
$\alpha^*(s)=a^*(\eta^*(s))$ fixed. We then differentiate the resulting
integral representation of $p$ in time. This procedure uses only first
derivatives of $V$ and avoids evaluating its Hessian along $\eta^*$.
}

{\color{black}
The proof has three steps: derive the adjoint equation by sensitivity,
add a scalar correction to recover the individual Bellman equation,
and verify optimality and consistency with the population dynamics.
}

\medskip
\noindent\textcolor{black}{\textbf{Step 1: Sensitivity and the adjoint equation.}}
Fix $0\le t<T$ and let $P(t,s)$ be the stochastic evolution family
generated by $\alpha^*$, as in \eqref{eq:evolution-family}.
In addition to $P(t,s)\one=\one$, it satisfies the backward equation
\begin{equation}\label{eq:planner-transition-backward}
\partial_tP(t,s)=-\alpha^*(t)P(t,s).
\end{equation}
For a tangent row vector $h\in\T$ and sufficiently small
$|\varepsilon|$, apply these same open-loop rates from
$\eta^*(t)+\varepsilon h$. They remain admissible by
\Cref{def:admissible}, and linearity of the state equation gives
\begin{equation}\label{eq:perturbed-path-MFG}
\eta^\varepsilon(s)
=(\eta^*(t)+\varepsilon h)P(t,s)
=\eta^*(s)+\varepsilon hP(t,s),\qquad t\le s\le T.
\end{equation}
The DPP \eqref{eq:dpp} implies
\begin{equation}\label{eq:perturbed-DPP-MFG-proof}
V(\eta^*(t)+\varepsilon h)
\le
\int_t^T e^{-\rho(s-t)}f(\eta^\varepsilon(s),\alpha^*(s))\,ds
+e^{-\rho(T-t)}V(\eta^\varepsilon(T)),
\end{equation}
with equality at $\varepsilon=0$ by optimality of the planner's
continuation. Both signs of $\varepsilon$ are allowed. Hence the
right-hand side minus the left-hand side has a differentiable local
minimum at zero, and its derivative there vanishes. Differentiation
under the integral is justified by the bounded rates, the constant
reference generator, and the bounded first derivatives of $C$.

At fixed rates, the running-cost derivative in a tangent row direction
$z$ is $D_\eta f(\eta^*(s),\alpha^*(s))[z]=zg(s)$, where
\begin{equation*}
g_i(s):=c(i,\eta^*(s))
+\frac1\beta\KL(\alpha_i^*(s)\Vert\bar q_i).
\end{equation*}
Here we used \eqref{eq:constant-reference} and
\eqref{eq:potential-condition}. Differentiating
\eqref{eq:perturbed-DPP-MFG-proof} at zero and using
\eqref{eq:perturbed-path-MFG} therefore yields the sensitivity identity
\begin{equation}\label{eq:DPP-sensitivity-MFG}
\begin{gathered}
hp(t)=hY_T(t)\quad(h\in\T),\\
Y_T(t):=\int_t^T e^{-\rho(s-t)}P(t,s)g(s)\,ds
+e^{-\rho(T-t)}P(t,T)p(T).
\end{gathered}
\end{equation}
Since $p(t)\in\T$, this identity gives $p(t)=\Pi_{\T}Y_T(t)$.

Differentiate the expression for $Y_T$ in \mbox{\eqref{eq:DPP-sensitivity-MFG}} with respect to $t$, keeping $T$ fixed. Leibniz's rule and
\eqref{eq:planner-transition-backward} give
$\dot Y_T=\rho Y_T-\alpha^*Y_T-g$.
Moreover, $Y_T-p$ is a multiple of $\one$ and
$\alpha^*\one=0$. Projection onto $\T$ consequently gives
$\dot p=\rho p-\Pi_{\T}(\alpha^*p+g)$ almost everywhere.
By \eqref{eq:gibbs-feedback}, the row $\alpha_i^*(t)$ minimizes
\eqref{eq:MFG-row-Hamiltonian} at $z=p(t)$, so
$(\alpha^*p)_i+g_i=c(i,\eta^*)+H_i^{\mathrm{MFG}}(p)$.
Expanding the projection \eqref{eq:tangent-projection} yields
\begin{equation}\label{eq:p-equation-with-kappa}
\dot p_i(t)=\rho p_i(t)-c(i,\eta^*(t))
-H_i^{\mathrm{MFG}}(p(t))+\kappa(t),\qquad i\in[d],
\end{equation}
where
\begin{equation}\label{eq:kappa-formula}
\kappa(t):=\frac1d\sum_{i=1}^d
\bigl[c(i,\eta^*(t))+H_i^{\mathrm{MFG}}(p(t))\bigr].
\end{equation}
The same scalar occurs in every component because the orthogonal
complement of $\T$ is $\operatorname{span}\{\one\}$; equivalently,
summing \eqref{eq:p-equation-with-kappa} and using
$\sum_i p_i=0$ gives \eqref{eq:kappa-formula}.
Since $p$ and $c$ are bounded, the explicit row Hamiltonian
\eqref{eq:MFG-row-Hamiltonian-explicit} shows that $\kappa$ is bounded.

\medskip
\noindent\textcolor{black}{\textbf{Step 2: The scalar correction.}}
Define
\begin{equation}\label{eq:MFG-value-from-planner-gradient}
\gamma(t):=\int_t^\infty e^{-\rho(s-t)}\kappa(s)\,ds,
\qquad v^*(t):=p(t)+\gamma(t)\one.
\end{equation}
Boundedness of $\kappa$ gives
$|\gamma(t)|\le\|\kappa\|_\infty/\rho$ and
$\dot\gamma=\rho\gamma-\kappa$ almost everywhere. Thus $v^*$ is bounded
and locally absolutely continuous. The common shift preserves
component differences, proving \eqref{eq:value-gradient-differences},
and leaves the row Hamiltonians unchanged. Substituting
\eqref{eq:MFG-value-from-planner-gradient} into
\eqref{eq:p-equation-with-kappa} therefore gives precisely
\eqref{eq:MFG-Bellman} with $m=\eta^*$ and $v^m=v^*$.
In particular, \eqref{eq:MFG-row-minimizer} and
\eqref{eq:value-gradient-differences} identify the individual Gibbs
selector $\widehat\alpha(v^*(t))$ with the already defined
$\alpha^*(t)=a^*(\eta^*(t))$.

\medskip
\noindent\textcolor{black}{\textbf{Step 3: Verification and consistency.}}
Let $\alpha$ be any admissible individual control. Write its running
cost as $L_i^\alpha(s):=c(i,\eta^*(s))+
\beta^{-1}\KL(\alpha_i(s)\Vert\bar q_i)$.
Hamiltonian minimization \eqref{eq:MFG-row-Hamiltonian}, combined with
the Bellman equation just established, gives
\begin{equation}\label{eq:MFG-verification-inequality-pointwise}
\dot v_i^*(s)-\rho v_i^*(s)
+(\alpha(s)v^*(s))_i+L_i^\alpha(s)\ge0
\end{equation}
for almost every $s$, with equality when $\alpha=\alpha^*$.
For a chain $X$ with generator $\alpha$ and $X_t=i$, discounted
Dynkin's formula and \eqref{eq:MFG-verification-inequality-pointwise}
yield
\begin{equation}\label{eq:Dynkin-MFG-proof}
\begin{aligned}
&\mathbb E_{t,i}^{\alpha}
 [e^{-\rho(T-t)}v_{X_T}^*(T)]-v_i^*(t)\\
&\quad=\mathbb E_{t,i}^{\alpha}\int_t^T e^{-\rho(s-t)}
 [\dot v_{X_s}^*(s)-\rho v_{X_s}^*(s)
 +(\alpha(s)v^*(s))_{X_s}]\,ds\\
&\quad\ge-\mathbb E_{t,i}^{\alpha}\int_t^T
 e^{-\rho(s-t)}L_{X_s}^\alpha(s)\,ds.
\end{aligned}
\end{equation}
This identity is valid for locally integrable rates: the state
probabilities $\mu(s)$ are absolutely continuous and solve
$\dot\mu=\mu\alpha$, so the equality follows by applying the product
rule to $e^{-\rho(s-t)}\mu(s)v^*(s)$ and integrating.
All terms in that product rule are integrable on finite intervals.
Controls of infinite cost satisfy the required lower bound trivially;
for finite-cost controls, the running-cost integral is also finite on
each such interval.

The terminal expectation in \eqref{eq:Dynkin-MFG-proof} tends to zero
as $T\to\infty$, because $v^*$ is bounded. The running cost has
bounded negative part, so passage to the infinite horizon gives
\begin{equation}\label{eq:MFG-verified-value}
\begin{gathered}
v_i^*(t)\le J_i^{\mathrm{MFG}}(t;\eta^*,\alpha)
\quad\text{for every admissible }\alpha,\\
v_i^*(t)=J_i^{\mathrm{MFG}}(t;\eta^*,\alpha^*).
\end{gathered}
\end{equation}
The equality follows by using the bounded admissible Gibbs control,
for which the inequality in \eqref{eq:Dynkin-MFG-proof} is an equality
and the running cost is bounded. Thus \eqref{eq:MFG-verified-value}
identifies $v^*$ as the representative-agent value.
Any other bounded locally absolutely continuous solution of
\eqref{eq:MFG-Bellman} has bounded Gibbs rates, and the same
verification identifies it with this value. This proves uniqueness
of the bounded Bellman solution.

Strict convexity in \eqref{eq:MFG-row-Hamiltonian} gives a unique
minimizing row at every time and state. To see uniqueness among
optimal time-dependent controls, apply the verification with the
strictly positive initial law $\eta_0$. The resulting nonnegative
Hamiltonian gaps have zero integral for an optimal control. Every
state probability remains positive at finite times by
\eqref{eq:interior-preservation}, so each gap vanishes almost
everywhere and the rates agree with the Gibbs selector almost
everywhere.
Finally, the planner state equation \eqref{eq:state-equation}, with
rates $a^*(\eta^*(t))=\alpha^*(t)$, is exactly
\eqref{eq:MFG-consistency} with $m=\eta^*$ and the prescribed initial
law. Hence $(\eta^*,v^*,\alpha^*)$ is a discounted finite-state MFG
equilibrium.
\hfill\qed

\bibliographystyle{abbrv}
\bibliography{references}

\begin{thebibliography}{10}

\bibitem{AtarCohen2016}
R.~Atar and A.~Cohen.
\newblock A differential game for a multiclass queueing model in the
  moderate-deviation heavy-traffic regime.
\newblock {\em Math. Oper. Res.}, 41(4):1354--1380, 2016.

\bibitem{AtarCohen2017}
R.~Atar and A.~Cohen.
\newblock Asymptotically optimal control for a multiclass queueing model in the
  moderate deviation heavy traffic regime.
\newblock {\em Ann. Appl. Probab.}, 27(5):2862--2906, 2017.

\bibitem{mmAS2008}
M.~Avellaneda and S.~Stoikov.
\newblock High-frequency trading in a limit order book.
\newblock {\em Quantitative Finance}, 8(3):217--224, 2008.

\bibitem{mmBBP2023}
B.~Baldacci, P.~Bergault, and D.~Possama\"i.
\newblock A mean-field game of market-making against strategic traders.
\newblock {\em SIAM Journal on Financial Mathematics}, 2023.

\bibitem{mmBBG2023}
A.~Barzykin, P.~Bergault, and O.~Gu\'eant.
\newblock Market making by an {FX} dealer: tiers, pricing ladders and hedging
  rates for optimal risk control.
\newblock \url{https://arxiv.org/abs/2112.02269v3}, 2023.
\newblock arXiv:2112.02269, version 3.

\bibitem{BayraktarCohen2018}
E.~Bayraktar and A.~Cohen.
\newblock Analysis of a finite state many player game using its master
  equation.
\newblock {\em SIAM Journal on Control and Optimization}, 56(5):3538--3568,
  2018.

\bibitem{BayraktarHernandezYanZhu2026PhiBE}
E.~Bayraktar, M.~Hernandez, Q.~Yan, and Y.~Zhu.
\newblock Mean-field {PhiBE}: Continuous-time mean-field reinforcement learning
  from discrete-time data, 2026.

\bibitem{BayraktarHernandezYanZhu2026PolicyGradient}
E.~Bayraktar, M.~Hernandez, Q.~Yan, and Y.~Zhu.
\newblock Policy gradient for continuous-time mean-field control, 2026.

\bibitem{BensoussanFrehseYam2013}
A.~Bensoussan, J.~Frehse, and P.~Yam.
\newblock {\em Mean Field Games and Mean Field Type Control Theory}.
\newblock SpringerBriefs in Mathematics. Springer, New York, 2013.

\bibitem{mmBG2021}
P.~Bergault and O.~Gu\'eant.
\newblock Size matters for {OTC} market makers: General results and
  dimensionality reduction techniques.
\newblock {\em Mathematical Finance}, 31(1):279--322, 2021.

\bibitem{BudhirajaDupuisFischerRamanan2015}
A.~Budhiraja, P.~Dupuis, M.~Fischer, and K.~Ramanan.
\newblock Local stability of kolmogorov forward equations for finite state
  nonlinear markov processes.
\newblock {\em Electronic Journal of Probability}, 20(81):1--30, 2015.

\bibitem{CannarsaSinestrari2004}
P.~Cannarsa and C.~Sinestrari.
\newblock {\em Semiconcave Functions, Hamilton--Jacobi Equations, and Optimal
  Control}, volume~58 of {\em Progress in Nonlinear Differential Equations and
  Their Applications}.
\newblock Birkh{\"a}user, Boston, 2004.

\bibitem{CarmonaDelarue2018I}
R.~Carmona and F.~Delarue.
\newblock {\em Probabilistic Theory of Mean Field Games with Applications I:
  Mean Field FBSDEs, Control, and Games}, volume~83 of {\em Probability Theory
  and Stochastic Modelling}.
\newblock Springer, Cham, 2018.

\bibitem{CarmonaDelarue2018II}
R.~Carmona and F.~Delarue.
\newblock {\em Probabilistic Theory of Mean Field Games with Applications II:
  Mean Field Games with Common Noise and Master Equations}, volume~84 of {\em
  Probability Theory and Stochastic Modelling}.
\newblock Springer, Cham, 2018.

\bibitem{CarmonaWang2021}
R.~Carmona and P.~Wang.
\newblock A probabilistic approach to extended finite state mean field games.
\newblock {\em Mathematics of Operations Research}, 46(2):471--502, 2021.

\bibitem{mmCJR2014}
{\'A}.~Cartea, S.~Jaimungal, and J.~Ricci.
\newblock Buy low, sell high: A high frequency trading perspective.
\newblock {\em SIAM Journal on Financial Mathematics}, 5(1):415--444, 2014.

\bibitem{Cecchin2021}
A.~Cecchin.
\newblock Finite state {$N$}-agent and mean field control problems.
\newblock {\em ESAIM: Control, Optimisation and Calculus of Variations}, 27:31,
  2021.

\bibitem{CecchinFischer2020}
A.~Cecchin and M.~Fischer.
\newblock Probabilistic approach to finite state mean field games.
\newblock {\em Applied Mathematics \& Optimization}, 81(2):253--300, 2020.

\bibitem{CecchinPelino2019}
A.~Cecchin and G.~Pelino.
\newblock Convergence, fluctuations and large deviations for finite state mean
  field games via the master equation.
\newblock {\em Stochastic Processes and their Applications},
  129(11):4510--4555, 2019.

\bibitem{ChakrabortyCohenYoung2023}
P.~Chakraborty, A.~Cohen, and V.~R. Young.
\newblock Optimal {D}ividends {U}nder {M}odel {U}ncertainty.
\newblock {\em SIAM J. Financial Math.}, 14(2):497--524, 2023.

\bibitem{Cohen2019QueueingUncertainty}
A.~Cohen.
\newblock Asymptotic analysis of a multiclass queueing control problem under
  heavy traffic with model uncertainty.
\newblock {\em Stoch. Syst.}, 9(4):359--391, 2019.

\bibitem{Cohen2019BrownianUncertainty}
A.~Cohen.
\newblock Brownian control problems for a multiclass {M}/{M}/1 queueing problem
  with model uncertainty.
\newblock {\em Math. Oper. Res.}, 44(2):739--766, 2019.

\bibitem{CohenHeningSun2022}
A.~Cohen, A.~Hening, and C.~Sun.
\newblock Optimal ergodic harvesting under ambiguity.
\newblock {\em SIAM J. Control Optim.}, 60(2):1039--1063, 2022.

\bibitem{CohenHuffman2026}
A.~Cohen and E.~Huffman.
\newblock Uniform-in-time convergence rates to a nonlinear markov chain for
  mean-field interacting jump processes.
\newblock {\em SIAM Journal on Control and Optimization}, 64(2):593--623, 2026.

\bibitem{CohenSun2026}
A.~Cohen and C.~Sun.
\newblock Existence of {O}ptimal {S}tationary {S}ingular {C}ontrols and {M}ean
  {F}ield {G}ame {E}quilibria.
\newblock {\em Math. Oper. Res.}, 51(3):2041--2072, 2026.

\bibitem{DelarueLavigne2026RobustMFC}
F.~Delarue and P.~Lavigne.
\newblock Robust mean field control: Stochastic maximum principle and
  variational mean field games, 2026.

\bibitem{DelarueLavigne2026RobustMFG}
F.~Delarue and P.~Lavigne.
\newblock Robust mean-field games under entropy-based uncertainty, 2026.

\bibitem{mmDL2025}
R.~Donnelly and Z.~Li.
\newblock Dynamic inventory management with mean-field competition.
\newblock {\em Applied Mathematical Finance}, 2025.

\bibitem{ErbarFathiLaschosSchlichting2016}
M.~Erbar, M.~Fathi, V.~Laschos, and A.~Schlichting.
\newblock Gradient flow structure for {McKean--Vlasov} equations on discrete
  spaces.
\newblock {\em Discrete and Continuous Dynamical Systems}, 36(12):6799--6833,
  2016.

\bibitem{ErbarFathiSchlichting2020}
M.~Erbar, M.~Fathi, and A.~Schlichting.
\newblock Entropic curvature and convergence to equilibrium for mean-field
  dynamics on discrete spaces.
\newblock {\em ALEA, Latin American Journal of Probability and Mathematical
  Statistics}, 17:445--471, 2020.

\bibitem{mmFL2013}
P.~Fodra and M.~Labadie.
\newblock High-frequency market-making for multi-dimensional {Markov}
  processes.
\newblock \url{https://arxiv.org/abs/1303.7177}, 2013.

\bibitem{FrikhaGermainLaurierePhamSong2025}
N.~Frikha, M.~Germain, M.~Lauri{\`e}re, H.~Pham, and X.~Song.
\newblock Actor-critic learning for mean-field control in continuous time.
\newblock {\em Journal of Machine Learning Research}, 26(127):1--42, 2025.

\bibitem{FrikhaPhamSong2024}
N.~Frikha, H.~Pham, and X.~Song.
\newblock Full error analysis of policy gradient learning algorithms for
  exploratory linear quadratic mean-field control problem in continuous time
  with common noise, 2024.

\bibitem{GomesMohrSouza2013}
D.~A. Gomes, J.~Mohr, and R.~R. Souza.
\newblock Continuous time finite state mean field games.
\newblock {\em Applied Mathematics \& Optimization}, 68(1):99--143, 2013.

\bibitem{mmG2017}
O.~Gu\'eant.
\newblock Optimal market making.
\newblock {\em Applied Mathematical Finance}, 24(2):112--154, 2017.

\bibitem{mmGLFT2013}
O.~Gu\'eant, C.-A. Lehalle, and J.~Fernandez-Tapia.
\newblock Dealing with the inventory risk: a solution to the market making
  problem.
\newblock {\em Mathematics and Financial Economics}, 7:477--507, 2013.

\bibitem{mmGP2013}
F.~Guilbaud and H.~Pham.
\newblock Optimal high-frequency trading with limit and market orders.
\newblock {\em Quantitative Finance}, 13(1):79--94, 2013.

\bibitem{GuoXuZariphopoulou2022}
X.~Guo, R.~Xu, and T.~Zariphopoulou.
\newblock Entropy regularization for mean field games with learning.
\newblock {\em Mathematics of Operations Research}, 47(4):3239--3260, 2022.

\bibitem{mmHS1981}
T.~Ho and H.~R. Stoll.
\newblock Optimal dealer pricing under transactions and return uncertainty.
\newblock {\em Journal of Financial Economics}, 9(1):47--73, 1981.

\bibitem{HuangCainesMalhame2006}
M.~Huang, P.~E. Caines, and R.~P. Malham{\'e}.
\newblock Large-population stochastic dynamic games: Closed-loop
  {McKean--Vlasov} systems and the nash certainty equivalence principle.
\newblock {\em Communications in Information and Systems}, 6(3):221--252, 2006.

\bibitem{KesidisWalrand1993}
G.~Kesidis and J.~Walrand.
\newblock Relative entropy between markov transition rate matrices.
\newblock {\em IEEE Transactions on Information Theory}, 39(3):1056--1057,
  1993.

\bibitem{Kraaij2016}
R.~Kraaij.
\newblock Large deviations for finite state markov jump processes with
  mean-field interaction via the comparison principle for an associated
  hamilton--jacobi equation.
\newblock {\em Journal of Statistical Physics}, 164(2):321--345, 2016.

\bibitem{LasryLions2007}
J.-M. Lasry and P.-L. Lions.
\newblock Mean field games.
\newblock {\em Japanese Journal of Mathematics}, 2(1):229--260, 2007.

\bibitem{Malrieu2001}
F.~Malrieu.
\newblock Logarithmic sobolev inequalities for some nonlinear {PDE}'s.
\newblock {\em Stochastic Processes and their Applications}, 95(1):109--132,
  2001.

\bibitem{MielkePeletierRenger2014}
A.~Mielke, M.~A. Peletier, and D.~R.~M. Renger.
\newblock On the relation between gradient flows and the large-deviation
  principle, with applications to {Markov} chains and diffusion.
\newblock {\em Potential Analysis}, 41(4):1293--1327, 2014.

\bibitem{PeletierSchlichting2023}
M.~A. Peletier and A.~Schlichting.
\newblock Cosh gradient systems and tilting.
\newblock {\em Nonlinear Analysis}, 231:113094, 2023.

\bibitem{Todorov2006}
E.~Todorov.
\newblock Linearly-solvable markov decision problems.
\newblock In {\em Advances in Neural Information Processing Systems},
  volume~19, pages 1369--1376, 2006.

\bibitem{Todorov2009}
E.~Todorov.
\newblock Efficient computation of optimal actions.
\newblock {\em Proceedings of the National Academy of Sciences},
  106(28):11478--11483, 2009.

\bibitem{WangZariphopoulouZhou2020}
H.~Wang, T.~Zariphopoulou, and X.~Y. Zhou.
\newblock Reinforcement learning in continuous time and space: A stochastic
  control approach.
\newblock {\em Journal of Machine Learning Research}, 21(198):1--34, 2020.

\end{thebibliography}

\end{document}